\documentclass{book}

\usepackage{amsopn,amstext,amsbsy,amsmath,amscd,amsthm,amsfonts,xfrac,xcolor,enumitem,geometry,comment}

\usepackage[titletoc]{appendix}

\usepackage[ma,mnf,master]{mccover} 
\usepackage[latin9]{inputenc} 
\usepackage[all]{xy} 
\usepackage{enumitem}
\usepackage{hyperref}
\usepackage{url}
\usepackage{color}
\usepackage{mathtools}
\theoremstyle{plain}
\newtheorem{theorem}                 {\bf Theorem}      [chapter]
\newtheorem{proposition}  [theorem]  {\bf Proposition}
\newtheorem{corollary}    [theorem]  {\bf Corollary}
\newtheorem{lemma}        [theorem]  {\bf Lemma}

\theoremstyle{definition}
\newtheorem{example}      [theorem]  {\bf Example}
\newtheorem{definition}   [theorem]  {\bf Definition}
\newtheorem{remark}       [theorem]  {\bf Remark}
\numberwithin{equation}{chapter}

\allowdisplaybreaks

\authors{Adam Lindstr\"om}
\maintitle{Harmonic Morphisms and p-Harmonic Functions on the  Classical Compact Symmetric Spaces via the Cartan Embedding}
\issnum{2023}{40}
\sernum{LUNFMA}{3136}{2023}
\begin{document}

\def\restr#1#2{{
  \left.\kern-\nulldelimiterspace 
  #1 
  \vphantom{\big|} 
  \right|_{#2} 
  }}

\def\nab#1#2#3{\nabla^{\hbox{$\scriptstyle{#1}$}}_{\hbox{$\scriptstyle{#2}$}}{\hbox{$#3$}}}

\def\onab#1#2#3{\overline{\nabla}^{\hbox{$\scriptstyle{#1}$}}_{\hbox{$\scriptstyle{#2}$}}{\hbox{$#3$}}}

\def\pnab#1#2{\nab{\perp}{#1}{#2}}

\def\tnab#1#2{\widetilde{\nabla}_{\hbox{$\scriptstyle{#1}$}}{\hbox{$#2$}}}

\def\shape#1#2{S_{\hbox{$\scriptstyle{#1}$}}{\hbox{$#2$}}}

\def\ip#1#2{\left< #1, #2 \right>}

\def\wt#1{\widetilde{#1}}

\def\ind#1{\mathcal{#1}}

\def\snab#1#2#3{\hbox{$\nabla$\kern-.1em\raise 1.2 ex\hbox{$\scriptstyle{#1}$}\kern-.5em\lower 0.8 ex\hbox{$#2$}\kern-.0em{$#3$}}}

\def\nnab#1{\nabla^{\hbox{$\scriptstyle{#1}$}}}
\def\nsnab#1{\hbox{$\nabla$\kern-.1em\raise 1.0 ex\hbox{$\scriptstyle{#1}$}}}

\newcommand{\al}[1]{{\textcolor{blue}{#1}}}
\newcommand{\sed}{\qquad\qquad}

\def\tnabla{\widetilde{\nabla}}
\def\rn{\mathbb R}
\def\cn{\mathbb C}
\def\hn{\mathbb H}
\def\cp{\cn \text{\rm P}}
\def\s{\mathbb{S}}
\def\ci{\mathcal I}
\def\tr{\textrm{\upshape{trace}}}
\def\pdt{\frac{\partial}{\partial t}}
\def\smooth#1{C^\infty({#1})}
\def\dtatzero{\restr{\frac{d}{dt}}{t=0}}
\def\Ad#1{\mathrm{Ad}_{#1}}
\def\ad{\mathrm{ad}}
\def\re{\mathrm{Re}}
\def\basis{\mathcal{B}}
\def\grad{\textrm{grad }}
\def\div{\textrm{div}}

\def\hor{\mathcal{H}}
\def\vert{\mathcal{V}}
\def \U#1{\mathrm{\bf U}(#1)}
\def \O#1{\mathrm{\bf O}(#1)}
\def \SU#1{\mathrm{\bf SU}(#1)}
\def \SO#1{\mathrm{\bf SO}(#1)}
\def \Sp#1{\mathrm{\bf Sp}(#1)}
\def \GL#1{\mathrm{\bf GL}(#1)}
\def \SL#1{\mathrm{\bf SL}(#1)}
\def \la#1{\mathfrak{#1}}
\def \g{\mathfrak{g}}
\def \k{\mathfrak{k}}
\def \h{\mathfrak{h}}
\def \m{\mathfrak{m}}
\def \u#1{\mathfrak{u}(#1)}
\def \o#1{\mathfrak{o}(#1)}
\def \su#1{\mathfrak{su}(#1)}
\def \sp#1{\mathfrak{sp}(#1)}
\def \so#1{\mathfrak{so}(#1)}
\def \gl#1#2{\mathfrak{gl}_{#1}(#2)}
\def \sl#1#2{\mathfrak{sl}_{#1}(#2)}
\def \sym#1{\mathrm{\bf Sym}(#1)}


\large 

\frontcover 

\thispagestyle{empty}
\centerline {\bf\Large Abstract}
\vskip1cm

Given a symmetric triple $(G,K,\sigma)$ of compact type, with $G^{\sigma} = K$, the well known \emph{Cartan embedding} $\hat{\Phi}: G/K \to G$ homothetically embeds the symmetric space $M = G/K$ as a totally geodesic submanifold of $G$. In this thesis we show that $\hat{\Phi}$ and the related $K$-invariant \emph{Cartan map} $\Phi = \hat{\Phi}\circ \pi$ are harmonic. This yields simple formulae relating the tension field $\tau$ and the recently introduced conformality operator $\kappa$ on the symmetric space $M$ to those on the image $\Phi(G)$ of the Cartan map.
\par
We use these formulae to construct common eigenfunctions of the tension field and conformality operator on all the classical irreducible compact symmetric spaces. On the complex and quaternionic Grassmannians we, in addition, construct \emph{eigenfamilies}, which are families of compatible eigenfunctions. In the case of the quaternionic Grassmannians, some of these eigenfamilies are entirely new.
\par
It was recently discovered by S. Gudmundsson, A. Sakovich and M. Sobak that such eigenfunctions can be employed to construct proper $p$-harmonic functions, and that eigenfamilies can be applied to construct complex-valued harmonic morphisms. The results of this thesis can thus be used to construct harmonic morphisms and $p$-harmonic functions on the classical compact irreducible symmetric spaces, and in the case of the quaternionic Grassmannians, new examples of such maps.

\vskip8cm
{\it Throughout this work it has been my firm intention to give reference to the stated results and credit to the work of others. All theorems, propositions, lemmas and examples left unmarked are assumed to be too well known for a reference to be given, or are fruits of my own efforts.}


\newpage 
\thispagestyle{empty}
\phantom{m}

\newpage 
\thispagestyle{empty}
\centerline {\bf\Large Acknowledgments}
\vskip1cm
First and foremost I would like to thank my supervisor Prof. Sigmundur Gudmundsson for helping me find such an interesting topic for a thesis. I would also like to thank him for his patience and attentiveness in helping me write a thesis I can be truly proud of. Also of great help were a few good comments from Thomas Munn for which I am very thankful. Many thanks also go to Johanna Marie Gegenfurtner for her helpful comments and corrections. I also thank Prof. Wolfgang Ziller for giving me access to his very well-written lecture notes, as well as some nice correspondence during the writing of my thesis. Finally I would like to express my deepest gratitude to my friends and family for being so supportive throughout my time here in Lund.
\vskip1pc
\hskip9cm Adam Lindstr\"om
\phantom{m}

\newpage 
\thispagestyle{empty}

\newpage 
\tableofcontents
\thispagestyle{empty}
\phantom{m}

\newpage 
\setcounter{page}{0} 
\thispagestyle{empty}

\chapter{Introduction}\label{ch: intro}

In this thesis we employ the recently discovered tecniques of eigenfamilies and eigenfunctions to find harmonic morphisms and $p$-harmonic maps, respectively, on the classical irreducible symmetric spaces of compact type. In the case of the quaternionic Grassmannians
$$G_m(\hn^{m+n}) = \Sp{m+n}/\Sp{m} \times \Sp{n}$$
we find new eigenfamilies, extending those found by Gudmundsson and Ghandour in \cite{Gud-Gha 2}. These can then be used, as explained in Chapter \ref{ch: eigen}, to construct new $p$-harmonic functions and complex-valued harmonic morphisms on these spaces.

We also describe a new method of constructing eigenfunctions and eigenfamilies on a Riemannian product $N = M_1 \times M_2$, given eigenfunctions and eigenfamilies on the factors $M_1$ and $M_2$. The eigenfunctions and eigenfamilies constructed in this way further do not, in general, descend to eigenfunctions and eigenfamilies on the factors.
\medskip

For the remainder of this introduction we shall give an overwiew and background of the equations studied in this thesis, the domains considered and the methods used for finding solutions.

Recall that a complex-valued function $\phi \in C^2(U)$, defined on some open subset $U$ of $\rn^n$, is said to be harmonic if it satisfies the Laplace equation
\begin{align}\label{intr eq: Laplace rn}
    \Delta(\phi) = \sum_{k=1}^n \frac{\partial^2\phi}{\partial x_k^2} = 0.
\end{align}
Further recall that we may also express the action of the Laplace operator on $\phi$ as the divergence of the gradient, so that Equation \eqref{intr eq: Laplace rn} reads
\begin{align*}
    \Delta(\phi) = \div(\nabla \phi) = 0.
\end{align*}
This second formulation extends readily to the more general setting of the domain being a Riemannian manifold $(M,g)$. For $C^2$-functions $\phi: M \to \cn$ we thus define the \emph{tension field} $\tau$ to be the operator 
\begin{align*}
    \tau(\phi) = \div(\nabla \phi).
\end{align*}
We say that $\phi$ is \emph{harmonic} if it satisfies 
\begin{align*}
    \tau(\phi) = 0.
\end{align*}
\par
Having extended the notion of harmonicity to functions defined on Riemannian manifolds, it is natural to ask which maps between Riemannian manifolds preserve this property, which brings us to the topic of harmonic morphisms.
\par

Let $(M,g)$ and $(N,h)$ be Riemannian manifolds. A mapping $\varphi: M \to N$ of $M$ into $N$ is called a \emph{harmonic morphism} if for each locally defined harmonic function $\phi: U \subset N \to \cn$, the composition $\phi \circ \varphi: \varphi^{-1}(U) \subset M \to \cn$ is harmonic as well. 
\par
The study of such maps began with the 1848 paper \cite{Jacobi} by Jacobi. Motivated by the desire to find new harmonic function on $\rn^3$ by lifting harmonic functions defined on the complex plane, he gave a classification of the harmonic morphisms in the special case of $M = \rn^3$ and $N = \cn \cong \rn^2$. He was able to show that such maps must satisfy both the Laplace equation \eqref{intr eq: Laplace rn} and the equation
\begin{align*}
    \ip{\nabla \varphi}{\nabla \varphi} = \sum_{k=1}^3 \left(\frac{\partial\varphi}{\partial x_k}\right)^2 = 0.
\end{align*}
\par

In this thesis we are mainly interested in complex-valued harmonic morphisms. By adapting a result due, independently, to Fuglede \cite{Fuglede} and Ishihara \cite{Ishihara}, we find a remarkably similar characterisation. The complex valued function $\varphi: (M,g) \to \cn$ is a harmonic morphism if and only if it is harmonic and satisfies
\begin{align}\label{intr eq: conformal}
    g(\nabla \varphi,\nabla\varphi) = 0.
\end{align}
By introducing the \emph{conformality operator} $\kappa(\phi,\psi) = g(\nabla \phi,\nabla \psi)$ we can rewrite equation \eqref{intr eq: conformal} as
\begin{align*}
    \kappa(\varphi,\varphi) = 0.
\end{align*}
This is a non-linear and overdetermined system of partial differential equations, hence the problem of existence is highly non-trivial.
\medskip

Acting with the tension field $\tau$ of a Riemannian manifold $(M,g)$ on a smooth function $\phi: M \to \cn$ yields another smooth function $\tau(\phi): M \to \cn$. It is a natural question under which condition this resulting function is harmonic. In this case we say that $\phi$ is \emph{bi-harmonic}. In general we call the function $\phi$ $p$-harmonic if it satisfies the equation
\begin{align*}
    \tau^p(\phi) = 0.
\end{align*}
Here $\tau^p$ is the $p$-th iterate of the tension field, defined recursively by
$$
    \tau^0(\phi) = \phi,
$$
$$
\tau^p(\phi) = \tau(\tau^{p-1}(\phi)).
$$
Of course a $p$-harmonic function is also $q$-harmonic for any $q \geq p$, so what we are in particular trying to find are $p$-harmonic functions which are not $(p-1)$-harmonic. These are known as \emph{proper} $p$-harmonic functions.
\medskip

In their 2008 paper \cite{Gud-Sak}, Gudmundsson and Sakovich developed a new method for generating complex valued harmonic morphisms. Even more recently in the 2020 paper \cite{Gud-Sob}, Gudmundsson and Sobak describe a new method for generating $p$-harmonic functions. Both of these methods rely on \emph{eigenfunctions} and \emph{eigenfamilies}.
\par
We say that complex-valued function $\phi: (M,g) \to \cn$ is an eigenfunction if there exist complex numbers $\lambda$ and $\mu$ such that
$$
\tau(\phi) = \lambda \cdot \phi \quad \text{and} \quad \kappa(\phi,\phi) = \mu \cdot \phi^2.
$$
Let $\mathcal{E}_{\lambda,\mu}$ is a collection of eigenfunctions, all with the same eigenvalues $\lambda$ and $\mu$. We call $\mathcal{E}_{\lambda,\mu}$ an eigenfamily if we also have
$$
\kappa(\phi,\psi) =\mu \cdot \phi\psi
$$
for all $\phi,\psi \in \mathcal{E}_{\lambda,\mu}$.
\par

These methods are described in detail in Chapter \ref{ch: eigen}. The problem of finding harmonic morphisms and $p$-harmonic maps is thus reduced to this joint eigenvalue problem. 
\medskip

The choice of Riemannian manifold $(M,g)$ to serve as the domain has a great impact on the difficulty of, and methods needed to solve, this eigenvalue problem. In this thesis we focus our attention on the case when $(M,g)$ is a Riemannian symmetric space. These are Riemannian manifolds where for each point $p\in M$ there is an isometry $s_p$ of $M$ fixing $p$ and reversing geodesics through $p$. These spaces are always Riemannian homogeneous, so that we may write them as a quotient
$$
M = G/K
$$
of the identity component $G$ of the isometry group of $M$ and a compact subgroup $K$ being the isotropy group of a fixed point $p_0 \in M$. One can then show that there is a one-to-one correspondence between eigenfunctions $\phi$ on $M = G/K$ and $K$-invariant eigenfunctions on the Lie group $G$. This correspondence also holds for eigenfamilies. Our goal in this thesis is then to find such eigenfunctions and eigenfamilies in the special case that $M$ is one of the classical compact irreducible symmetric spaces, which are described in detail in Chapter \ref{ch: symmetric spaces} and Chapter \ref{ch: results}.
\par

It can be shown that for each symmetric space $M = G/K$, there exists an involutive automorphism $\sigma$ of $G$ such that $K$ lies between the compact subgroup $G^{\sigma}$ fixed by $\sigma$ and its identity component $G_{0}^\sigma$. We refer to $(G,K,\sigma)$ as a symmetric triple whenever this holds. More remarkable is that this correspondence holds in both directions, so that for each symmetric triple $(G,K,\sigma)$, there exists a suitable metric $g$ on the homogeneous space $M = G/K$ making $(M,g)$ a Riemannian symmetric space.
\par

In the particular case that the symmetric space $M = G/K$ is of compact type, with $K = G^{\sigma}$ we can use $\sigma$ to define a totally geodesic embedding of $M$ into $G$ known as the \emph{Cartan embedding}. We first define the \emph{Cartan map} $\Phi: G \to G$ by
\begin{align*}
    \Phi(g) = g\sigma(g^{-1}).
\end{align*}
This map is clearly $K$-invariant, and the Cartan embedding is the induced map $\hat{\Phi}: M \to G$ given as
\begin{align}
    \hat{\Phi}(gK) = \Phi(g),
\end{align}
\par

We show in this thesis that both the Cartan map and Cartan embedding are harmonic, in a sense defined in Chapter \ref{ch: harmonic morphisms}. We then use this to find workable formulae to calculate the action of the tension field and conformality operator of $G$ on $K$-invariant functions of the form $\psi = \phi \circ \Phi: G \to \cn$. Here $\phi: G \to \cn$ is an arbitrary smooth complex-valued function on $G$. 
\par
In Chapter \ref{ch: results} we use this to find $K$-invariant eigenfunctions of this form for all of the classical compact irreducible symmetric spaces $G/K$. In the case of the complex and quaternionic Grassmannians we also find eigenfamilies of such functions. By $K$-invariance these induce eigenfunctions and eigenfamilies on the corresponding symmetric spaces.

\chapter{Lie Groups and Lie Algebras}\label{ch: lie groups}

In order to describe the properties of the Cartan embedding, and how it can be used to find harmonic morphisms and $p$-harmonic functions, it will be neccessary to give a solid introduction to both Lie groups and symmetric spaces. In this first chapter we give an introduction to what Lie groups and Lie algebras are, and how they are connected. We introduce important tools for the study of Lie groups, such as the exponential map and the adjoint representation. A standard reference for this theory is the book \emph{Differential Geometry, Lie groups, and Symmetric Spaces} \cite{Helgason} by S. Helgason. An alternative more recent introduction to these topics can be found in the lecture notes \cite{Ziller's notes} by W. Ziller.
\par
Particular attention will be given to the class of compact Lie groups, as this is the case treated mostly in the rest of the thesis. For these we shall show that they can be equipped with a natural bi-invariant Riemannian metric. In chapter \ref{ch: classical groups} we will give a more detailed list of examples consisting of some of the classical compact matrix Lie groups.
\par
A Lie group is a mathematical object which is at the same time both a group and a smooth manifold, with these two different structures being compatible. More precisely we define them as follows.

\begin{definition}
A \emph{Lie group} is a smooth manifold $G$ equipped with a group structure such that the multiplication map $\mu:G \times G \to G$
\begin{align*}
    \mu:(p,q) \mapsto p\cdot q
\end{align*}
and the inversion map $\iota: G \to G$
\begin{align*}
    \iota: p \mapsto p^{-1}
\end{align*}
are both smooth.
\end{definition}

\begin{definition}
Let $G$ be a Lie group and $p$ an element of $G$. Then we refer to the maps $L_p, R_p: G \to G$ given by
\begin{align*}
    L_p: q \mapsto p\cdot q,
\end{align*}
\begin{align*}
    R_p: q \mapsto q \cdot p, 
\end{align*}
as left and right translation by $p$, respectively.
\end{definition}
\begin{proposition}
Let $G$ be a Lie group and $p$ an element of $G$. Then the left and right translations by $p$ are diffeomorphisms of $G$.
\end{proposition}
\begin{proof}
The maps are bijective since $G$ is a group, and they are smooth since the multiplication in $G$ is smooth. 
\end{proof}
\begin{definition}
Let $\mathbb{F} \in \{\rn,\cn\}$ be either the field of real numbers or that of complex numbers. Then a \emph{Lie algebra} over $\mathbb{F}$ is a vector space $\la{g}$ over $\mathbb{F}$ equipped with a bilinear form $[\cdot,\cdot]: \la{g}\times \la{g} \to \la{g}$ called the Lie bracket on $\la{g}$ satisfying,
\begin{enumerate}[label = (\roman*)]
    \item $[X,Y] = -[Y,X]$,
    \item $[X,[Y,Z]] + [Y,[Z,X]] + [Z,[X,Y]] = 0$,
\end{enumerate}
for all $X,Y,Z \in \la{g}$. Condition (i) is simply saying that the Lie bracket is antisymmetric, while condtion (ii) is refered to as the Jacobi identity.
\end{definition}
\begin{remark}
If $\la{g}$ is a Lie algebra and $\la{h}$ and $\la{k}$ are vector subspace of $\la{g}$. Then by the notation $[\la{h}, \la{k}]$ we mean the subspace of $\la{g}$ spanned by all elements of the form $[X,Y]$ with $X \in \ \la{h}$ and $Y \in \la{k}$.
\end{remark}
Most of the Lie algebras encountered in this thesis will be real. Before we continue we give a few examples of Lie algebras which should be familiar.
\begin{example}
Let V = $\rn^3$ and $[X,Y] = X \times Y$ for $X,Y \in \rn^3$ be the ordinary cross product. Then $(\rn^3,\times)$ is a real Lie algebra. The skew symmetry is immediate from the definition of the cross product, while showing the Jacobi identity is a standard exercise in linear algebra.
\end{example}
\begin{example}
Let $M$ be a smooth manifold and recall that the commutator $[X,Y]$ of a pair of smooth vector fields $X, Y \in \smooth{TM}$ on $M$ is given by
\begin{align*}
    [X,Y](f) = X(Y(f))-Y(X(f)),
\end{align*}
for smooth functions $f \in \smooth{TM}$. Then $(\smooth{TM},[\cdot,\cdot])$ is a real Lie algebra of infinite dimension. For a proof of this see for example \cite{Sigmundur's notes}.
\end{example}
\begin{definition}
Let $G$ and $H$ be Lie groups such that $H$ is a subgroup of $G$. If the inclusion map of $H$ in $G$ is an immersion, we say that $H$ is a \emph{Lie subgroup} of $G$.
\end{definition}

\begin{definition}
Let $\la{g}$ be a Lie algebra and $\la{h}$ be a subspace of $\la{g}$. Then we say that $\la{h}$ is a \emph{Lie subalgebra} of $\la{g}$ if 
\begin{align*}
    [\la{h},\la{h}] \subset \la{h}.
\end{align*}
If in addition 
\begin{align*}
    [\la{h},\la{g}] \subset \la{h}
\end{align*}
holds, we say that $\la{h}$ is an \emph{ideal} in $\la{g}$. 
\end{definition}
The following important result is usually refered to as the closed subgroup theorem. It will be very useful later on when we want to check whether certain subgroups are Lie subgroups.
\begin{theorem}[\cite{Helgason}]\label{closed subgroup}
Let $G$ be a Lie group and $H$ a subgroup of $G$. Then if $H$ is a closed subset of $G$ it is a Lie subgroup.
\end{theorem}
\begin{proof}
 For a proof of this result see the proof of Theorem 2.3. in \cite{Helgason}.
\end{proof}
The Lie algebra associated to a given Lie group will be a finite dimensional subalgebra of the Lie algebra of smooth vector fields. To define it we first will need to introduce the notion of a left-invariant vector field.
\begin{definition}
Let $G$ be a Lie group. We say that a vector field $X \in \smooth{TG}$ is \emph{left-invariant} if for all points $p,q \in G$ 
\begin{align*}
    (dL_p)_q(X_q) = X_{pq},
\end{align*}
or at the level of vector fields
\begin{align*}
    dL_p(X) = X.
\end{align*}
One easily sees that the left-invariant vector fields form a real vector space, which we denote by $\la{g}$.
\end{definition}
The following well known result often allows for a simple description of the space of left invariant vector fields on a Lie group.
\begin{proposition}\label{left-invariant tangent space correspondence}
Let $G$ be a Lie group. Then the space $\la{g}$ of left-invariant vector fields is isomorphic to the tangent space $T_e G$ at the identity of $G$. In particular, we have $\dim(\la{g}) = \dim(G)$.
\end{proposition}
\begin{proof}
Define the maps $\phi: \la{g} \to T_e G$ and $\psi: T_e G \to \la{g}$ by 
\begin{align*}
    \phi: X \mapsto X_e
\end{align*}
and 
\begin{align*}
    \psi: X_e \mapsto (X: g \mapsto (dL_g)_e (X_e)),
\end{align*}
respectively. Then both $\phi$ and $\psi$ are linear and injective and we have
\begin{align*}
    \psi(\phi(X))_g = (dL_g)_e(X_e) = X_g
\end{align*}
since $X$ is left-invariant, and
\begin{align*}
    \phi(\psi(X_e)) = (dL_e)_e(X_e) = X_e
\end{align*}
so that $\phi$ and $\psi$ are inverses of each other and $\la{g} \simeq T_e G$ as vector spaces.
\end{proof}

\begin{remark}
In view of Proposition \ref{left-invariant tangent space correspondence} we shall often use the symbol $\la{g}$ to denote both the set of left-invariant vector fields on the Lie group $G$ and its tangent space at $e$. For the most part this should not cause any confusion, but whenever ambiguity may arise we shall point out explicitly which vector space is meant.
\end{remark}

 We shall show that the set of left-invariant vector fields on a Lie group $G$ forms a Lie subalgebra $\la{g}$ of the set of smooth vector fields on $G$. For this we first need the following result.
\begin{lemma}[\cite{Sigmundur's notes}]\label{lemma phi related}
Let $M$ and $N$ be smooth manifolds, $\varphi:M \to N$ be a smooth surjective map and the vector fields $\bar{Y}$, $ \bar{X} \in \smooth{TN}$ be $\varphi$-related to $X$, $Y \in \smooth{TM}$, i.e. $d\varphi(X) = \bar{X}$ and $d\varphi(Y) = \bar{Y}$. Then we have
\begin{align*}
    d\varphi_p([X,Y]_p) = [\bar{X},\bar{Y}]_{\varphi(p)},
\end{align*}
or equivalently, at the level of vector fields
\begin{align*}
    d\varphi([X,Y]) = [\bar{X},\bar{Y}].
\end{align*}
\end{lemma}
For a proof of Lemma \ref{lemma phi related}, see Proposition 4.27 in \cite{Sigmundur's notes}.
\begin{proposition}[\cite{Sigmundur's notes}]\label{Liealg of Liegrp}
Let $G$ be a Lie group and $X, Y \in \la{g}$ be left-invariant vector fields on $G$. Then their Lie bracket $[X,Y]$ is also left-invariant.
Hence, $(\la{g},[\cdot,\cdot])$ is a Lie subalgebra of $\smooth{TM}$. We will refer to $\la{g}$ as the \emph{Lie algebra of $G$}.
\end{proposition}
\begin{proof}
Let $g \in G$. Then by Lemma \ref{lemma phi related}
\begin{align*}
    dL_p([X,Y]) = [dL_p(X),dL_p(Y)] = [X,Y] 
\end{align*}
since $X$ and $Y$ are left-invariant.
\end{proof}
We will now desccribe some of the most important examples of Lie groups, those being the general linear groups $\GL{n,\rn}$ and $\GL{n,\cn}$.
\begin{definition}
Let $n > 0$ be an integer and define the general linear group $\GL{n,\mathbb{F}}$, where $\mathbb{F}$ is either of the fields $\rn$ or $\cn$, by
\begin{align*}
    \GL{n,\mathbb{F}} = \{ a \in \mathbb{F}^{n \times n} \ | \ \det(a) \neq 0 \}.
\end{align*}
Equivalently, $\GL{n,\mathbf{F}}$ is the set of invertible $n$ by $n$ matricies over $\mathbb{F}$. 
\end{definition}
$\GL{n,\rn}$ and $\GL{n,\cn}$ are clearly groups under matrix multiplication. We shall show that they satisfy the requirements of being Lie groups.
\begin{proposition}\label{GL a Lie group}
Let $n>0$ be an integer. Then the matrix groups $\GL{n,\rn}$ and $\GL{n,\cn}$ are Lie groups of dimensions $n^2$ and $2n^2$, with Lie algebras $\gl{n}{\rn} = \rn^{n\times n}$ and $\gl{n}{\cn} = \cn^{n\times n}$, respectively. In both cases the Lie bracket is given by
\begin{align*}
    [X,Y] = X\cdot Y-Y\cdot X
\end{align*}
of $n\times n$ matrices.
\end{proposition}
\begin{proof}
We shall prove the proposition holds for $\GL{n,\cn}$, as the proof for the real case is identical.

The determinant $\det: \cn^{n\times n} \to \cn$ is a continuous map. Hence, since the set $\{0\}\subset\cn$ is closed, so is the set of matrices with zero determinant.
Thus the complement, $\GL{n,\cn}$, is an open subset of $\cn^{n\times n}$. Since open subsets of euclidean spaces are always smooth manifolds, this shows $\GL{n,\cn}$ is a smooth manifold of dimension $2n^2$. Multiplication and inversion of matrices is given by polynomial, respectively, rational functions of the components. Both are hence smooth operations, making $\GL{n,\cn}$ a Lie group. 
\par
For the Lie algebra, we can identify the tangent space at the identity of $\GL{n,\cn}$ by the ambient euclidean space $\cn^{n \times n}$. Hence, $\gl{n}{\cn} = \cn^{n\times n}$.
To see that the Lie bracket is given by the matrix commutator, we note first that left translation in these cases is a linear operation. Hence, for  $X \in T_e \GL{n,\cn}$ and $a \in \GL{n,\cn}$, we have
\begin{align*}
    (dL_a)_e(X) = a\cdot X.
\end{align*}
For a  vector field $X \in \gl{n}{\cn}$ we then have $X_a = a\cdot X_e$. 

Let $x_{ij}$ be the coordinate functions on $\cn^{n\times n}$ given by $x_{ij}(a) = a_{ij}$. Since these functions are also linear we have 
\begin{align*}
    X_a(x_{ij}) = x_{ij}(X_a) = (aX)_{ij}.
\end{align*}
Further, $X_e(Y(x_{ij}))$ is given by
\begin{align*}
    X_e(Y(x_{ij})) = \restr{(X_e aY_e)_{ij}}{a = e} = (X_e Y_e)_{ij}
\end{align*}
and hence, 
\begin{align*}
    ([X,Y]_e)_{ij} = (X_e Y_e - Y_e X_e)_{ij}
\end{align*}
for $0 < i,j < n$, which proves the claim.
\end{proof}
Another very important class of Lie groups are the isometry groups of Riemannian manifolds. This result was first proven by S. B Myers and N. B. Steenrod in their 1939 paper {\it The group of isometries of a Riemannian manifold} \cite{Myers-Steenrod}. 
\begin{theorem}[\cite{Myers-Steenrod}]\label{isometry lie}
Let $(M,g)$ be a Riemannian manifold and $G = I(M)$ the isometry group of $M$. Then $G$ is a Lie group when endowed with the compact open topology.
\end{theorem}
\begin{definition}
Let $G$ be a Lie group with Lie algebra $\la{g}$. A Riemannian metric $g$ on $G$ is said to be \emph{left-invariant} if all left translations are isometries with respect to $g$.
\end{definition}

\begin{proposition}[\cite{Sigmundur's notes,Ziller's notes}]\label{left invariant metrics}
Let $G$ be a Lie group. Then there exists a left invariant Riemannian metric $g$ on $G$.
\end{proposition}
\begin{proof}
Let $\la{g} \cong T_e G$ be the Lie algebra of $G$ and let $\ip{\cdot}{\cdot}$ be an inner product on $\la{g}$. We will extend it to a Riemannian metric $g$ on $G$ by defining
\begin{align*}
    g_p(X_p,Y_p) = \ip{(dL_{p^{-1}})_{p}X_p}{(dL_{p^{-1}})_{p}Y_p},
\end{align*}
for $p \in G$ and $X_p, Y_p \in T_p G$.
Clearly $g_p$ is an inner product on $T_p G$ and the metric thus defined on $G$ is smooth since left translations are smooth. To see that $g$ is left invariant let $p, q \in G$ and $X_p, Y_p \in T_p G$. Then,
\begin{align*}
    g_{L_q(p)}((dL_q)_p X_p, (dL_q)_p Y_p) &= \ip{(dL_{p^{-1}q^{-1}})_{qp}(dL_q)_p X_p}{(dL_{p^{-1}q^{-1}})_{qp}(dL_q)_p Y_p}\\
    &= \ip{(dL_{p^{-1}})_p X_p}{(dL_{p^{-1}})_p Y_p}\\
    &= g_p(X_p,Y_p)
\end{align*}
as desired.
\end{proof}
\begin{definition}\label{Homomorphisms}
Let $G$ and $H$ be Lie groups and $\la{g}$ and $\la{h}$ be their Lie algebras. Then a group homomorphism $\varphi: G \to H$ is a \emph{Lie group homomorphism} if it is also smooth. A linear map $\psi: \la{g} \to \la{h}$ is a \emph{Lie algebra homomorphism} if it satisfies
\begin{align*}
    \psi([X,Y]_{\la{g}}) = [\psi(X),\psi(Y)]_{\la{h}},
\end{align*}
for all $X, Y \in \la{g}$. If the maps $\varphi$ and $\psi$ are isomorphisms of the respective structures we call them Lie group, respectively, algebra - isomorphisms.
\end{definition}
\begin{proposition}[\cite{Ziller's notes}]
Let $G$ and $H$ be Lie groups with Lie algebras $\la{g}$ and $\la{h}$, and let $\varphi: G \to H$ be a Lie group homomorphism. Then the differential $d\varphi_e: \la{g} \to \la{h}$ of $\varphi$ at the identity $e \in G$ is a Lie algebra homomorphism.
\end{proposition}
\begin{proof}
Let $X$ and $Y$ be left-invariant vector fields on $G$ and denote by $\bar{X}$ and $\bar{Y}$ the left-invariant vector fields on $H$ such that $\bar{X}_e = d\varphi_e(X_e)$ and $\bar{Y}_e = d\varphi_e(Y_e)$. We shall show that $\bar{X}$ and $\bar{Y}$ are $\varphi$-related to $X$ and $Y$, respectively. Indeed we have
\begin{align*}
    \bar{X}_{\varphi(g)} &= (dL_{\varphi(g)})_e(\bar{X}_e)\\
    &= (dL_{\varphi(g)})_e(d\varphi_e(X_e))\\
    &= d(L_{\varphi(g)}\circ \varphi)_e(X_e)\\
    &= d(\varphi \circ L_g)_e(X_e)\\
    &= d\varphi_g(X_g),
\end{align*}
so that $\bar{X} = d\varphi(X)$. A similar computation shows that $\bar{Y} = d\varphi(Y)$. With this we are done as then
\begin{align*}
    d\varphi_e([X,Y]_e) = [\bar{X},\bar{Y}]_e,
\end{align*}
 as desired, since differentials respect the Lie bracket.
\end{proof}
\begin{corollary}
Let $G$ be a Lie group with Lie algebra $\la{g}$ and let $H$ be a Lie subgroup of $G$. Then the Lie algebra of $H$ satisfies $\la{h} \cong T_e H \subset T_e G \cong \la{g}$ and $\la{h}$ is a subalgebra of $\la{g}$
\end{corollary}
\begin{proof}
All we need to do is notice that the inclusion map $i:H \to G$ is a Lie group homomorphism. Hence the differential at identity is a Lie algebra homomorphism. But $di_e:T_{}$ is clearly just the inclusion map of $\la{h} \cong T_e H$ into $\la{g}\cong T_e G$. Hence $\la{h}$ is a Lie subalgebra of $\la{g}$ 
\end{proof}

\begin{definition}
Let $M$ be a smooth manifold and $X \in \smooth{M}$ a vector field on $M$. We say that a smooth curve $\gamma:I \to M$ is an \emph{integral curve to $X$} if for each $t \in I$
\begin{align*}
    \dot{\gamma}(t) = X_{\gamma(t)}
\end{align*}
\end{definition}
One can show both local existence and uniqueness of such curves in the following sense.
\begin{proposition}[\cite{Lee}]\label{local existence and uniqueness}
Let $M$ be a smooth manifold and $X$ a smooth vector field on $M$. Then for each point $p$ of $M$ there exists a real number $\epsilon > 0$ and a curve $\gamma:(-\epsilon,\epsilon) \to M$ such that $\gamma$ is an integral curve of $X$ and $\gamma(0) = p$.
Furthermore, if $\gamma_1:(-a,a) \to M$ and $\gamma_2: (-b,b) \to M$ are integral curves of $X$ such that $\gamma_1(0) = \gamma_2(0)$, then $\gamma_1(t) = \gamma_2(t)$ for all $t \in (-a,a)\cap(-b,b)$.
\end{proposition}
\begin{definition}\label{definition of local flow}
Let $M$ be a smooth manifold and $X$ a smooth vector field on $M$. A \emph{flow box} for $X$ at a point $p\in M$ is then a triple $(U,\epsilon,\varphi^X)$ where $U$ is an open neighbourhood of $p$, $\epsilon$ a positive real number, and $\varphi^X$ a smooth map
\begin{align*}
    \varphi^X: (-\epsilon, \epsilon)\times U \to M
\end{align*}
such that for each point $q \in U$ the curve $\gamma_q: t \mapsto \varphi^X(t,q)$ is an integral curve of $X$ with $\gamma_q(0) = q$. We call the map $\varphi^X$ the \emph{local flow} of the vector field $X$. It is also common to use the notation $\varphi^X(t,q) = \varphi^X_t(q)$ for the local flow.
\end{definition}
The following shows that local flows always exist for smooth vector fields. This corresponds to Theorem 2.91 in \cite{Lee}.
\begin{proposition}[\cite{Lee}]\label{existence of local flows}
Let $M$ be a smooth manifold and $X$ a smooth vector field on $M$. Then for each point $p \in M$ there exists a local flow box $(U,\epsilon,\varphi^X)$. Furthermore, if $(U_1, a, \varphi_1^X)$ and $(U_2,b,\varphi_2^X)$ are flow boxes at $p \in M$, then $\varphi_1^X$ and $\varphi_2^X$ agree on the intersection $((-a,a)\cap (-b,b))\times (U_1\cap U_2)$.
\end{proposition}
For a proof of this result see the proof of Theorem 2.91 in the book \emph{Manifolds and differential geometry} by J. M. Lee (\cite{Lee}). We recommend chapter 2 of the aforementioned book for a thorough exposition on the properties of flows of vector fields.
\begin{remark}
If $(U,\epsilon, \varphi^X)$ is a local flow box for the smooth vector field $X$ on $M$, then it follows from the properties of integral curves that the maps $\varphi_t^X:q \mapsto \varphi^X(t,q)$ and $\varphi^X_{-t}: q\mapsto \varphi^X(-t,q)$ are mutual inverses. Since they are also smooth it follows that the local flow of a smooth vector field is a family of local diffeomorphism of the manifold $M$.
\end{remark}
Local flows can be used to give an alternative definition and interpretation of the Lie bracket between smooth vector fields as the rate of change of one vector field along the flow of another.
\begin{proposition}[\cite{Lee}]\label{Lie bracket from flow}
Let $M$ be a smooth manifold and $X, Y \in \smooth{TM}$. Let $(U,\epsilon, \varphi^X)$ be a local flow box for $X$ at the point $p \in M$. Set $\varphi^X_t: q \mapsto \varphi^X(t,q)$. Then the Lie bracket of $X$ and $Y$ at $p$ satisfies
\begin{align*}
    [X,Y]_p
    = \lim_{t\to 0}\frac{(d\varphi^X_{-t})_{\varphi^X_t(p)}(Y_{\varphi^X_t(p)})-Y_p}{t}
    = \dtatzero(d\varphi^X_{-t})_{\varphi^X_t(p)}(Y_{\varphi^X_t(p)}).
\end{align*}
\end{proposition}
For a proof of this result see the proof of Proposition 2.105 in \cite{Lee}. 
As we shall later see, the flows of left invariant vector fields on Lie groups can be described explicitly.
In order to define the important exponential map of a Lie group, we must first define what is meant by a one-parameter subgroup.
\begin{definition}
Let $G$ be a Lie group. Then a \emph{one-parameter subgroup} is a Lie group homomorphism $\varphi$ of $(\rn,+)$ into $G$. It is hence a smooth map $\varphi: \rn \to G$ satisfying $\varphi(s+t)=\varphi(s)\cdot \varphi(t)$ for all $s,t \in \rn$.
\end{definition}
The one-parameter subgroups are by definition smooth curves in $G$ through the identity $e$. In fact, much like for geodesics, there is a unique one-parameter subgroup $\varphi_X$ with $\dot{\varphi}_X (0) = X$ for each vector $X \in T_e G \cong \la{g}$. To see this we will need the following result. 
\begin{proposition}[\cite{Ziller's notes}]\label{uniquenes lemma}
Let $G$ and $H$ be Lie groups with Lie algebras $\la{g}$ and $\la{h}$, and with $H$ simply connected. Then for any Lie algebra homomorphism $\psi: \la{h} \to \la{g}$, there is a unique Lie group homomorphism $\varphi: H \to G$ satisfying $d\varphi_e = \psi$.
\end{proposition}
The proof of Proposition \ref{uniquenes lemma} relies on some notions from covering space theory, and will not be provided here. For the full proof, as well as an introduction to the covering space theory of Lie groups, we recommend again the lecture notes \cite{Ziller's notes} of W. Ziller.

\begin{corollary}\label{existence of 1-par subgroups}
Let $G$ be a Lie group with Lie algebra $\la{g}$. Then for each $X \in \la{g}$ there exists a unique one-parameter subgroup $\varphi_X$ such that $\dot{\varphi}_X(0) =X$.
\end{corollary}
\begin{proof}
Let $\psi_X:\rn \to \la{g}$ be the Lie algebra homomorphism $\psi_X(t) = tX$. Then since $\rn$ is the Lie algebra of $(\rn,+)$ and the later is simply connected, there exists by Proposition \ref{uniquenes lemma} a unique Lie group homomorphism $\varphi:\rn \to G$ such that $(d\varphi_X)_0 = \psi_X$. Then
$$\dot{\varphi}_X(0) = (d\varphi_X)_0(1) = \psi_X(1) = X.$$
\end{proof}
\begin{definition}\label{exponential map}
Let $G$ be a Lie group with Lie algebra $\la{g}$, and define the \emph{exponential map}, $\exp: \la{g} \to G$ of G by
\begin{align*}
    \exp(X) = \varphi_X(1)
\end{align*}
where $\varphi_X$ is the unique one-parameter subgroup with $\dot{\varphi}_X(0) = X$.
\end{definition}
\begin{proposition}[\cite{Ziller's notes}]\label{properties of exp}
Let $G$ be a Lie group with Lie algebra $\la{g}$. Then the exponential map of $G$ satisfies the following properties.
\begin{enumerate}[label = (\roman*)]
    \item For each $X \in \la{g}$, $\varphi(t) = \exp(tX)$ is the one-parameter subgroup $\varphi_{X}(t)$ of $G$ with $\dot{\varphi}(0) = X$.
    \item The integral curve of $X\in \la{g}$ through the point $p\in G$ is given by $\gamma(t) = p\exp(tX)$.
    \item If $H$ is a Lie group and $\varphi: G \to H$ a Lie group homomorphism, then
    $$\varphi(\exp_{G}(X)) = \exp_{H}(d\varphi_e(X)).$$
    \item If $H\subset G$ is a Lie subgroup with Lie algebra $\la{}$, then  $X\in \la{g}$ satisfies $\exp(tX) \in H$ for all $t$ if and only if $X\in \la{h}$.
\end{enumerate}
\end{proposition}
\begin{proof}
To show parts (i) and (ii) we first show that the one-parameter subgroup $\varphi_X$ is an integral curve of the vector field $X$ through $e$. Indeed we have,
\begin{align*}
    \dot{\varphi}_X(s) &= \dtatzero \varphi_X(s+t)\\
    &= \dtatzero\varphi_X(s)\varphi_X(t)\\
    &= \dtatzero(L_{\varphi_X(s)})\circ \varphi_X(t)\\
    &=(dL_{\varphi_X(s)})_e(\dot{\varphi}_X(0))\\
    &= (dL_{\varphi_X(s)})_e(X)\\
    &= X_{\varphi_X(s)},
\end{align*}
as desired. Now fix $s$ and consider the two curves $\varphi_{sX}(t)$ and $\varphi_{X}(s\cdot t)$. We claim that these are both one-parameter subgroups with derivative $sX$ at $t=0$. For the first curve this is immediate by definition, while for the second curve, the chain rule gives
\begin{align*}
    \dtatzero\varphi_{X}(s\cdot t) = s\cdot(d\varphi_X)_0(1) = sX.
\end{align*}
Then by uniqueness of one-parameter subgroups the two curves are the same. Taking $s=1$ property (i) follows since then
\begin{align*}
    \exp(tX) = \varphi_{tX}(1) = \varphi_{X}(t)
\end{align*}
is a one-parameter subgroup. Since it is also by the previous computations an integral curve, and since left translations take integral curves of the left invariant vector field $X$ to integral curves of $X$, we have also established (ii). For (iii) we note that both of the curves $\varphi(\exp_G(tX))$ and $\exp_H(d\varphi_e(tX))$ are one-parameter subgroups in $H$. Hence, we need only show that they have the same derivative at identity. For the second curve we have,
\begin{align*}
    \dtatzero \exp_H(d\varphi_e(tX)) = \dtatzero \exp_H(t\cdot d\varphi_e(X)) = d\varphi_e(X), 
\end{align*}
while for the first we get,
\begin{align*}
    \dtatzero \varphi(\exp_G(tX)) = d\varphi_e(\dtatzero \exp_G(tX)) = d\varphi_e(X).
\end{align*}
By the uniqueness of one-parameter subgroups we have then established (iii). For (iv) to see that $X \in \la{h}$ implies $\exp(tX) \in H$ for all $t$ we simply apply (iii) to the inclusion map of $H$ in $G$. The other direction is straightforward since the derivative at $t = 0$ of the curve $\exp(tX)$ is certainly in $\la{h} \cong T_e H$.
\end{proof}
\begin{corollary}[\cite{Ziller's notes}]\label{flow of left invariant}
Let $G$ be a Lie group and $X \in \la{g}$ a left invariant vector field on $G$. Then the flow of $X$ is globaly defined and given by
\begin{align*}
    \varphi^X_t = R_{\exp(tX)}.
\end{align*}
\end{corollary}
\begin{proof}
Set $\varphi_t = R_{\exp(tX)}$. Then $\varphi_t$ is clearly a one parameter family of diffeomorphisms of $G$ defined for all values of $t$. Fixing $p \in G$ we have
\begin{align*}
    \varphi_t(p) = p\cdot \exp(tX)
\end{align*}
which by property (ii) is the integral curve of $X$ through the point $p$. Hence $\varphi_t$ is a globaly defined local flow of the vector field $X$ as desired.
\end{proof}
\begin{remark}\label{uniqueness of exp}
By the uniqueness of one-parameter subgroups, the exponential map is determined uniquely by property (i) of Proposition \ref{properties of exp}. So that if we can find a map $\psi:T_e G \to G$ for a Lie group $G$ satisfying this property we can be sure it is the unique exponential map.
\end{remark}
With this remark in mind we shall show that the exponential map of the general linear groups is nothing but the ordinary matrix exponential, explaining the name exponential map. First we will need to recall some properties of the matrix exponential. These will also be of use when we later give more examples of matrix Lie groups
\begin{lemma}\label{matrix exp properties}
Recall that the matrix exponential map $\exp: \cn^{n\times n} \to \cn^{n \times n}$ is defined by the convergent power series
\begin{align*}
    \exp(X) = e^X = \sum_{k = 0}^{\infty} \frac{X^k}{k!}.
\end{align*}
This map then satisfies
\begin{enumerate}[label = (\roman*)]
    \item $e^{(X^t)} = (e^X)^t$,
    \item $e^{\bar{X}} = \overline{e^X}$,
    \item $\det(e^X) = e^{\tr(X)}$,
    \item If $XY = YX$, then $e^{X+Y} = (e^X)(e^Y)$.
\end{enumerate}
\end{lemma}
\begin{proposition}\label{GL exp}
Let $G$ be either of the Lie groups $\GL{n,\rn}$ or $\GL{n,\cn}$. Then the exponential map $\exp_G$ is the ordinary matrix exponential
\begin{align*}
    \exp_G(X) = e^{X}
\end{align*}
\end{proposition}
\begin{proof}
We shall prove the proposition for $\GL{n,\cn}$ as the proof for the real case is identical. By property (iii) of Proposition \ref{matrix exp properties} the image of $\gl{n}{\cn} \cong \cn^{n\times n}$ lies in $\GL{n,\cn}$.
Therefore, by Remark \ref{uniqueness of exp} we need only show that the matrix exponential satisfies property (i) of Proposition \ref{properties of exp}. For this let $s,t \in \rn$ and $X \in \gl{n}{\cn}$. Then $sX$ and $tX$ commute, so by property (iv) of Proposition \ref{matrix exp properties} we have
\begin{align*}
    e^{(s+t)X} = e^{sX + tX} = (e^{sX})(e^{tX}).
\end{align*}
Since $e^{0\cdot X} = e^0 = id$, we have shown $e^{tX}$ is a one-parameter subgroup. For the derivative at $0$ we compute
\begin{align*}
    \dtatzero e^{tX} = \dtatzero \sum_{k=0}^{\infty} t^k \frac{X^k}{k!} = \restr{\left(X + \sum_{k=1}^{\infty}t^k \frac{X^{k+1}}{k!} \right)}{t=0} = X.
\end{align*}
Hence, the matrix exponential is the exponential map of the general linear group, as desired.
\end{proof}
\begin{remark}
The Lie group $\GL{n,\rn}$ is not connected since there can be no continuous path connecting any two matrices with determinants of opposite signs lying in $\GL{n,\rn}$. Hence for example, the matrix exponential cannot be surjective in this case.
\end{remark}
It will be of use later to give an expression for the differentials of the multiplication map $\mu$ and the inversion map $\iota$ of a Lie group $G$ at generic points of $G\times G$ and $G$, respectively. We can use the exponential map to derive such an expression.
\begin{proposition}\label{diff mul}
Let $G$ be a Lie group and $\mu: G \times G \to G$, $\iota: G \to G$ be the multiplication map and inversion map, respectively, of $G$. Then for $p, q \in G$, $(X_p,Y_q) \in T_{(p,q)}(G\times G)$ and $Z_p \in T_p G$ we have
\begin{align*}
    d\mu_{(p,q)}(X_p,Y_q) = (dR_q)_p(X_p) + (dL_p)_q(Y_q)
\end{align*}
and
\begin{align*}
   d\iota_{p}(Z_p) = -(dR_{p^{-1}})_e(dL_{p^{-1}})_p(Z_p) 
\end{align*}
\end{proposition}
\begin{proof}
We shall proceed by calculating the differential when only one of the vectors ia nonzero, and then use the linearity of the differential to get a general formula. If $Y_q = 0$ and $X \in \la{g}$ is such that $X_p = (dL_p)_e(X)$ we have
\begin{align*}
    d\mu_{(p,q)}(X_p,0) &= \dtatzero [\mu(p\cdot \exp(tX),q)]\\
    &= \dtatzero[ p \cdot \exp(tX)\cdot q]\\
    &= \dtatzero[(R_q \circ L_p)(\exp(tX))]\\
    &= (dR_q)_p(X_p).
\end{align*}
If instead $X_p = 0$ and $Y \in \la{g}$ is such that $Y_q = (dL_q)_e(Y)$ we get
\begin{align*}
    d\mu_{(p,q)}(0,Y_q) &= \dtatzero[\mu(p,q \cdot \exp(tX))]\\
    &= \dtatzero[p\cdot q \cdot \exp(tX)]\\
    &= \dtatzero[(L_p \circ L_q)(\exp(tX))]\\
    &= (dL_p)_q(Y_q).
\end{align*}
Then by linearity we have
\begin{align*}
    d\mu_{(p,q)}(X_p,Y_q) = (dR_q)_p(X_p) + (dL_p)_q(Y_q),
\end{align*}
 so that we have established the first formula. For the second formula, let $Z\in \la{g}$ be such that $Z_p = (dL_p)_e(Z)$. Then
\begin{align*}
    d\iota_p(Z_p) &= \dtatzero (p\cdot\exp{tZ})\\
    &= \dtatzero \exp(-tZ)\cdot p^{-1}\\
    &= (dR_{p^{-1}})_e(-Z)\\
    &= -(dR_{p^{-1}})_e(dL_{p^{-1}})_p(Z_p)
\end{align*}
as desired.
\end{proof}
We now introduce the important adjoint representations of Lie groups and Lie algebras
\begin{definition}\label{Adjoint def}
Let $G$ be a Lie group with Lie algebra $\la{g}$. The conjugation map $C_p: q \mapsto pqp^{-1}$ is a smooth Lie group automorphim of $G$ for each $q \in G$. We use this to define an automorphism $\Ad{q}$ of $\la{g}$ by
\begin{align*}
    \Ad{q}(X) = (dC_q)_e(X).
\end{align*}
We call the map $\Ad: G \to \GL{\la{g}}$ given by
\begin{align*}
    \Ad{}: q \mapsto \Ad{q}
\end{align*}
the \emph{adjoint representation of $G$}.
\end{definition}
\begin{definition}\label{adjoint def}
Let $\la{g}$ be a Lie algebra. Let $X \in \la{g}$ and denote by $\ad_X$ the vector space endomorphism 
\begin{align*}
    \ad_X(Y) = [X,Y].
\end{align*}
We call the map $\ad{}: \la{g} \to \mathrm{End}(\la{g}) \cong \mathfrak{gl}(\la{g})$ given by
\begin{align*}
    \ad: X  \mapsto \ad_X
\end{align*}
the \emph{adjoint representaion of $\la{g}$}.
\end{definition}
\begin{proposition}
Let $G$ be a Lie group with Lie algebra $\la{g}$. Then $\Ad{}: G \to \mathrm{Aut}(\la{g})$ is a Lie group homomorphims, and $\ad: \la{g} \to \mathrm{End}(\la{g})$ is a Lie algebra homomorphism.
\end{proposition}
\begin{proof}
The conjugation map of $G$ satisfies $C_p\circ C_q = C_{pq}$, so that by the chain rule $\Ad{p}\Ad{q} = \Ad{pq}$. For the Lie algebra adjoint we can write the jacobi identity as
\begin{align*}
    \ad_X\circ \ad_Y(Z) - \ad_Y \circ \ad_X (Z) - \ad_{[X,Y]}(Z) = 0,
\end{align*}
for all $X, Y$ and $Z \in \la{g}$. Hence,
\begin{align*}
    \ad([X,Y]) = [\ad_X,\ad_Y],
\end{align*}
as desired.
\end{proof}
As the names suggest the Lie group and Lie algebra adjoint representations are related.
\begin{proposition}[\cite{Ziller's notes}]\label{adjoint properties}
Let $G$ be a Lie group with Lie algebra $\la{g}$. Let $X \in \la{g}$ and $p \in G$. Then the adjoint representations of $G$ and $\la{g}$ satisfy the following properties.
\begin{enumerate}[label = (\roman*)]
    \item $d(\Ad{})_e(X) = \ad_X$.
    \item $\Ad{\exp(X)} = e^{\ad_X}$.
    \item $\exp(\Ad{p}(X)) = p\exp(X)p^{-1}$.
\end{enumerate}
\end{proposition}
\begin{proof}
For (i) let $Y \in \la{g}$. We then have
\begin{align*}
    d(\Ad{})_e(X)(Y) &= \dtatzero \Ad{\exp(tX)}(Y)\\
    &= \dtatzero (dR_{\exp(-tX)})_{\exp(tX)}(dL_{\exp(tX)})_e(Y)\\
    &= \dtatzero(dR_{\exp(-tX)})_{\exp(tX)}(Y_{\exp(tX)})\\
    &= [X,Y],
\end{align*}
where in the last step we use the formula for the Lie bracket from Proposition \ref{Lie bracket from flow}, since the flow of $X$ is given by $R_{\exp{tX}}$. Both properties (ii) and (iii) then follow from property (iii) in Proposition \ref{properties of exp}.
\end{proof}
\begin{proposition}\label{adjoint on matrix groups}
The adjoint representations of $\GL{n,\rn}$ and $\GL{n,\cn}$ are 
\begin{align*}
    \Ad{p}(X) = pXp^{-1}
\end{align*}
for $p$ in $\GL{n,\rn}$ or $\GL{n,\cn}$ and $X$ in $\gl{n}{\rn}$ or $\gl{n}{\cn}$.
\end{proposition}
\begin{proof}
This follows immediately by the linearity of matrix multiplication.
\end{proof}
\begin{definition}
Let $\la{g}$ be a Lie algebra. Then the \emph{Killing form} on $\la{g}$ is the symmetric bilinear form given by 
\begin{align*}
    B(X,Y) = \tr(\ad_X \circ \ad_Y).
\end{align*}
We note that taking the trace here is a well defined operation, as $\ad_X$ and $\ad_Y$ are linear maps on a finite dimensional vector space. Hence they can be represented by matrices given some basis for $\la{g}$ and the trace is then independent of the particular choice of basis.
\end{definition}
\begin{proposition}\label{Killing form properties}
Let $G$  be a Lie group with Lie algebra $\la{g}$ and let $B$ be the Killing form on $\la{g}$. Then,  for all $X,Y, Z \in \la{g}$ and all $p \in G$,  $B$ satisfies
\begin{enumerate}[label = (\roman*)]
    \item $B(\Ad{p}(X),\Ad{p}(Y)) = B(X,Y)$,
    \item $B(\ad_Z X, Y) + B(X,\ad_Z Y) = 0$,
    \item $\ker(B) = \{ X \in \la{g} \ | \ B(X,Y) = 0 \ \text{for all } Y \in \la{g}\}$, is an ideal in $\la{g}$.
\end{enumerate}
\end{proposition}
\begin{proof}
We shall show (i) by proving the more general formula
\begin{align*}
    B(AX,AY) = B(X,Y)
\end{align*}
for each Lie algebra automorphism $A$ of $\la{g}$. To see this we first note that 
$$\ad_AX = A \circ \ad_X \circ A^{-1}$$
since,
\begin{align*}
    [AX,Y] = A([X,A^{-1}Y]),
\end{align*}
for all $X, Y \in \la{g}$. Then we have
\begin{align*}
    B(AX,AY) &= \tr(\ad_{AX} \circ \ad_{AY})\\
    &= \tr(A \circ \ad_X \circ \ad_Y \circ A^{-1})\\
    &= \tr(\ad_X \circ \ad_Y)\\
    &= B(X,Y).
\end{align*}
This since the trace is invariant under conjugation by isomorphisms. Since $\Ad{p}$ is an automorphism of $\la{g}$, this proves (i). For the second equation we make use of property (ii) of Proposition \ref{adjoint properties} and properties of the matrix exponential to get
\begin{align*}
    \dtatzero \Ad{\exp{tZ}} = \dtatzero e^{t\ad_Z} = \ad_Z,
\end{align*}
for  $Z \in \la{g}$. Then, using (i),
\begin{align*}
    0 &= \dtatzero B(X,Y)\\
    &= \dtatzero B(\Ad{\exp{tZ}}X,\Ad{\exp{tZ}}Y)\\
    &= B(\ad_Z X,Y) + B(X, \ad_Z Y)
\end{align*}
which proves (ii). For (iii) assume that $X \in \ker(B)$ and $Y \in \la{g}$. Then for all $Z \in \la{g}$ we have
\begin{align*}
    B([X,Y],Z) &= -B(\ad_Y(X),Z)\\
    &= B(X,\ad_Y(Z))\\
    &= 0
\end{align*}
by (ii) and since $X \in \ker(B)$. Thus $[Y,X] \in \ker(B)$ as desired. 
\end{proof}
The Killing form is in general not non-degenerate, however it is for many important Lie groups, and we then have the following definition.
\begin{definition}
Let $\la{g}$ be a Lie algebra. Then $\la{g}$ is said to be \emph{semisimple} if the Killing form of $\la{g}$ is non-degenerate.
\end{definition}

\begin{definition}
Let $\la{g}$ be a Lie algebra. Then we call $\la{g}$ a \emph{compact Lie algebra} if there exists a compact Lie group with Lie algebra $\la{g}$.
\end{definition}

The following powerful result about compact Lie groups and Lie algebras we will state without a full proof.
\begin{theorem}[\cite{Knapp,Ziller's notes}]\label{compact bi invariant}
Let $G$ be a compact Lie group with Lie algebra $\la{g}$. Then 
\begin{enumerate}[label = (\roman*)]
    \item There exists a \emph{bi-invariant} Riemannian metric $g$ on $G$, i.e. a metric such that all left and right translations by elements of $G$ are isometries.
    \item There exists an inner product $\ip{\cdot}{\cdot}$ on $\la{g}$ such that $\ad_X$ is skew-symmetric for each $X \in \la{g}$. 
\end{enumerate}
\end{theorem}
\begin{proof}
For a full proof of the statement in \cite{Ziller's notes}. We shall here show that (i) implies (ii). First we claim that the existence of a bi-invariant metric $g$ is equivalent to the existence of an $\Ad{p}$-invariant inner product $\ip{\cdot}{\cdot} = g_e(\cdot,\cdot)$, for each $p\in G$. To see this first assume $g$ is a bi-invariant metric. Then for $X, Y \in \la{g}$ and $p \in G$ we have
\begin{align*}
    g_e(\Ad{p}(X),\Ad{p}(Y)) &= g_e((dL_p)_{p^{-1}}(dR_{p^{-1}})_e(X),(dL_p)_{p^{-1}}(dR_{p^{-1}})_e(Y))\\
    &= g_{p^{-1}}((dR_{p^{-1}})_e(X),(dR_{p^{-1}})_e(Y))\\
    &= g_e(X,Y),
\end{align*}
so that $\ip{\cdot}{\cdot}$ is $\Ad{p}$-invariant. Now let $\ip{\cdot}{\cdot}$ be $\Ad{p}$-invariant for each $p \in G$. Extend this inner product to a left invariant metric as in Proposition \ref{left invariant metrics}. Then for $X, Y \in T_e G$ and $p \in G$ we get
\begin{align*}
    g_p((dR_p)_e(X),(dR_p)_e(Y)) &= \ip{(dL_{p^{-1}})_p(dR_p)_e(X)}{(dL_{p^{-1}})_p(dR_p)_e(Y)}\\
    &= \ip{\Ad{p^{-1}}(X)}{\Ad{p^{-1}}(Y)}\\
    &= \ip{X}{Y}\\
    &= g_e(X,Y)
\end{align*}
so that $g$ is also right invariant. As $\ad_X = d(Ad)_e(X)$ we may proceed exactly as in Proposition \ref{Killing form properties} with $g_e$ in place of $B$ to show that $\ad_X$ is skew symmetric.
\end{proof}
\begin{corollary}\label{compact negative def}
Let $\la{g}$ be a compact semisimple Lie algebra, then its Killing form is negative definite. 
\end{corollary}\label{Killing form compact}
\begin{proof}
Let $X \in \la{g}$. By Theorem \ref{compact bi invariant}, there exists a basis for $\la{g}$ such that $[\ad_X]^t = -[\ad_X]$. Then 
\begin{align*}
    B(X,X) = \tr([\ad_X]^2) = -\tr([\ad_X][\ad_X]^t) \leq 0
\end{align*}
With $B(X,X) = 0$ if and only if $\ad_X = 0$, in which case $B(X,Y) = 0$ for all $Y\in \la{g}$, contradicting semisimplicity.
\end{proof}
From Corollary \ref{compact negative def} it follows that if $G$ is a compact semisimple Lie group and $B$ the killing form of its Lie algebra $\la{g}$, then $-B$ is an inner product on $\la{g}$. Since $B$ is $\Ad{}$-invariant, so is $-B$. Hence an explicit choice of a Bi-invariant metric on $G$ is given by the tensor obtained by left-translating $-B$.
\begin{corollary}\label{inversion isometry}
Let $(G,g)$ be a compact Lie group equipped with a bi-invariant metric. Then the inversion map $\iota: G \to G$ is an isometry.
\end{corollary}
\begin{proof}
Let $p \in G$ and $X_p,Y_p \in T_p G$. Then by Proposition \ref{diff mul} we have
\begin{align*}
    g_{p^{-1}}(d\iota_p(X_p),d\iota_p(Y_p)) &= g_{p^{-1}}((dR_{p^{-1}})_e(dL_{p^{-1}})_p(X_p),(dR_{p^{-1}})_e(dL_{p^{-1}})_p(Y_p))\\
    &= g_e((dL_{p^{-1}})_p(X_p),(dL_{p^{-1}})_p(Y_p))\\
    &= g_p(X_p,Y_p),
\end{align*}
as desired.
\end{proof}
In the case of a compact Lie group we may also express the Levi-Civita connection entirely in terms of the Lie bracket.
\begin{proposition}\label{Levi-Civita on compact}
Let $G$ be a compact Lie group with Lie algebra $\la{g}$, and let $G$ be equipped with a bi-invariant metric $g$. Let $\nabla$ be the Levi-Civita connection on $(G,g)$. Then for left invariant vector fields $X,Y \in \la{g}$ we have
\begin{align*}
    \nab{}{X}{Y} = \tfrac{1}{2} \cdot[X,Y].
\end{align*}
\end{proposition}
\begin{proof}
Let $X, Y, Z \in \la{g}$ be left invariant vector fields. Then the Koszul formula gives
\begin{align}\label{prop 1.50 eq 1}
\begin{split}
    g(\nab{}{X}{Y},Z) &= \tfrac{1}{2}\cdot \{X (g(Y,Z)) + Y (g(Z,X)) - Z (g(X,Y))\\
    &\qquad + g(Z,[X,Y]) + g([Z,X],Y) + g([Z,Y],X)\}\\
    &= \tfrac{1}{2} \cdot \{g(Z,[X,Y]) + g([Z,X],Y) + g([Z,Y],X)\},    
\end{split}
\end{align}
where the first three terms vanish by left invariance. As $g$ is bi-invariant, $\ad_X$ and $\ad_Y$ are skew symmetric with respect to $g$ and equation \eqref{prop 1.50 eq 1} can be rewritten as
\begin{align*}
    g(\nab{}{X}{Y},Z) &= \tfrac{1}{2} \cdot \{g(Z,[X,Y]) + g(Z,[X,Y]) - g(Z,[X,Y])\}\\ &= \tfrac{1}{2}\cdot  g([X,Y],Z).
\end{align*}
We then obtain the desired formula by summing over $Z$ in an orthonormal basis of $\la{g}$.
\end{proof}
\begin{corollary}\label{one param geodesics}
Let $(G,g)$ be a compact Lie group equipped with a bi-invariant metric. Then the one-parameter subgroups
\begin{align*}
    \gamma(t) = \exp(tX), 
\end{align*}
with $X \in \la{g}$ are geodesics through the identity $e$ of $G$. Equivalently, if $\mathrm{Exp_e}$ is the Riemannian exponential map at $e$, then
\begin{align*}
    \exp = \mathrm{Exp_e}.
\end{align*}
\end{corollary}
\begin{proof}
Let $X \in \la{g}$ and put $\gamma(t) = \exp(tX)$ so that $\gamma(0) = e$ and $\dot{\gamma}(t) = X_{\gamma(t)}$. Then by Proposition \ref{Levi-Civita on compact},
\begin{align*}
    (\nab{}{\dot{\gamma}}{\dot{\gamma}})_{\gamma(t)} = \frac{1}{2}[X,X]_{\gamma(t)} = 0
\end{align*}
for all $t \in \rn$. Hence $\gamma$ is a geodesic through the origin with $\dot{\gamma}(0) = X_e$. Since $X \in \la{g}$ was arbitrary we have by the uniqueness of geodesics,
\begin{align*}
    \exp = \mathrm{Exp_e},
\end{align*}
as desired.
\end{proof}
\chapter{The Classical Compact Lie Groups}\label{ch: classical groups}
We will now give some examples of classical Lie groups, with particular focus on the compact ones. First we shall equip the previously discussed Lie groups $\GL{n,\rn}$ and $\GL{n,\cn}$ with their standard left invariant Riemannian metrics.
\begin{example}\label{metric on general linear groups}
We may express the usual inner product on the vector space $\gl{n}{\cn} \cong \cn^{m\times m}$ as
\begin{align*}
    \ip{Z}{W} = \re(\tr(\bar{Z}^t \cdot W)).
\end{align*}
The left invariant metric $g$ on $\GL{n,\cn}$ obtained by left translation of this inner product is then also
\begin{align*}
    g_p(pZ,pW) = \ip{(dL_{p^{-1}})_p(pZ)}{(dL_{p^{-1}})_p(pW)} = \ip{Z}{W} = \re(\tr(\bar{Z}^t \cdot W))
\end{align*}
for each $p \in \GL{n,\cn}$ and $pZ,pW \in T_p \GL{n,\cn}$. Here we have used that $$T_p \GL{n,\cn} = (dL_p)_e(T_e \GL{n,\cn}) = (dL_p)(\gl{n}{\cn}) = p\cdot \gl{n}{\cn},$$
as left translation by $p$ is linear and hence its own differential. This induces a metric $h$ on the subgroup $\GL{n,\rn}$ given by
\begin{align*}
    h_p(pX,pY) = \tr(X^t \cdot Y)
\end{align*}
for each $p\in \GL{n,\rn}$ and each $pX,pY \in T_p \GL{n,\rn}$. One can check that these metrics are \emph{not} bi-invariant. However, we shall see that they induce bi-invariant metrics on certain subgroups.
\end{example}
\begin{definition}
The real and complex \emph{special linear groups}, $\SL{n,\rn}$ and $\SL{n,\cn}$ are the sets of $n\times n$ real and complex matrices $p$ with determinant $\det(p) = 1$, respectively.
\end{definition}
\begin{proposition}
The special linear groups $\SL{n,\rn}$ and $\SL{n,\cn}$ are Lie groups with Lie algebras
\begin{align*}
    \sl{n}{\rn} = \{X \in \gl{n}{\rn} \ | \ \tr(X) = 0 \}
\end{align*}
and
\begin{align*}
    \sl{n}{\cn} = \{X \in \gl{n}{\cn} \ | \ \tr(X) = 0 \}.
\end{align*}
\end{proposition}
\begin{proof}
We prove the Proposition for $\SL{n,\cn}$ as the real case is identical. Since all elements of $\SL{n,\cn}$ have determinant $1 \neq 0$, all elements are invertible so $\SL{n,\cn} \subset \GL{n,\cn}$. If $p, q \in \SL{n,\cn}$ we have
\begin{align*}
    \det(pq) = \det(qp) = 1\cdot 1 = 1
\end{align*}
and 
\begin{align*}
    \det(p^{-1}) = \frac{1}{\det(p)} = 1,
\end{align*}
so that $\SL{n,\cn}$ is a group with respect to matrix multiplication.
Since $\SL{n,\cn} = \det^{-1}({1}) \subset \GL{n,\cn}$ and $\det$ is a continuous map, $\SL{n,\cn}$ is a closed subgroup of $\GL{n,\cn}$. By Proposition \ref{closed subgroup}, $\SL{n,\cn}$ is thus a Lie group. To find an expression of its Lie algebra we note that by (iv) of Proposition \ref{properties of exp} and Proposition \ref{GL exp}, $X \in \gl{n}{\cn}$ belongs to $\sl{n}{\cn}$ if and only if
\begin{align*}
    \det\left(e^{tX}\right) = 1
\end{align*}
for all $t$. Using Proposition \ref{matrix exp properties}
\begin{align*}
    \det\left(e^{tX}\right) = e^{\tr(tX)} = 1
\end{align*}
if and only if $\tr(tX) = 0$. This holds for all $t$ if and only if $\tr(X) = 0$ so that $\sl{n}{\cn}$ is the set of traceless matrices with the Lie bracket being the matrix commutator inherited from $\gl{n}{\cn}$. 
\end{proof}
\begin{definition}
The \emph{orthogonal group}  $\O{n}$ is the set of matrices $x \in \GL{n,\rn}$ such that $x^t\cdot x = I$. The \emph{special orthogonal group} is the intersection $$\SO{n} = \O{n} \cap \SL{n,\rn}.$$
\end{definition}
\begin{definition}
The \emph{unitary group} $\U{n}$ is the set of matrices $z \in \GL{n,\cn}$ such that $\bar{z}^t\cdot z = I$. The \emph{special unitary group} is the intersection $$\SU{n} = \U{n} \cap \SL{n,\cn}.$$
\end{definition}
\begin{proposition}
The matrix Lie groups $\O{n}, \SO{n}, \U{n}$ and $\SU{n}$ are compact groups with Lie algebras
\begin{align*}
    \la{o}(n) = \so{n} = \{X \in \gl{n}{\rn} \ | \ X^t = -X\},
\end{align*}
\begin{align*}
    \u{n} = \{ Z \in \gl{n}{\cn} \ | \ \bar{Z}^t = -Z\}
\end{align*}
and 
\begin{align*}
    \su{n} = \{ Z \in \sl{n}{\cn} \ | \ \bar{X}^t = -X\}.
\end{align*}
\end{proposition}
\begin{proof}
Let $x, y \in \O{n}$. Then
\begin{align*}
    (xy)^t = y^t x^t = y^{-1}x^{-1} = (xy)^{-1}
\end{align*}
and 
\begin{align*}
    (x^{-1})^t = (x^t)^{-1} = x.
\end{align*}
Similarly for $z, w \in \U{n}$ we have
\begin{align*}
    \overline{(zw)}^t = (zw)^{-1}
\end{align*}
and 
\begin{align*}
    \overline{(z^{-1})}^t = z
\end{align*}
so that both $\O{n}$ and $\U{n}$ are groups. Let $\varphi: \GL{n,\rn} \to \GL{n,\rn}$ and $\psi: \GL{n,\cn} \to \GL{n,\cn}$ be the continuous maps given by
\begin{align*}
    \varphi: x \mapsto x^t\cdot x
\end{align*}
and
\begin{align*}
    \psi: z \mapsto \bar{z}^t \cdot z,
\end{align*}
respectively. Then $\O{n} = \varphi^{-1}(\{I_n\})$ and $\U{n} = \psi^{-1}(\{I_n\})$ so they are both closed. By Proposition \ref{closed subgroup} they are then both Lie groups. Recall that the usual euclidean norms on $\rn^{n \times n}$ and $\cn^{n \times n}$ are given by 
\begin{align*}
    \Vert x \Vert^{2} = \tr(x^t\cdot x) 
\end{align*}
and
\begin{align*}
    \Vert z \Vert^{2} = \re(\tr(\bar{z}^t \cdot z)),
\end{align*}
respectively. Then for $x \in \O{n}$ and $z \in \U{n}$ we have $$\Vert x \Vert^{2}_{\rn^{n\times n}} = \Vert z \Vert^{2}_{\cn^{n\times n}} = 1$$ so that $\O{n}$ and $\U{n}$ are bounded and hence compact. To find the Lie algebra of $\O{n}$ we again, by (iv) from Proposition \ref{properties of exp} and Proposition \ref{GL exp}, look for $X \in \gl{n}{\rn}$ satisfying
\begin{align*}
    \left(e^{tX}\right)^t \left(e^{tX}\right) = I_n
\end{align*}
for all values of $t$. Assuming $X \in \o{n}$ and using Proposition \ref{matrix exp properties} we compute
\begin{align*}
    I_n = \left(e^{tX}\right)^t \left(e^{tX}\right) = e^{t(X^t + X)}
\end{align*}
and hence we must have $X^t + X = 0$ for $X \in \o{n}$. If $X^t = -X$ we have
\begin{align*}
    \left(e^{tX}\right)^t = e^{tX^t} = e^{-tX} = \left(e^{tX}\right)^{-1}
\end{align*}
so $\o{n}$ is precisely the set
\begin{align*}
    \o{n} = \{ X \in \gl{n}{\rn} \ | \ X^t = -X \}.
\end{align*}
Similar computations show
\begin{align*}
    \u{n} = \{ Z \in \gl{n}{\cn} \ | \ \bar{Z}^t = - Z\}.
\end{align*}
The groups $\SO{n}$ and $\SU{n}$ are clearly also compact Lie groups, being the intersections of a compact Lie group with a closed Lie group. Their Lie algebras are 
\begin{align*}
    \so{n} = (\o{n} \cap \sl{n}{\rn}) = \{ X \in \o{n} \ | \ \tr(X) = 0\}
\end{align*}
and 
\begin{align*}
    \su{n} = (\u{n} \cap \sl{n}{\cn}) = \{ Z \in \u{n} \ | \ \tr(Z) = 0\}.
\end{align*}
The condition $X^t = - X$ is already enough to guarantee $\tr(X) = 0$ and hence $\o{n} = \so{n}$. 
\end{proof}
It takes bearing in mind that although $\u{n}$ and $\su{n}$ are subalgebras of $\gl{n}{\cn}$, they are real and not complex Lie algebras. In fact one can show that their complexifications are $\u{n} \oplus i\u{n} = \gl{n}{\cn}$ and $\su{n} \oplus i\su{n} = \sl{n}{\cn}$.
\begin{definition}\label{symplectic group definition}
Let $\GL{n,\mathbb{H}}$ denote the group of invertible $n\times n$ quaternionic matrices, and for $q \in \GL{n,\mathbb{H}}$ define $q^* = \bar{q}^t$, where $\bar{q}$ is the entry-wise quaternionic conjugate of $q$. Then the symplectic group $\Sp{n}$ is the set
$$
\Sp{n} =\{ q \in \GL{n, \mathbb{H}} \ | \ q^* = q^{-1} \}.
$$
\end{definition}

\begin{proposition}\label{Sp(n), GL(H) Lie groups}
The group $\GL{n,\mathbb{H}}$ is a Lie group, and $\Sp{n}$ a compact Lie subgroup of $\GL{n,\mathbb{H}}$.
\end{proposition}
\begin{proof}
Since the quaternions are isomorphic to $\rn^4$ as a real vector space, the set $\mathbb{H}^{n \times n}$ of $n$ by $n$ quaternionic matrices is isomorphic to $\rn^{4n^2}$ as a real vector space, and we may endow it with the standard topology on $\rn^{4n^2}$. We may write a quaternionic matrix $q \in \mathbb{H}^{n\times n}$ as a sum $q = z + jw$, where $z$ and $w$ are $n\times n$ complex matricies. Consider then the map $\varphi: \mathbb{H}^{n\times n} \to \cn^{2n\times 2n}$ given by
\begin{align*}
    \varphi: z +jw \mapsto
    \begin{pmatrix}
    z & w\\
    -\bar{w} & \bar{z}
    \end{pmatrix}.
\end{align*}
It is clear that the map $\varphi$ is smooth and $\rn$-linear. The image of $\varphi$ is the set of block matrices
$$
\text{Im}(\varphi) = \left\{
\begin{pmatrix}
z & w\\
-\bar{w} & \bar{z}
\end{pmatrix} \in \cn^{2n \times 2n} \  \middle | \ z, w \in \cn^{n\times n} 
\right\}
$$
It will be important that $\text{Im}(\varphi)$ is closed under taking inverses whenever they exist. To see this, define
$$
J = \varphi(jI_n) = 
\begin{pmatrix}
0 & I_n\\
-I_n & 0
\end{pmatrix}.
$$
Then one checks by direct computation that for $g \in \cn^{2n \times 2n}$, $g \in \text{Im}(\varphi)$ if and only if $g$ satisfies
$$
Jg = \bar{g}J.
$$
We have $J^{-1} = -J$ so that if $g \in \text{Im}(\varphi)$ is invertible then
\begin{align*}
    Jg^{-1} = (-gJ)^{-1} = (-J\bar{g})^{-1} = \bar{g}^{-1}J
\end{align*}
and hence $g^{-1} \in \text{Im}(\varphi)$.
We now claim that for all $p$, $q \in \mathbb{H}^{n\times n}$, $\varphi(pq) = \varphi(p)\varphi(q)$. Indeed, for $p =z_1 + jw_1$, $q = z_2 + jw_2 \in \mathbb{H}^{n\times n}$
\begin{align*}
    pq &= (z_1 + jw_1)(z_2 + jw_2)\\
    &= z_1z_2 + jw_1z_2 +j\bar{z_1}w_2 - \bar{w_1}\bar{w_2}\\
    &= z_1z_2 -\bar{w_1}\bar{w_2} + j(w_1z_2 + \bar{z_2}w_2),
\end{align*}
and 
\begin{align*}
    \varphi(p)\varphi(q) &= \begin{pmatrix}
    z_1 & w_1\\
    -\bar{w}_1 & \bar{z_1}
    \end{pmatrix}
    \begin{pmatrix}
    z_2 & w_2\\
    -\bar{w}_2 & \bar{z_2}
    \end{pmatrix}\\
    &= \begin{pmatrix}
    z_1z_2 -w_1\bar{w}_2 & z_1w_2 +w_1\bar{z}_2\\
    -\overline{z_1w_2 +w_1\bar{z}_2} & \overline{z_1z_2 -w_1\bar{w}_2}
    \end{pmatrix}.
\end{align*}
so that 
$$
\varphi(pq) = \varphi(p)\varphi(q).
$$
Since $\varphi(e)$ is the identity matrix in $\cn^{2n\times 2n}$, $\varphi$ is bijective and $\text{Im}(\varphi)$ is closed under inversion, an element $p \in \mathbb{H}^{n\times n}$ is invertible if and only if $\varphi(p)$ is invertible. The set of such elements is an open subset of $\text{Im}(\varphi)$ and hence, by continuity, $\GL{n,\mathbb{H}}$ is open in $\mathbb{H}^{n\times n}$, and thus an open submanifold of $\mathbb{H}^{n\times n}$. Since multiplication and inversion are smooth in $\cn^{2n \times 2n}$ they are smooth in $\GL{n,\mathbb{H}}$ by smoothness of $\varphi$, making $\GL{n}$ a Lie group.
\par
For $q = z +jw \in \GL{n,\mathbb{H}}$ we have $\bar{q} = \bar{z} - jw$ and $q^* = \bar{z}^t -jw^t$. If $q \in \Sp{n}$ then $q^{-1}=q^* \in \Sp{n}$ since $(q^*)^* = q$. One shows by a simple direct computation that $(pq)^* = q^*p^*$ holds for quaternionic matrices so for $p, q \in \Sp{n}$
$$
(pq)^* = q^*p^* = q^{-1}p^{-1} = (pq)^{-1}.
$$
Hence $\Sp{n}$ is a subgroup of $\GL{n,\mathbb{H}}$. Furthermore, we have
\begin{align*}
    \varphi((z+jw)^*) &= \begin{pmatrix}
    \bar{z}^t& -w^t\\
    \bar{w}^t & z^t
    \end{pmatrix}\\ 
    &= \overline{\begin{pmatrix}
    z & w\\
    -\bar{w} & \bar{z}
    \end{pmatrix}}^t\\
    &= \overline{(\varphi(z + jw))}^t
\end{align*}
so that $\varphi(\Sp{n}) = (\text{Im}(\varphi)\cap \U{2n})$. But this is a closed subgroup of $\varphi(\GL{n,\mathbb{H}})$ and hence by continuity $\Sp{n}$ is closed. By the closed subgroup Theorem \ref{closed subgroup} $\Sp{n}$ is a Lie subgroup of $\GL{n,\mathbb{H}}$. Since $\Sp{n}$ is closed and contained in $\U{2n}$, it is also compact. 
\end{proof}

\begin{remark}\label{Complex repr}
The map $\varphi: \GL{n,\mathbb{H}} \to \GL{2n,\cn}$ defined by
$$
\varphi: q = z + jw \mapsto 
\begin{pmatrix}
z & {w}\\
-\bar{w} & \bar{z}
\end{pmatrix} 
$$
appearing in the proof of Proposition \ref{Sp(n), GL(H) Lie groups} is an injective Lie group homomorphism and hence an isomorphism onto its image. We shall henceforth identify the groups $\GL{n,\mathbb{H}}$ and $\Sp{n}$ with their images under $\varphi$. With this identification in mind, it is clear that we may represent the Lie algebra $\gl{n}{\mathbb{H}}$ using complex matrices as
$$
\gl{n}{\mathbb{H}} = \left\{
\begin{pmatrix}
Z & W\\
-\bar{W} & \bar{Z}
\end{pmatrix} \in \gl{2n}{\cn} \  \middle | \ Z, W \in \gl{n}{\cn} 
\right\}.
$$
\end{remark}
\begin{proposition}
In view of Remark \ref{Complex repr}, $\Sp{n}$ is contained in $\SU{2n}$, with Lie algebra
\begin{align*}
    \sp{n} = \left\{
\begin{pmatrix}
Z & W\\
-\bar{W} & \bar{Z}
\end{pmatrix} \in \gl{2n}{\cn}\ \middle| \ \bar{Z}^t+Z = W^t - W = 0 \right\} 
\end{align*}
\end{proposition}

\begin{proof}
We have already seen in the proof of Proposition \ref{Sp(n), GL(H) Lie groups} that $\Sp{n}$ is contained in $\U{2n}$. To see that $\Sp{n}$ is contained in $\SU{2n}$ we shall show that $\sp{n}$ is contained in $\su{2n}$. But this holds since for 
$$
Q = 
\begin{pmatrix}
Z & W\\
-\bar{W} & \bar{Z}
\end{pmatrix} \in \sp{n} \cong \gl{n}{\mathbb{H}}\cap \u{2n},
$$
we have $Q+Q^* = 0$ so that in particular $\bar{Z}^t + Z = 0$, which gives 
$$
\tr(Q) = \tr(Z) + \tr(\bar{Z}) = \tr(Z + \bar{Z}^t) = 0.
$$
Hence $Q \in \su{2n}$ as desired. The expression for the Lie algebra also follows directly from the condition $Q^* + Q = 0$.
\end{proof}

We now show that the standard left-invariant Riemannian metric on $\GL{n,\cn}$ induces bi-invariant metrics on the compact Lie groups $\O{n},\SO{n}, \U{n}, \SU{n}$ and $\Sp{n}$.
\begin{proposition}\label{metrics on the classical groups}
Denote by $g$ the Riemannian metrics
\begin{align*}
    g_p(pX,pY) = \tr(X^t \cdot Y)
\end{align*}
on $\O{n}$ and $\SO{n}$, and by $h$ the Riemannian metrics
\begin{align*}
    h_p(pX,pY) = \re(\tr(\bar{X}^t \cdot Y))
\end{align*}
on $\U{n}, \SU{n}$ and $\Sp{n}$. Then in each case $g$ and $h$ are bi-invariant metrics.
\end{proposition}
\begin{proof}
Since a bi-invariant metric on a Lie group induces bi-invariant metrics on each subgroup, it is sufficient to consider the case $\U{n}$ as it contains the other Lie groups as subgroups. As $h$ is obtained by left translating an inner product on $\GL{n,\cn}$, we need only show that the restriction of this inner product to $\u{n}$ is $\Ad{}$-invariant. Let $X, Y \in \u{n}$ and $p \in \U{n}$. Then
\begin{align*}
    h_e(\Ad{p}(X),\Ad{p}(Y)) &= \re(\tr(\overline{\Ad{p}(X)}^t \cdot \Ad{p}(Y)))\\
    &= \re(\tr(\overline{(pX\bar{p}^t)}^t \cdot pY\bar{p}^t))\\
    &= \re(\tr(p\bar{X}^t \bar{p}^t \cdot pY\bar{p}^t))\\
    &= \re(\tr(p\bar{X}^t \cdot Y \bar{p}^t))\\
    &= \re(\tr(\bar{X}^t \cdot Y))\\
    &= h_e (X,Y).
\end{align*}
Hence the metrics on $\U{n}, \SU{n}$ and $\Sp{n}$ are all bi-invariant.
Since for real matrices $X$ and $Y$ we have
\begin{align*}
    \re(\tr(\bar{X}^t \cdot Y)) = \tr(X^t\cdot Y)
\end{align*}
we have also shown the metrics on $\O{n}$ and $\SO{n}$ are bi-invariant.
\end{proof}
\par
To conclude the chapter, we shall show that the matrix Lie algebras $\sl{n}{\rn},\sl{n}{\cn},$ $ \so{n}, \su{n}$ and $\sp{n}$ are semisimple, while $\gl{n}{\rn}, \gl{n}{\cn}$ and $\u{n}$ are not, by computing their Killing forms.
\begin{theorem}
The Killing forms of the classical Lie groups covered in this chapter are
\begin{enumerate}[label = (\roman*)]
    \item $B_{\gl{n}{\rn}}(X,Y) = 2n \cdot \tr(X\cdot Y) - 2\cdot \tr(X) \cdot \tr(Y)$,
    \item $B_{\gl{n}{\cn}}(Z,W) = 2n \cdot \tr(Z\cdot W) - 2\cdot \tr(Z) \cdot\tr(W)$,
    \item $B_{\sl{n}{\rn}}(X,Y) = 2n \cdot \tr(X \cdot Y)$,
    \item $B_{\sl{n}{\cn}}(Z,W) = 2n \cdot \tr(Z \cdot W)$,
    \item $B_{\so{n}}(X,Y) = (n-2) \cdot \tr(X \cdot Y)$,
    \item $B_{\u{n}}(Z,W) = 2n \cdot \tr(Z\cdot W) - 2 \cdot \tr(Z)\cdot \tr(W)$,
    \item $B_{\su{n}}(Z,W) = 2n \cdot \tr(Z\cdot W)$,
    \item $B_{\sp{n}}(Z,W) = 2n \cdot \tr(Z\cdot W)$.
\end{enumerate}
\end{theorem}
The calculations for the Killing forms are rather lengthy and thus they are performed in Appendix A.
\begin{corollary}\label{semisimple list}
The Lie algebras $\so{n},\su{n}$ and $\sp{n}$ are compact and semisimple.
\end{corollary}
We also note that the inner products on the Lie algebras $\so{n},\su{n}$ and $\sp{n}$ from Proposition \ref{metrics on the classical groups} are all of the form
\begin{align*}
    \ip{X}{Y} = -\alpha\cdot B(X,Y)
\end{align*}
where $\alpha>0$ is a real constant depending on the Lie algebra, and $B$ is the Killing form. Hence, most statements about the Killing form on these Lie algebras may be translated into equivalent statements in terms of their standard bi-invariant metrics.
\chapter{Symmetric Spaces}\label{ch: symmetric spaces}

In this chapter we will introduce the notion of a Riemannian symmetric space, which is a Riemannian manifold such that for every point there is an isometry reversing geodesics through that point. We shall show that every Riemannian symmetric space is also a Riemannian homogeneous space. This will let us describe the symmetric space $M$ as a quotient of two Lie groups $G/K$, with $G$ being the component of the isometry group of $M$ containing the identity, and  $K$ the isotropy group of a fixed point $p\in M$.
\par
With this we will be able to use Lie group and Lie algebraic notions to characterise the symmetric spaces. In particular we will see that there is a one-to-one correspondence between Riemannian symmetric spaces and so-called symmetric triples. These are triples $(G,K,\sigma)$ where $G$ is a Lie group, $\sigma$ an invloutive automorphism of $G$, and $K$ a compact subgroup sitting between the fix-point set $G^{\sigma}$ of $\sigma$ and its identity component $G^{\sigma}_0$. In all the cases treated in this thesis we will be able to take $K = G^{\sigma}$, with $G^{\sigma}$ the fix-point set  of $\sigma$. On the level of Lie algebras this gives a decomposition
\begin{align*}
    \la{g} = \la{k} \oplus \la{p}
\end{align*}
of the Lie algebra $\la{g}$ of $G$ into the eigenspaces of the eigenvalues $\pm 1$ of the differential $d\sigma$ of the automorphism $\sigma$.
\par
We then define what is meant by an irreducible symmetric space. We will then describe how these can be divided into the three types, compact, non-compact and flat, and introduce the important duality between the symmetric spaces of compact and non-compact type. We conclude the chapter by mentioning the famous classification result of irreducible symmetric spaces due to \`Elie Cartan. 
\par
Throughout this chapter and the remainder of the text we shall assume all symmetric spaces to be connected unless otherwise explicitly stated.
\begin{remark}\label{notation for fixed sets and identity components}
Throughout the remainder of this thesis we will use the notation $G^\sigma$ for the fix-point set of an automorphism $\sigma$ of the Lie group $G$. We will also use the notation $G_0$ to mean the identity component of $G$. In particular we write $I(M)$ and $I_0(M)$ for the isometry group of a Riemannian manifold $(M,g)$ and its identity component, respectively.
\end{remark}
\begin{definition}\label{def symmetric space}
A \emph{Riemannian symmetric space} is a Riemannian manifold $(M,g)$ such that for each point $p\in M$ there is a global isometry $s_{p}$ of $M$, called the symmetry at $p$, such that 
\begin{enumerate}[label = (\roman*)]
    \item $s_p$ fixes $p$,
    \item $(ds_{p})_p = -id_{T_p M}$.
\end{enumerate}
We call a Riemannian metric $g$ on a smooth manifold $M$ such that $(M,g)$ is a Riemannian symmetric space a \emph{symmetric metric} on $M$.
\end{definition}
From now on, the \emph{Riemannian} qualifier in \emph{Riemannian symmetric space} will be dropped, and we will simply refer to these objects as \emph{symmetric spaces}. 
We shall show that the symmetry at a point of a symmetric space is unique. For this we need the following well-known result.
\begin{proposition}\label{isometry uniqueness thm}
Let $(M,g)$ be a connected Riemannian manifold and $\varphi, \psi$ be isometries of $M$. If there is a point $p\in M$ such that 
$$
    \varphi(p) = \psi(p) \quad \text{and} \quad d\varphi_p = d\psi_p,
$$
then 
$$
    \varphi = \psi.
$$
\end{proposition}
\begin{proof}
As we will only need this result in cases when $M$ is known to be complete, we will present the proof only under this special assumption. Suppose that $p$ satisfies $\varphi(p) =\psi(p)$ and $d\varphi_p = d\psi_p$ and let $q$ be another point of $M$. Since $M$ is complete there is a geodesic segment $\gamma:[0,1] \to M$ such that $\gamma(0) = p$ and $\gamma(1) = q$. Since $\varphi$ and $\psi$ are isometries mapping $p$ to the same point, the compositions
\begin{align*}
    \alpha = \varphi \circ \gamma,\\
    \beta = \psi \circ \gamma,
\end{align*}
are both geodesic segments starting at $\varphi(p)$. Since $d\varphi_p = d\psi_p$ we have $\dot{\alpha}(0) = \dot{\beta}(0)$. But then by the uniqueness of geodesics we are done, since then $\phi(q) = \alpha(1) = \beta(1) = \psi(q)$ and $q$ was arbitrary.
\end{proof}
This means that isometries of a Riemannian manifold are completely characterised by their value and differential at a single point. As a consequence we get the following.
\begin{corollary}
Let $(M,g)$ be a symmetric space and let $p \in M$. Then there is precisely one isometry $s_{p}$ satisfying (i) and (ii) of Definition \ref{def symmetric space}.
\end{corollary}
We remark that if $\gamma$ is a geodesic in a symmetric space $M$, with $\gamma(0) = p$, then $\alpha =s_{p}\circ \gamma$ is also a geodesic in $M$ through $p$ and we have $\dot{\alpha}(0) = -\dot{\gamma}(0)$. Another geodesic through $p$ with tangent vector $-\dot{\gamma}(0)$ at $p$ is given by
\begin{align*}
    t\mapsto \gamma(-t).
\end{align*}
By the uniqueness of geodesics these curves must then be the same, so that $\alpha(t) = \gamma(-t)$ for values of $t$ in some interval $(-\epsilon,\epsilon)$.
Hence, in essence, a symmetric space is then a Riemannian manifold such that for each point there is an isometry reversing geodesics through that point. 
\par
This geodesic-reversing property of the symmetries endows symmetric spaces with many important and useful properties, much of which we will encounter and make good use of in this thesis. One such  property is geodesical completeness.
\begin{proposition}
Let $(M,g)$ be a symmetric space. Then $(M,g)$ is a geodesically complete Riemannian manifold.
\end{proposition}
\begin{proof}
Recall that a Riemannian manifold is geodesically complete, or simply complete, if for each point $p \in M$ and tangent vector $X_p \in T_p M$, there is a geodesic $\gamma: \rn \to M$ defined on all of $\rn$ such that $\gamma(0) = p$ and $\dot{\gamma}(0) = X_p$. Assume by way of contradiction that $M$ is not complete. Then there exists some point $p\in M$ and tangent vector $X_p \in T_p M$ such that the geodesic segment $\gamma: [0,a) \to M$, $a < \infty$, with $\gamma(0) = p$ and $\dot{\gamma}(0)= X_p$ cannot be extended to an interval $[0,b)$ with $b > a$, i.e it is maximally extended in the positive direction. Now pick $\epsilon > 0$ and set $q = \gamma(a-\epsilon)$. Define a new curve $\tilde{\gamma}:[0,2a-2\epsilon) \to M$ by
\begin{align*}
    \tilde{\gamma}(t) = 
    \begin{cases}
        \gamma(t) & \text{for } t\in [0,a-\epsilon)\\
        s_q(\gamma(2a-2\epsilon - t)) & \text{for } t \in [a-\epsilon, 2a - 2\epsilon)
    \end{cases}
\end{align*}
Clearly $\tilde{\gamma}$ consists of two geodesic segments. Furthermore, we have $\tilde{\gamma}(a-\epsilon) = \gamma(a-\epsilon)$ and
$$
\dot{\tilde{\gamma}}(a-\epsilon) = (ds_q)_q(-\dot{\gamma}(a-\epsilon)) = \dot{\gamma}(a-\epsilon).
$$
Thus $\tilde{\gamma}$ is an extension of $\gamma$ to the interval $[0,2a-2\epsilon)$. Since $\epsilon$ was arbitrary, we may take $\epsilon$ small enough so that $2a-2\epsilon > a$. But this contradicts maximality and hence $M$ is complete. 
\end{proof}

The next property that we will derive for symmetric spaces concerns their totally geodesic submanifolds. A lose statement of this result would be that a submanifold of a symmetric space $M$ is totally geodesic if and only if it inherits the symmetric structure from $M$. This will be important later when we wish to show that the image of the Cartan embedding is a totally geodesic submanifold. For this we first need the following well known result.
\begin{lemma}\label{2nd fund form inv}
Let $(M,g)$ be a Riemannian manifold and  $N$ be a submanifold of $M$. Let $\varphi: M \to M$ be an isometry such that $\varphi(N) = N$. Then, denoting the restrictions of $\varphi$ and $d\varphi$ to $N$ by the same letter, the second fundamental form $B$ of the submanifold $N$ satisfies
\begin{align*}
    d \varphi(B(X,Y)) = B(d \varphi(X),d \varphi(Y)).
\end{align*}
For all vector fields $X, Y \in \smooth{TN}$. In particular, if $\varphi(p) = p$ for some $p \in N$, we have
\begin{align}\label{point invariance}
    d \varphi_p(B_p(X_p,Y_p)) = B_p(d \varphi_p(X_p), d \varphi_p(Y_p))
\end{align}
for all tangent vectors $X_p, Y_p \in T_p N$.
\end{lemma}
\begin{proof}
Let $\tnabla$ and $\nabla$ be the Levi-Civita connections of $M$ and $N$, respectively. Let $p$ be a point of $N$ and $X,Y \in \smooth{TN}$. Let $U$ be an open neighbourhood in $M$ of $p$ such that the vector fields $X$ and $Y$ admit local extensions $\widetilde{X}$ and $\widetilde{Y}\in \smooth{TU}$ onto $U$. For the existence of such extensions see 6.16. in \cite{Sigmundur's notes}. As $\varphi$ restricts to an isometry of $N$, the vector fields $d\varphi(\widetilde{X}), d\varphi(\widetilde{Y})\in \smooth{T\varphi(U)}$ are local extensions of $d\varphi(X), d\varphi(Y) \in \smooth{T\varphi(U \cap M)}$. We shall show that they locally satisfy the following relations:
\begin{align}\label{Eq: 2nd fund form lemma eq1}
    d \varphi(\tnab{\wt{X}}{\wt{Y}}) = \tnab{d \varphi(\wt{X})}{d \varphi (\wt{Y})}
\end{align}
and
\begin{align}\label{Eq: 2nd fund form lemma eq2}
    d \varphi(\nab{}{X}{Y}) = \nab{}{d \varphi(X)}{d \varphi (Y)}.
\end{align}
This, together with the fact that $d \varphi$ preserves the normal and tangent bundles of $N$ will be enough to establish the desired result as then
\begin{align*}
    d \varphi(B(X,Y)) &= d \varphi(\tnab{\wt{X}}{\wt{Y}} - \nab{}{X}{Y})\\
    &= \tnab{d \varphi(\wt{X})}{d \varphi (\wt{Y})} - \nab{}{d \varphi(X)}{d \varphi (Y)}\\
    &= B(d \varphi(X), d \varphi(Y))
\end{align*}
holds for $X, Y \in \smooth{T(U \cap N)}$ and $N$ can be covered by neighbourhoods of this form. To establish Equation \eqref{Eq: 2nd fund form lemma eq1} let $p$ and $U$ be as above. Since $\varphi$ is an isometry, for vector fields $\wt{X}, \wt{Y}$ and $\wt{Z} \in \smooth{TU}$ we have
\begin{align*}
    g(d \varphi(\tnab{\wt{X}}{\wt{Y}}),d \varphi(\wt{Z}))(\varphi(p)) = g(\tnab{\wt{X}}{\wt{Y}}, \wt{Z})(p). 
\end{align*}
Using the Koszul formula we get
\begin{align}\label{Koszul}
\begin{split}
    &\quad g(\tnab{d \varphi(\wt{X})}{d \varphi(\wt{Y})}, d \varphi(\wt{Z}))(\varphi(p))\\ 
    & = \frac{1}{2}\big\{ d\varphi_p(\wt{X}_p)(g(d\varphi(\wt{Y}),d \varphi(\wt{Z}))) 
     + d\varphi_p(\wt{Y}_p)(g(d\varphi(\wt{Z}),d \varphi(\wt{X})))\\
    &\quad - d\varphi_p(\wt{Z}_p)(g(d\varphi(\wt{X}),d \varphi(\wt{Y})))
     + g_{\varphi(p)}(d \varphi_p(\wt{Z}_p), [d \varphi(\wt{X}),d \varphi(\wt{Y})]_{\varphi(p)})\\
    &\quad + g_{\varphi(p)}(d \varphi_p(\wt{X}_p), [d \varphi(\wt{Y}),d \varphi(\wt{Z})]_{\varphi(p)})\\
     &\quad - g_{\varphi(p)}(d \varphi_p(\wt{Y}_p), [d \varphi(\wt{Z}),d \varphi(\wt{X})]_{\varphi(p)}) \big \}.
\end{split}
\end{align}
For the first three terms of equation \eqref{Koszul} we have 
\begin{align*}
    d \varphi_p(\wt{X}_p)(g(d\varphi(\wt{Y}),d\varphi(\wt{Z}))) &= \wt{X}_p(g(d\varphi(\wt{Y}),d\varphi(\wt{Z})) \circ \varphi)\\
    &= \wt{X}_p(g(\wt{Y},\wt{Z})).
\end{align*}
For the remaining terms we get
\begin{align*}
    g_{\varphi(p)}(d \varphi_p(\wt{Y}_p),[d \varphi(\wt{Z}),d \varphi(\wt{X})]_{\varphi(p)}) &= g_{\varphi(p)}(d\varphi_p(\wt{Y}_p),d\varphi_p([\wt{Z},\wt{X}]_p))\\ &= g_p(\wt{Y}_p,[\wt{Z},\wt{X}]_p).
\end{align*}
Hence, in particular, for all $p \in N$ and all $\wt{Z}_p \in T_p M$,
\begin{align*}
    g(d \varphi(\tnab{\wt{X}}{\wt{Y}}),d \varphi(\wt{Z}))(\varphi(p)) = g(\tnab{d \varphi(\wt{X})}{d \varphi(\wt{Y})}, d \varphi(\wt{Z}))(\varphi(p)).
\end{align*}

Since both $\varphi$ and $d \varphi$ are bijections this is enough to show $\tnabla$ satisfies equation \eqref{Eq: 2nd fund form lemma eq1}, as we may choose $\wt{X}$ and $\wt{Y}$ to be the local extensions of $X$ and $Y$ in $\smooth{T(U\cap N)}$. For equation \eqref{Eq: 2nd fund form lemma eq2} we simply note that $\varphi$ is also an isometry of $N$ so that the preceding arguments apply with $\nabla$ in place of $\tnabla$ and $X, Y$ and $Z$ in $\smooth{T(U\cap N)}$ in place of $\wt{X}, \wt{Y}$ and $\wt{Z}$. This concludes the proof.
\end{proof}
\begin{proposition}[\cite{Kobayashi-Nomizu,Ziller's notes}]\label{symmetry invariance}
Let $(M,g)$ be a symmetric space and $N$ be a submanifold of $M$. If for all points $p \in N$, $N$ is closed under the symmetry $s_p$ of $M$ at $p$, i.e. $s_p(N) = N$, then $N$ is a totally geodesic submanifold of $M$. Conversely, if $N$ is a totally geodesic and complete submanifold of $M$, then $s_p(N) = N$ for all $p \in N$.
\end{proposition}
\begin{proof}
To show that $N$ is a totally geodesic submanifold of $M$ is equivalent to showing that the second fundamental form $B$ of $N$ vanishes. Let $p$ be a point of $N$. By Lemma \ref{2nd fund form inv} we have
\begin{align}\label{Eq: sym inv eq1}
    d (s_p)_p B_p(X_p,Y_p) = B_p(d(s_p)_p(X_p),d(s_p)_p(Y_p)),
\end{align}
for all $X_p, Y_p \in T_p N$. But since $d(s_p)_p = -id_{T_p M}$, equation \eqref{Eq: sym inv eq1} reads 
\begin{align*}
    -B_p(X_p,Y_p) = B_p(X_p,Y_p),
\end{align*}
so $B$ vanishes at $p$. Since this holds for all points $p$ of $N$, the second fundamental form vanishes identically and thus, $N$ is a totally geodesic submanifold.
\par
For the converse, suppose by way of contradiction that $N$ is a totally geodesic and complete submanifold of $M$, but that $s_p(N) \neq N$ for some $p \in N$. Let $q$ be a point of $N$ such that $s_{p}(q) \notin N$. By completeness there exists $\gamma_{pq}: [0,1] \to M$ with $\gamma_{pq}(0) = p$ and $\gamma_{pq}(1) = q$ which is also a geodesic in $N$. Since $s_p$ is an isometry of $M$ fixing $p$, $\alpha = s_p(\gamma_{pq})$ is a geodesic from $p$ to $s_p(q)$. However, the differential $(ds_p)_p = -id_{T_p M}$ preserves the tangent space $T_p N$, so that $\dot{\alpha}(0) \in T_p N$. Then, by uniqueness of geodesics and since $N$ is totally geodesic, the curve $\alpha$ is contained in $N$. This contradicts $\alpha(1)=s_p(q)\notin N$ and we are done.
\end{proof}
\begin{remark}
It is clear that the symmetry invariant submanifolds $N$ in Proposition \ref{symmetry invariance} are themselves symmetric spaces, with the symmetry at a point $p\in N$ being simply the restriction of the symmetry of $M$ at $p$ to $N$.
\end{remark}
We have in fact already in this thesis encountered several examples of symmetric spaces, as the following result will make clear.
\begin{proposition}\label{Lie group symmetric}

Let $(G,g)$ be a compact Lie group equipped with a  bi-invariant Riemannian metric. Then $G$ is a symmetric space and the symmetry $s_p$ at $p \in G$ is the map
\begin{align*}
    s_p:q \mapsto p\cdot q^{-1}\cdot p.
\end{align*}
\end{proposition}
\begin{proof}
Clearly we have $s_p(p) = pp^{-1}p = p$ for each $p \in G$. We can decompose the map $s_p$ as
\begin{align*}
    s_p = L_p \circ \iota \circ L_{p^{-1}}.
\end{align*}
Then for $X_p \in T_p G$ we have
\begin{align*}
    d(s_p)_p(X_p) &= (dL_p)_e(d\iota)_e(dL_{p^{-1}})_p(X_p)\\
    &= -(dL_p)_e(dL_{p^-1})_p(X_p)\\
    &= -X_p
\end{align*}
since $d\iota_e = -id_{T_e M}$ and $L_p$ and $L_{p^{-1}}$ are inverses of each other. Furthermore, the inversion map $\iota$ is an isometry by Corollary \ref{inversion isometry}, since $G$ is compact. 
\end{proof}
We will later see a somewhat different way to conceive of the Lie group $G$ as a symmetric space.
\par
In order to use Lie theory to understand symmetric spaces we will also recall the following definition of a Riemannian homogeneous space.
\begin{definition}\label{def homogeneuos space}
A \emph{Riemannian homogeneous} space is a Riemannian manifold $(M,g)$ such that the isometry group $I(M)$ acts transitively on $M$. In other words, for each pair of points $p,q \in M$ there is an isometry $\varphi$ of $M$ such that
\begin{align*}
    \varphi(p) = q.
\end{align*}
\end{definition}
Another consequence of the geodesic-reversing property of the symmetries is then the following.
\begin{proposition}
Every symmetric space $(M,g)$ is a Riemannian homogeneous space.
\end{proposition}
\begin{proof}
Let $p, q \in M$. Since $M$ is complete there is a geodesic $\gamma: \rn \to M$ such that $\gamma(0) = p$ and $\gamma(1) = q$. Set $\varphi = s_{\gamma\left(\tfrac{1}{2}\right)}$. Then, since $s_{\gamma\left(\tfrac{1}{2}\right)}$ reverses geodesics through $\gamma(\tfrac{1}{2})$, in particular it reverses $\gamma$ so that
\begin{align*}
    \varphi(p) = \varphi(\gamma(\tfrac{1}{2}-\tfrac{1}{2})) =\gamma(\tfrac{1}{2}+\tfrac{1}{2}) = q.
\end{align*}
This then concludes the proof.
\end{proof}
As indicated we prefer to work only with connected Lie groups and manifolds in general. However, there is no guarantee that the full isometry group $I(M)$ of a Riemannian homogeneous space $M$ is connected. It will then be most convenient to work instead with the identity component $I_{0}(M)$. The following lemma, which is a special case of Proposition 4.3 of Chapter II in \cite{Helgason}, shows that this will not pose any problems.

\begin{lemma}[\cite{Helgason}]\label{identity componen is transitive}
Let $(M,g)$  be a Riemannian manifold with isometry group $I(M)$. If $I(M)$ acts transitively on $M$, then so does its identity component $I_{0}(M)$.
\end{lemma}
For a proof see \cite{Helgason}. We now wish to describe Riemannian homogeneous spaces, and by extension symmetric spaces, in terms of Lie groups. 

We let $(M,g)$ be a Riemannian homogeneous space, and denote by $G$ the connected component of the isometry group of $M$. By Proposition \ref{Lie group symmetric}, $G$ has the structure of a Lie group. We now fix a point $p_0 \in M$ and consider the subgroup
\begin{align*}
    H = G_{p_0} = \{ g \in G \ | \ g \cdot p_0 = p_0\}
\end{align*}
of elements fixing the point $p_0$. In other words $H$ is the isotropy group of $p_0$. The choice of basepoint $p_0$ does not matter as the isotropy groups are all conjugate to each other. One now easily shows the map $\varphi: gH \mapsto g\cdot p_0$ is a bijection between the coset space $G/H$ and $M$. The Lie group $G$ acts on the coset space $G/H$ by left translation $L_p: gH \mapsto (pg)H$ and the bijection $\varphi$ commutes with this action. So far $G/H$ is simply a set with an action of $G$ defined on it. However with the help of the following powerful result we are able to identify $G/H$ with $M$ also as smooth manifolds.
\begin{theorem}[\cite{Helgason,Kobayashi-Nomizu}]\label{coset space smooth}
Let $G$ be a Lie group and $H$ be a closed subgroup of $G$. Then there is a unique smooth structure on the coset space $G/H$ such that all left translations $L_p: G/H \to G/H$, $p\in G$ are diffeomorphisms of $G/H$ and the natural projection $\pi:g \mapsto gH$ is a submersion.
\end{theorem}
In our particular case when $G$ is the identity component of the isometry group of a Riemannian homogeneous space, and $H$ the isotropy group of a fixed point, we then have.
\begin{theorem}[\cite{Helgason,Kobayashi-Nomizu}]\label{homogenous coset equiv}
Let $(M,g)$ be a Riemannian homogeneous space in which we fix a point $p_0$. Let $G$ be the connected component of the isometry group of $M$ and let $H = G_{p_0}\subset G$ be the isotropy group of $p_0$. Then there is a unique $C^{\infty}$-structure on the coset space $G/H$ such that the natural projection $\pi: G \mapsto G/H$ is a submersion. Furthermore, with respect to this structure, the bijection $gH \mapsto g\cdot p_0$ becomes a diffeomorphism.
\end{theorem}

The isotropy group $H$ of a point $p_0$ of a Riemannian homogeneous space $M = G/H$ can be identified with a subgroup of the orthogonal group of the tangent space at that point in the following way.

\begin{proposition}\label{isotropy rep}
Let $M = G/H$ be a Riemannian homogeneous space with $H = G_{p_0}$. Then the map $\varphi: H \to \O{T_{p_0}M}$
\begin{align*}
    \varphi: h \mapsto dh_{p_0}
\end{align*}
is an injective Lie group homomorphism, generally refered to as the \emph{isotropy representation} of $H$ on $T_{p_0} M$.
\end{proposition}
\begin{proof}
Smoothness is immediate since the action of $G$, and hence $H$ on $M$ is smooth. It then follows from the chain rule that $\varphi$ is a Lie group homomorphism. We note that the image of $\varphi$ must lie in $\O{T_{p_0} M}$, since $H$ acts by isometry on $M$. What remains to be checked is the injectivity. This follows from Proposition \ref{isometry uniqueness thm} since for any two elements $h_1,h_2 \in H$, they by definition both map $p_0$ to itself, so they are the same if and only if they have the same differential at $p_0$. But this is then precisely the statement $\varphi(h_1) = \varphi(h_2)$ if and only if $h_1 = h_2$. Hence, $\varphi$ is injective.
\end{proof}
This has the immediate consequence.
\begin{corollary}
Let $M = G/H$ be a Riemannian homogeneous space. Then $H$ is a compact subgroup of $G$.
\end{corollary}
As mentioned at the beginning of this chapter, we aim to characterise the symmetric spaces purely in terms of Lie-theoretic notions. The first half of the result which allows for this approach is the following.
\begin{proposition}[\cite{Helgason, Ziller's notes}]\label{symmetric spaces give symmetric pairs}
Let $(M,g)$ be a symmetric space. Fix a point $p_0\in M$ and set $G = I_{0}(M)$ and $K = G_{p_0}$. Then there exists an involutive automorphism $\sigma$ of the Lie algebra $G$ such that $G^{\sigma}$ is compact and $G^\sigma_0 \subset K \subset G^{\sigma}$.
\end{proposition}
\begin{proof}
We define $\sigma(p) = s_{p_0}\cdot p \cdot s_{p_0}$ for all $p \in G$. Since $s_{p_0}$ is an isomorphism and $G$ the identity component of the isomorphism group of $M$, $G$ is closed under $\sigma$, and $\sigma$ is an automorphism of $G$. Since $s_{p_0}$ is involutive, so is $\sigma$. The inclusion $K \subset G^{\sigma}$ follows directly from Proposition \ref{isometry uniqueness thm} since for $k\in K$ we have $p_0 = k(p_0) = \sigma(k)(p_0)$ and $dk_{p_0} = d(\sigma(k))_{p_0}$. For the other inclusion let $\la{h}\subset \la{g}$ be the Lie algebra of $G_0^{\sigma}$ and let $X \in \la{h}$. Then $\exp(X) \in G_0^{\sigma}$ and hence $\sigma(\exp(X)) = \exp(X)$. By the definition of $\sigma$ we then have
\begin{align*}
    s_{p_0}(\exp(X)\cdot {p_0}) = (s_{p_0} \circ \exp(X) \circ s_{p_0})({p_0}) = \sigma(\exp(X))\cdot {p_0} = {p_0}. 
\end{align*}
Thus, $\exp(\la{h}) \subset K$. As $G_0^{\sigma}$ is compact and connected, it is generated by the set $\exp(\la{h})$ so that $G_0^{\sigma} \subset K$ as desired.
\end{proof}
This then motivates the following definition.
\begin{definition}
Let $G$ be a Lie group and $K$ a compact subgroup of $G$. Then if there exists an involutive automorphism of $G$ such that $G^{\sigma}$ is compact and $G_0^{\sigma}\subset K\subset G^{\sigma}$, we call the pair $(G,K)$ a \emph{symmetric pair}. In this thesis we shall also refer to the triple $(G,K,\sigma)$ as a \emph{symmetric triple}.
\end{definition}
In these terms, what we have shown in Proposition \ref{symmetric spaces give symmetric pairs} is that symmetric spaces give rise to symmetric triples. We shall soon show that each symmetric triple also gives rise to a symmetric space. Before we do that we shall first describe the important Cartan decomposition which can be associated to a symmetric pair.
\begin{proposition}\label{cartan decomposition}
Let $(G,K,\sigma)$ be a symmetric triple. Since $\sigma$ is an involutive automorphism of $G$, the differential $d\sigma = d\sigma_e$ is an involutive Lie algebra automorphism of $\la{g}$. Let $\la{k}$ and $\la{p}$ be the eigenspaces of $d\sigma$ corresponding to the eigenvalues $+1$ and $-1$, respectively. Then we have a decomposition of $\la{g}$ as a direct sum
\begin{align*}
    \la{g} = \la{k} \oplus \la{p}.
\end{align*}
This decomposition is orthogonal with respect to the Killing form $B$ of $\la{g}$ and the subspaces $\la{k}$ and $\la{p}$ further satisfy
\begin{enumerate}[label = (\roman*)]
    \item $[\la{k}, \la{k}] \subset \la{k}$ and $\la{k}$ is the Lie algebra of $K$,
    \item $[\la{p},\la{p}] \subset \la{k}$,
    \item $[\la{k},\la{p}] \subset \la{p}$,
    \item $\la{p}$ is $\Ad{K}$-invariant.
\end{enumerate}

\end{proposition}
\begin{proof}
It follows from elementary results in linear algebra that the involutive linear map $d\sigma$ has eigenvalues $\lambda =\pm 1$ and is diagonalisable. Hence, $\la{g} = \la{k}\oplus \la{p}$. Let $X\in \la{k}$ and $Y\in \la{p}$. By Proposition \ref{Killing form properties} we have
\begin{align*}
    B(X,Y) = B(d\sigma(X),d\sigma(Y)) = B(X,-Y) = -B(X,Y).
\end{align*}
Thus, $B(X, Y) = 0$.
For (i), notice that $G_0^{\sigma} \subset K \subset G^{\sigma}$ implies that $K$ and $G^{\sigma}_0$ have the same Lie algebra. The Lie algebra of $G_0^{\sigma}$ consists of all $X \in \la{g}$ satisfying $\sigma(\exp(tX)) = \exp(tX)$ for all $t$. By property (iii) of Proposition \ref{properties of exp}, this is equivalent to $\exp(td\sigma(X)) = \exp(tX)$ for all $t$, which holds if and only if $d\sigma(X) = X$. This is precisely the condition of belonging to the $+1$ eigenspace of $d\sigma$, which proves (i). For (ii), let $X, Y \in \la{p}$. We then have
\begin{align*}
    d\sigma([X,Y]) = [d\sigma(X),d\sigma(Y)] = [-X,-Y] = [X,Y]
\end{align*}
as desired. (iii) is implied by (iv). For (iv) note that we have $d\sigma\circ\Ad{p} = \Ad{\sigma(p)}\circ d\sigma$ for all $p\in G$. In particular, for $k \in K$ and $X \in \la{p}$ we get
\begin{align*}
    d\sigma(\Ad{k}(X)) = \Ad{\sigma(k)}(d\sigma(X)) = -\Ad{\sigma(k)}(X) = -\Ad{k}(X)
\end{align*}
as desired. This then concludes the proof.
\end{proof}
The decomposition in Proposition \ref{cartan decomposition} is known as the \emph{Cartan decomposition}.
\begin{remark}\label{identification}
Let $(G,K,\sigma)$ be a symmetric triple with the associated Cartan decomposition $\la{g} = \la{k} \oplus \la{p}$ and set $M = G/K$ and $p_0 = K \in G/K$. Then with the smooth structure from Theorem \ref{coset space smooth}, the map $d\pi = d\pi_e: \la{g} \cong T_e G \to T_{p_0} M$ is a surjective vector space isomorphism. It is easily verified that $\ker(d\pi) = \la{k}$ and hence we may identify $\la{p}$ with $T_{p_0} M$ via $X \mapsto d\pi(X)$ which we shall denote simply by $X^*$. We then have for smooth functions $f: M \to \rn$
\begin{align*}
    X^*(f) = X(f\circ \pi) = \dtatzero(f \circ \pi)(\exp(tX)) = \dtatzero f(\exp(tX)\cdot p_0).
\end{align*}
Under this identification, for each $k \in K$ and $X\in \la{p}$ we have 
$$\Ad{k}(X)^* = (dL_k)_{p_0}(X^*).$$ 
\end{remark}
\begin{proposition}[\cite{Ziller's notes,Helgason}]\label{G-metrics}
Let $(G,K,\sigma)$ be a symmetric triple with associated Cartan decomposition $\la{g} = \la{k} \oplus \la{p}$. Then each $\Ad{K}$-invariant inner product on $\la{p}$ induces a $G$-invariant metric on $G/K$, and such inner products exist.
\end{proposition}
\begin{proof}
We prove only the equivalence statement. For the existence part see the proof of Proposition 3.4 in Chapter IV of \cite{Helgason}. Set $M = G/K$ and $p_0 = K \in G/K$. Suppose that $\ip{\cdot}{\cdot}$ is an $\Ad{K}$-invariant inner product on $\la{p}$. Then $\ip{\cdot}{\cdot}$ induces an inner product $g_0$ on $T_{p_0} M$ via the identification $T_{p_0} M \cong \la{p}$ seen in Remark \ref{identification} such that $g_0(X^*,Y^*) = \ip{X}{Y}$ for $X,Y \in \la{p}$. We now extend $g_0$ to a $G$-invariant metric by defining 
\begin{align*}
    g_p(X_p,Y_p) = g_0((dL_{x^{-1}})_{p}(X_p),(dL_{x^{-1}})_{p}(Y_p))
\end{align*}
for each $p\in M$, where $x \in G$ such that $p = xK$. This is clearly a $G$-invariant metric if it is well defined. To see that it is, suppose $y\in G$ is such that $yK = p = xK$. Then there exists an element $k\in K$ such that $y^{-1} = k \cdot x^{-1}$. We then have
\begin{align*}
    g_{0}((dL_{y^{-1}})_{p}(X_p),(dL_{y^{-1}})_{p}(Y_p)) &= g_{0}((dL_{k})_{p_0}(dL_{x^{-1}})_{p}(X_p),((dL_{k})_{p_0}(dL_{x^{-1}})_{p}(Y_p))\\
    &= g_{0}((dL_{x^{-1}})_{p}(X_p),(dL_{x^{-1}})_{p}(Y_p))
\end{align*}
by Remark \ref{identification} since $\ip{\cdot}{\cdot}$ is $\Ad{K}$-invariant. Hence, $g$ is well defined. Now assume that $g$ is a $G$-invariant metric on $M$. We define an inner product $\ip{\cdot}{\cdot}$ on $\la{p}$ as the inner product induced by the inner product $g_{p_0}$ on $T_{p_0} M$ via the identification in Remark \ref{identification}. Since $g$ is $G$-invariant, the inner product $g_{p_0}$ is $dL_{K}$-invariant, and hence $\ip{\cdot}{\cdot}$ is $\Ad{K}$-invariant. 
\end{proof}
In particular, Proposition \ref{G-metrics} shows that each symmetric triple gives rise to a Riemannian homogeneous space. We now have the tools necessary to show that indeed this homogeneous space has the structure of a symmetric space.
\begin{theorem}[\cite{Helgason, Ziller's notes}]\label{symmetric pairs give symmetric spaces}
Let $(G,K,\sigma)$ be a symmetric triple. Then the coset space $M = G/K$ equipped with any $G$-invariant metric is a symmetric space, and such metrics exist.
\end{theorem}
\begin{proof}
We shall first construct the symmetry of $M$ at $p_0 = K \in M$ and then use left translations to construct the symmetries at arbitrary points of $M$. Let $g$ be any $G$-invariant metric on $M$ corresponding to some $\Ad{K}$-invariant inner product on $\la{p}$ as in Proposition \ref{G-metrics}.  We define 
\begin{align*}
    s_{p_0}(pK) = \sigma(p)K.
\end{align*}
Clearly $s_{p_0}$ is an involutive diffeomorphism of $M$, fixing $p_0$. Since $\restr{d\sigma}{\la{p}} = -id_{\la{p}}$ and $s_{p_0} \circ \pi = \pi \circ \sigma$ we get $d(s_{p_0})_{p_0} = -id_{T_{p_0} M}$. To see that $s_{p_0}$ is an isometry, let $x,y \in G$ and $p = yK \in M$. Then
\begin{align*}
    (s_{p_0} \circ L_x)(p) &= (s_{p_0} \circ L_x)(yK)\\
    &= s_{p_0}(\pi(xy))\\
    &= (\pi\circ \sigma)(xy)\\
    &= (\pi \circ L_{\sigma(x)} \circ \sigma)(y)\\
    &= L_{\sigma(x)}(s_{p_0}(p)), 
\end{align*}
so that $s_{p_0} \circ L_{x} = L_{\sigma(x)} \circ s_{p_0}$ for all $x \in G$. Now let $p = xK \in M$ and $X_p, Y_p \in T_p M$. Put $X_0 = (dL_{x^{-1}})_p(X_p)$ and $Y_0 = (dL_{x^{-1}})_p(Y_p)$. We then have
\begin{align*}
    g_{s_{p_0}p}((ds_{p_0})_{p}(X_p),(ds_{p_0})_{p}(Y_p)) &= g_{s_{p_0}p}((ds_{p_0})_{p}(dL_{x})_{p_0}(X_0),(ds_{p_0})_{p}(dL_{x})_{p_0}(Y_0))\\
    &=g_{s_{p_0}p}((dL_{\sigma(x)})_{p_0}(ds_{p_0})_{p_0}(X_0),(dL_{\sigma(x)})_{p_0}(ds_{p_0})_{p_0}(Y_0))\\
    &=g_{s_{p_0}p}((dL_{\sigma(x)})_{p_0}(X_0),(dL_{\sigma(x)})_{p_0}(Y_0))\\
    &=g_{p_0}(X_0,Y_0)\\
    &=g_{p}(X_p,Y_p)
\end{align*}
as desired. For arbitrary $p = xK \in M$ we define $s_p$ by $s_p = L_x \circ s_{p_0} \circ L_{x^{-1}}$. It is straightforward to check that $s_p$ is an involutive diffeomorphism with $p$ as an isolated fixed point. Since $s_p$ is a composition of isometries it is also an isometry, and with this the proof is complete.
\end{proof}
This completes the correspondence between symmetric pairs and symmetric\\ spaces, so that questions about symmetric spaces can be translated into questions about symmetric pairs. The full toolkit of Lie theory is then available to us to understand symmetric spaces. We now define what is meant by an irreducible symmetric space.
\begin{definition}\label{irreducible def}
Let $(G,K,\sigma)$ be a symmetric triple and put $M = G/K$ and $p_0 = K \in M$. Then we say that the triple $(G,K,\sigma)$ is \emph{irreducible} if the action of the identity component $K_0$ on $T_{p_0} M$ is irreducible. We say that a symmetric space $M$ is irreducible if the symmetric triple associated to $M$ in Proposition \ref{symmetric spaces give symmetric pairs} is.
\end{definition}
As the following result shows, the irreducible symmetric spaces serve as the basic building blocks in the study of symmetric spaces.
\begin{theorem}[\cite{Kobayashi-Nomizu,Ziller's notes}]\label{1st decomposition theorem}
Let $M$ be a simply connected symmetric space. Then $M$ is isometric to a Riemannian product $M_1 \times \cdots \times M_n$ of irreducible symmetric spaces.
\end{theorem}
The irreducible symmetric spaces come in three different types.
\begin{definition}
Let $(G,K,\sigma)$ be a symmetric triple with Cartan decomposition $\la{g} = \la{k} \oplus \la{p}$ and with $B$ the Killing form of $\la{g}$. We then say that $(G,K,\sigma)$ is of
\emph{Euclidean type} if $\restr{B}{\la{p}}$ vanishes identically, \emph{compact type} if $\restr{B}{\la{p}}$ is negative definite and
\emph{non-compact type} if $\restr{B}{\la{p}}$ is positive definite.
We say that a symmetric space $M$ is of Euclidean, compact or non-compact type, respectively, if the associated symmetric triple is.
\end{definition}
\begin{proposition}[\cite{Ziller's notes}]\label{irreducibles have types}
Let $(G,K,\sigma)$ be an irreducible symmetric triple. Then $(G,K,\sigma)$ is of either Euclidean, compact or non-compact type.
\end{proposition}
\begin{proposition}[\cite{Ziller's notes}]\label{geometry of the types}
Let $M = G/K$ be a symmetric space.
\begin{enumerate}[label = (\roman*)]
    \item If $M$ is of compact type then $G$ is semisimple and $M$ and $G$ are both compact.
    \item If $M$ is of non-compact type then $G$ is semisimple and both $M$ and $G$ are non-compact.
    \item If $M$ is of Euclidean type and simply connected then $M \cong \rn^{m}$ for some $m$.
\end{enumerate}
\end{proposition}
\begin{remark}
It is not hard to verify that the Riemannian product of any two symmetric spaces yields another symmetric space. It is further true and easy to verify that the product of any two Riemannian symmetric spaces of the same type will yield a symmetric space of that type.
\end{remark}
We thus get the following refinement of Theorem \ref{1st decomposition theorem}.
\begin{theorem}[\cite{Kobayashi-Nomizu,Ziller's notes}]\label{2nd decomposition theorem}
Let $M$ be a simply connected symmetric space. Then $M$ is isometric to a Riemannian product $M\cong M_0 \times M_1 \times M_2$, with $M_0$ of Euclidean type, $M_1$ of compact type and $M_2$ of non-compact type.
\end{theorem}
In this work we will focus on symmetric spaces of compact type. However, there is a remarkable duality between the symmetric spaces of compact type and those of non-compact type. This allows for some results on symmetric spaces of compact type to be translated directly into results on their non-compact duals. We shall now describe this duality. A more detailed and thorough account can be found in \cite{Helgason} and \cite{Kobayashi-Nomizu}.

Let $(G,K,\sigma)$ be a symmetric triple with $M = G/K$ simply connected. We let $\la{g} = \la{k}\oplus \la{p}$ be the Cartan decomposition of $\la{g}$. We write $\la{g}^{\cn} = \la{g} \otimes \cn$ for the complexification of $\la{g}$, which we turn into a complex Lie algebra by extending the bracket on $\la{g}$ so that it is complex bilinear. We then define a subalgebra $\la{g}^*$ of $\la{g}$ by
\begin{align*}
    \la{g}^* = \la{k} \oplus i\la{p}.
\end{align*}
We define the automorphism $d\sigma^*$ of $\la{g}^*$ as the restriction to $\la{g}^*$ of the complex linear extension of $d\sigma$ to $\la{g}^{\cn}$. By the computation
\begin{align*}
    d\sigma^*(X + iY) = d\sigma(X) + id\sigma(Y) = X - iY
\end{align*}
for $X \in \la{k}$ and $Y\in \la{p}$, we see that we may equivalently define $d\sigma^*$ to be the restriction of the conjugation map in $\la{g}^{\cn}$ to $\la{g}^*$.
Then $\la{k}$ and $i\la{p}$ are the eigenspaces of $d\sigma^*$ corresponding to the eigenvalues $+1$ and $-1$ respectively.

We let $G^*$ be the simply connected Lie group with Lie algebra $\la{g}^*$ and $K^*$ be the Lie subgroup of $G^*$ with Lie algebra $\la{k}$. There then exists a unique automorphism $\sigma^*$ of $G^*$ with $(d\sigma^*)_e = d\sigma^*$. The triple $(G^*,K^*,\sigma^*)$ is then a symmetric triple and hence, $M^* = G^*/K^*$ a simply connected symmetric space. The triple $(G^*,K^*,\sigma^*)$ is called the \emph{dual} symmetric triple to $(G,K,\sigma)$, and $M^* = G^*/K^*$ the dual symmetric space of $M= G/K$. The following shows duality to be reflexive.
\begin{proposition}[\cite{Ziller's notes}]
Let $(G,K,\sigma)$ be a symmetric triple with $M = G/K$ simply connected. Then $(G,K,\sigma)$ is the dual to the triple $(G^*,K^*,\sigma^*)$.
\end{proposition}
\begin{theorem}[\cite{Ziller's notes}]
Let $(G,K,\sigma)$ be a symmetric triple with $M = G/K$. Then if $(G,K,\sigma)$ is of compact type, the dual $(G^*,K^*,\sigma^*)$ is of non-compact type and vice versa. Further, if $(G,K,\sigma)$ is irreducible, so is its dual.
\end{theorem}
Hence, from a classification point of view, it is sufficient (if Euclidean factors are disregarded) to classify the irreducible symmetric spaces of compact type. This classification was achieved by the French mathematician \'Elie Cartan in his 1926 and 1927 papers {\it Sur une classe remarquable d'espaces de Riemann} \cite{Cartan 1} and {\it Sur une classe remarquable d'espaces de Riemann. II} \cite{Cartan 2}, respectively.

For the compact irreducible symmetric spaces, Cartan found that they belong either to one of $10$ infinite families, or one of $12$ exceptional cases. The symmetric spaces in the first case can all be realised as quotients of the classical compact matrix groups described in Chapter \ref{ch: classical groups} and are refered to as the \emph{classical compact irreducible symmetric spaces}. We list them here and shall return to cover them in greater detail in Chapter \ref{ch: results}.
\begin{example}[\cite{Cartan 1,Cartan 2}]\label{the classical compact spaces}
The $7$ families of classical compact irreducible symmetric spaces which are not themselves compact semisimple Lie groups can be realised as the following homogenous spaces. 
\begin{enumerate}[label = (\roman*)]
    \item $\U{m+n}/\U{m}\times \U{n}$, $m,n \geq 1$.
    \item $\SO{m+n}/\SO{m}\times\SO{n}$, $m,n \geq 1$.
    \item $\Sp{m+n}/\Sp{m}\times\Sp{n}$, $m,n \geq 1$.
    \item $\SU{n}/\SO{n}$, $n \geq 1$.
    \item $\SO{2n}/\U{n}$, $n \geq 1$.
    \item $\Sp{n}/\U{n}$, $n \geq 1$.
    \item $\SU{2n}/\Sp{n}$, $n \geq 1$.
\end{enumerate}
\end{example}
\begin{remark}\label{rem: natural metrics}
   Let $(G,K,\sigma)$ be an irreducible symmetric triple of compact type, with $G$ semisimple, and let $(G^*,K^*,\sigma^*)$ be its non-compact dual. From Theorem \ref{symmetric pairs give symmetric spaces} we know that each $\Ad{K}$-invariant inner product on $\la{g}$ or $\la{g}^*$ gives rise to a symmetric metric on $G/K$ or $G^*/K^*$, respectively. We shall now describe natural choices of such inner products. Let $\alpha > 0$ be a positive real scalar. On $\la{g}$ we then define an inner product $\ip{\cdot}{\cdot}$ by
   \begin{align*}
       \ip{X}{Y} = -\alpha\cdot B(X,Y)
   \end{align*}
   where $B$ is the killing form on $\la{g}$. We have previously shown this to be an inner product which is $\Ad{G}$ - and hence - $\Ad{K}$-invariant. For the non-compact dual, we instead define the inner product $\ip{\cdot}{\cdot}^*$ by
   \begin{align*}
       \ip{X}{Y}^* = -\alpha\cdot B(d\sigma^*(X),Y).
   \end{align*}
    To see that this is an inner product on $\la{g}^*$, note that for $X,Y \in \la{k}$ we have
    \begin{align*}
        \ip{X}{Y}^* = \ip{X}{Y}
    \end{align*}
    and that for $X,Y \in \la{p}$ we get
    \begin{align*}
        \ip{iX}{iY}^* = -\alpha\cdot B(-iX,iY) = -\alpha\cdot B(X,Y) = \ip{X}{Y}.
    \end{align*}
    Hence $\ip{\cdot}{\cdot}^*$ is an inner product on each of the subspaces $\la{k}$ and $i\la{p}$. As these are further orthogonal complements to each other with respect to $\ip{\cdot}{\cdot}^*$, it is an inner product on their direct sum $\la{g}^*$. $\Ad{K}$-invariance follows from $\Ad{K}$-invariance of the killing form $B$ and the fact that $\Ad{k}$ commutes with $d\sigma^*$ for each $k\in K$. 
    \par
    We let $\{X_1,\dots, X_k\}$ and $\{Y_1,\dots,Y_l\}$ be orthonormal bases for $\la{k}$ and $\la{p}$, respectively, with respect to $\ip{\cdot}{\cdot}$. Then it is easy to see that the set $\{X_1,\dots,X_k,Y_1,\dots Y_l\}$ is an orthonormal basis for $\la{g}$ and $\{X_1,\dots,X_k,iY_1,\dots iY_l\}$ is an orthonormal basis for $\la{g}^*$.
\end{remark}
\chapter{The Cartan Embedding}\label{ch: cartan embedding}
In this section we introduce the important Cartan embedding. We let $(G,K,\sigma)$ be a symmetric triple of compact type with $K = G^{\sigma}$. The Cartan embedding is a map $$\hat{\Phi}: M = G/K \to G$$
embedding the symmetric space $M = G/K$ as a totally geodesic submanifold of the compact Lie group $G$. Over the course of this chapter we shall prove, although by different methods, the following result equivalent to Proposition 3.42 in \cite{Cheeger-Ebin}.
\begin{theorem}[\cite{Cheeger-Ebin}]
    Let $(G,K,\sigma)$ be a symmetric triple of compact type, with $G^{\sigma} = K$, and let $\la{g} = \la{k} \oplus \la{p}$ be the corresponding Cartan decomposition of the Lie algebra $\la{g}$ of $G$. Then the Cartan embedding $\hat{\Phi}: G/K \to G$ satisfies
    \begin{enumerate}[label = (\roman*)]
        \item $\hat{\Phi}$ is an embedding,
        \item the image of $\hat{\Phi}$ is totally geodesic in $G$,
        \item $\hat{\Phi}(G/H) = \exp(\la{p})$,
        \item $\hat{\Phi}$ is a conformal map onto its image, with constant conformal factor $2$.
    \end{enumerate}
\end{theorem}
Later in Chapter \ref{ch: harmonic morphisms} we shall see that the Cartan embedding and related Cartan map are also harmonic, which we shall then use to construct eigenfunctions and eigenfamilies on the classical irreducible symmetric spaces of compact type. 
\begin{definition}\label{Cartan map}
Let $(G,K,\sigma)$ be a symmetric triple of compact type and $G^{\sigma} = K$. Then the \emph{Cartan map} $\Phi: G \to G$ is defined by
\begin{align*}
    \Phi: p \mapsto p\cdot\sigma(p^{-1}).
\end{align*}
Since for each $k \in K$, we have $\Phi(pk) = \Phi(p)$, the map $\Phi$ induces $\hat{\Phi}: G/K \to G$ with
\begin{align*}
    \hat{\Phi}: pK \mapsto \Phi(p) = p
    \cdot \sigma(p^{-1}),
\end{align*}
which we shall call the \emph{Cartan embedding}.
\end{definition}

In the following result we establish some rudimentary algebraic properties of the Cartan map.

\begin{proposition}\label{Cartan map properties}
Let $(G,K,\sigma)$ be a symmetric triple of compact type, then the Cartan map $\Phi$ satisfies
\begin{enumerate}[label = (\roman*)]
    \item $\sigma(\Phi(p)) = \Phi(\sigma(p)) = \Phi(p)^{-1}$,
    \item $\Phi(p)\Phi(q)\Phi(p) = \Phi(\Phi(p)q)$,
\end{enumerate}
for all $p,q \in G$.
In particular, this shows that the image $\Phi(G) = \hat{\Phi}(G/K)$ is closed under both inversion and squaring.
\end{proposition}

\begin{proof}
For the first equation we use the involutive propery of $\sigma$ to compute
\begin{align*}
    \sigma(\Phi(p)) = \sigma(p\sigma(p^{-1})) = \sigma(p)p^{-1} = (p\sigma(p^{-1}))^{-1} =\Phi(p)^{-1}
\end{align*}
and 
\begin{align*}
    \Phi(\sigma(p)) = \sigma(p)\sigma^{2}(p^{-1}) = \sigma(p\sigma(p^{-1})) = \sigma(\Phi(p)).
\end{align*}
For the second equation, when expanding the right hand side we get
\begin{align*}
    \Phi(\Phi(p)q) &= p\sigma(p^{-1})q\sigma(q^{-1}\sigma(p)p^{-1})\\
    &= p\sigma(p^{-1})q\sigma(q^{-1})p\sigma(p^{-1})\\
    &= \Phi(p)\Phi(q)\Phi(p),
\end{align*}
as desired.
\end{proof}

\begin{proposition}
Let $(G,K,\sigma)$ be a symmetric triple of compact type, then the Cartan embedding $\hat{\Phi}$ is injective.
\end{proposition}
\begin{proof}
Assume $p, q \in G$ such that $\hat{\Phi}(pK) = \Phi(p) = \Phi(q) = \hat{\Phi}(qK)$. Then
\begin{align*}
    p\sigma(p^{-1}) = q\sigma(q^{-1})
\end{align*}
implies
\begin{align*}
    q^{-1}p = \sigma(q^{-1}p)
\end{align*}
so that $q^{-1}p \in G^{\sigma} = K$. But then $pK = qK$, and hence $\hat{\Phi}$ is injective.
\end{proof}

\begin{theorem}\label{Embedding}
Let $(G,K,\sigma)$ be a symmetric triple of compact type. Then the Cartan embedding $\hat{\Phi}: G/K \to G$ is an embedding of the symmetric space $G/K$ as a submanifold of $G$. 
\end{theorem}

\begin{proof}
Set $N = \hat{\Phi}(G/K) = \Phi(G)$. Since $\hat{\Phi}$ is injective and the domain of $\hat{\Phi}$ is compact, we need only show that $\hat{\Phi}$ is an immersion, as an injective immersion of a compact domain it is always an embedding. The canonical projection $\pi: G \to G/K$ is a Riemannian submersion and the horizontal subspace at a point $p \in G$ is given by
\begin{align*}
    \hor_p = (dL_p)_e(\la{p}). 
\end{align*}
Here 
\begin{align*}
    \la{g} = \la{k} \oplus \la{p}
\end{align*}
is the Cartan decomposition of the Lie algebra $\la{g}$ of $G$ associated with $(G,K,\sigma)$. A proof of this claim can be found on page 466 of \cite{The fundamental equations of a submersion} . It will thus suffice to show that at each point $p \in G$, the differential
\begin{align}\label{injectivity condition}
\restr{d\Phi_p}{\hor_p}: \hor_p = (dL_p)_e(\la{p}) \to T_{\Phi(p)}N     
\end{align}
of $\Phi$ restricted to the horizontal subspace at $p$ is a vector space isomorphism. To calculate the differential of the Cartan map at a point $p \in G$ we rewrite $\Phi$ as a composition
\begin{align*}
    \Phi = \mu \circ \phi
\end{align*}
where $\phi: G \to G\times G$ is the map
\begin{align*}
    \phi: p \mapsto (p,\sigma(p^{-1}))
\end{align*}
and $\mu: G\times G \to G$ is the multiplication map. Let $X_p \in T_p G$ and $X \in \la{g}$ be such that $X_p = (dL_p)_e(X)$, then we get
\begin{align*}
    d\phi_p(X_p) &= \dtatzero\phi(p\cdot \exp(tX))\\ 
    &= \dtatzero(p\cdot \exp(tX),\sigma(\exp(-tX)\sigma(p^{-1})))\\
    &= ((dL_p)_e(X),-(dR_{\sigma(p^{-1})})_e(d\sigma(X))).
\end{align*}
In the final step we have used property (iii) from Proposition \ref{properties of exp} and the linearity of $(dR_{\sigma(p^{-1})})_e$. For the differential of $\Phi$ at $p$ the chain rule together with Proposition \ref{diff mul} then gives
\begin{align*}
    d\Phi_p(X_p) &= d\mu_{(p,\sigma(p^{-1}))}((dL_p)_e(X),-(dR_{\sigma(p^{-1})})_e(d\sigma(X)))\\
    &= (dR_{\sigma(p^{-1})})_p(dL_p)_e(X) + (dL_p)_{\sigma(p^{-1})}(dR_{\sigma(p^{-1})})_e(-d\sigma(X))\\
    &= d(R_{\sigma(p^{-1})} \circ L_p)_e (X) + d(L_p \circ R_{\sigma(p^{-1})})_e(-d\sigma(X))\\
    &= d(R_{\sigma(p^{-1})} \circ L_p)_e(X - d\sigma(X)).
\end{align*}
In the last step we have used that right and left translations commute with each other. For $X \in \la{k}$, we have $d\sigma(X) = X$ and the inner term vanishes as expected. On the other hand if $X \in \la{p}$, we have $d\sigma(X) = -X$ and we get
\begin{align*}
    d\Phi_p(X_p) &= d(R_{\sigma(p^{-1})} \circ L_p)_e(X - d\sigma(X))\\
    &= d(R_{\sigma(p^{-1})} \circ L_p)_e(2X)\\
    &= 2(dR_{\sigma(p^{-1})})_p(X_p).
\end{align*}
The map $2(dR_{\sigma(p^{-1})})_p$ is obviously a vector space isomorphism and hence $\Phi$ satisfies the condition \eqref{injectivity condition} for all points $p\in G$. Thus $\hat{\Phi}$ is an embedding of $G/K$ into $G$ as desired.
\end{proof}

\begin{corollary}\label{Cartan conformal}
The Cartan embedding is a conformal map with a constant conformal factor $4$. 
\end{corollary}
\begin{proof}
Since $\pi: G \to G/K$ is a Riemannian submersion and the Cartan map $\Phi$ satisfies $\Phi = \hat{\Phi} \circ \pi$, we need only show that the Cartan map is conformal when restricted to the horizontal distribution $\hor$ of the submersion. If $p \in G$ is a point of $G$, $X_p, Y_p \in \hor_p$ are horizontal vectors in $T_p G$ and $g$ is the bi-invariant metric on $G$ we have from the proof of Theorem \ref{Embedding} that
\begin{align*}
    g_{\Phi(p)}(d\Phi_p(X_p),d\Phi_p(Y_p)) &= g_{\Phi(p)}(2dR_{\sigma(p^{-1})}(X_p),2dR_{\sigma(p^{-1})}(Y_p))\\
    &= 4 \cdot g_p(X_p,Y_p),
\end{align*}
where the last equality follows from right-invariance of the metric.
\end{proof}

\begin{theorem}\label{Cartan totally geodesic}
Let $(G,K,\sigma)$ be a symmetric triple of compact type. Then the image, $N = \hat{\Phi}(G/K) = \Phi(G)$, of the Cartan embedding is a totally geodesic submanifold of $G$.
\end{theorem}

\begin{proof}
We shall show $N$ to be totally geodesic by considering $G$ as a symmetric space as in Proposition \ref{Lie group symmetric} and show that for each $p \in N$, the submanifold $N$ is invariant under $s_p$ where $s_p$ is the symmetry of $G$ at $p$. The result will then follow by Proposition \ref{point invariance}. Let $p = \Phi(a) \in N$ and $q = \Phi(b) \in N$ for $a,b \in G$. Then
\begin{align*}
    s_p(q) = pq^{-1}p = \Phi(a)\Phi(b)^{-1}\Phi(a).
\end{align*}
Using $(i)$ and $(ii)$ from Proposition \ref{Cartan map properties} we obtain
\begin{align*}
    \Phi(a)\Phi(b)^{-1}\Phi(a) &= \Phi(a)\Phi(\sigma(b))\Phi(a)\\
    &= \Phi(\Phi(a)\sigma(b)) \in N.
\end{align*}
This then concludes the proof.
\end{proof}
\begin{theorem}\label{exponential form}
Let $(G,K,\sigma)$ be a symmetric triple of compact type with $G$ connected, and let $N$ be the image of the Cartan map $\Phi$. Let $\la{g} = \la{k} \oplus \la{p}$ be the Cartan decomposition of $\la{g}$ associated with $(G,K,\sigma)$, then
\begin{align*}
    N = \exp(\la{p}).
\end{align*}
\end{theorem}
\begin{proof}
As $G$ is a compact Lie group, the geodesics through the identity element $e\in G$ with respect to the bi-invariant metric $g$ on $G$ correspond to the one-parameter subgroups of $G$, and hence
\begin{align*}
    \exp(\la{p}) = \text{Exp}_e(\la{p})
\end{align*}
where $\text{Exp}_e: T_e G \to G$ is
the Riemannian exponential map at the point $e$. As $G$ is complete $\exp_e$ is defined on the whole of $T_e G$ and therefore certainly on the whole of $\la{p}$. We first note that 
$$T_e N = \la{p} = T_e \exp(\la{p}),$$
as the differential of $\Phi$ at $e$ restricted to $\la{p}$ is simply $\restr{d\Phi_e}{\la{p}}(X_e) = 2\cdot X_e$. Since $N$ is totally geodesic and compact, it is also complete. Hence, if $p \in N$ there is a geodesic $\gamma_p:[0,1] \to N$ such that $\gamma_p(0) = e$ and $\gamma_p(1) = p$. As $\gamma_p$ is also a geodesic of $G$ with $\dot{\gamma}_p(0) \in \la{p}$ it is contained in $\exp(\la{p})$ and hence $N \subset \exp(\la{p})$. 

Conversely, if $q$ is a point of $\exp(\la{p})$, the geodesic $\alpha_q:[0,1] \to G$ with $\alpha_q(0) = e$ and $\alpha_q(1) = q$ is contained in $\exp(\la{p})$, with $\dot{\alpha}_q(0) \in \la{p}$. But then since $N$ is totally geodesic, $\alpha_q$ is a geodesic in $N$ as well, so that $q$ is also a point of $N$. Hence $\exp(\la{p}) \subset N$, which concludes the proof.
\end{proof}

\chapter{Harmonic Morphisms and p-Harmonic Functions}\label{ch: harmonic morphisms}
As stated in the introduction, the aim of this thesis is to find harmonic morphisms and $p$-harmonic functions from the classical compact irreducible symmetric spaces to the complex plane. We will achieve this goal using the powerful tools that are eigenfamilies and eigenfunctions, which will be introduced in Chapter \ref{ch: eigen}.  Here we give a brief general overview of harmonic morphisms between Riemannian manifolds and $p$-harmonic functions. The standard reference on the theory of harmonic morphisms is the exellent book {\it Harmonic Morphisms Between Riemannian Manifolds} \cite{Baird and Wood} by P. Baird and J.C. Wood. We further recommend the interested reader to consult the excellent and frequently updated online bibliographies \cite{biblio harm-morph} and \cite{biblio p-harm} of harmonic morphisms and $p$-harmonic functions, respectively.

First we will introduce the notion of the gradient of a real or complex-valued function on a Riemannian manifold. 
\begin{definition}
Let $(M,g)$ be a Riemannian manifold and $f:M \to \rn$ a smooth function on $M$. Then we define the \emph{gradient} $\nabla f$ of $f$ as the unique smooth vector field on $M$ satisfying
\begin{align*}
    g(\nabla f,X) = df(X) = X(f)
\end{align*}
for all smooth vector fields $X \in \smooth{TM}$.
\end{definition}
The notion of a gradient of a real valued function is then readily extended to complex-valued ones. Given a Riemannian manifold $(M,g)$ we shall by $T^{\cn}M$ denote the complexification of the tangent bundle of $M$. We may decompose a complex-valued function $f:M \to \cn$ into real and imaginary components $f = u + iv$ and define the gradient $\nabla f$ as the smooth complex vector field
\begin{align*}
    \nabla f = \nabla u + i\cdot \nabla v
\end{align*}
on $M$. Further, for a smooth vector field $Y\in \smooth{T^{\cn} M}$, we define the divergence $\div(Y)$ by
\begin{align*}
    \div(Y) = \tr(\nabla Y),
\end{align*}
where by $\nabla Y$ we mean the $(1,1)$-tensor field $\nabla Y(X) =\nab{}{X}{Y}$, $X \in \smooth{T^{\cn}M}$. It is important to note that we extend the metric tensor on $\smooth{TM}$ to $\smooth{T^{\cn}M}$ as a complex-bilinear symmetric $(0,2)$-tensor. It thus restricts to a complex-bilinear symmetric form on each complexified tangent space.

We are now ready to define the \emph{tension field} on a Riemannian manifold $(M,g)$. This will be a second order elliptic differential operator on $M$ which generalises the Laplacian on $\rn^n$.
\begin{definition}
Let $(M,g)$ be a Riemannian manifold and $f:M \to \cn$ a complex-valued $C^2$-function on $M$. We then define the tension field $\tau$ acting on $f$ by
\begin{align*}
    \tau(f) = \div(\nabla f).
\end{align*}
We say that $f$ is \emph{harmonic} if it satisfies
\begin{align*}
    \tau(f) = 0.
\end{align*}
\end{definition}
\begin{definition}[\cite{Baird and Wood}]
A map $\phi: (M,g) \to (N,h)$ between Riemannian manifolds is called a \emph{harmonic morphism} if for each open subset $U$ of $M$ with non-empty image under $\phi$ and each harmonic function $f:M \to \cn$, the function $f\circ \phi: \phi^{-1}(U) \to \cn$ is harmonic as well.
\end{definition}
We remark that on a smooth manifold $(M,g)$ we may for any given point $p \in M$ find a neighbourhood $U$ of $p$ on which we can define a local orthonormal frame $\basis = \{Z_1, \dots, Z_n\}$ for $T^{\cn}M$. In this setting we get the very useful local formula
\begin{align}\label{eq:local tension}
    \tau(f) = \sum_{i = 1}^n Z_k^2(f) - (\nab{}{Z_k}{Z_k})(f)
\end{align}
for the tension field $\tau(f)$ of a complex-valued function $f$ defined on $U$. To see this we compute
\begin{align*}
    \tau(\phi) &= \tr(\nabla(\nabla f))\\
    &= \sum_{k=1}^n g(Z_k, \nab{}{Z_k}{\nabla f})\\
    &= \sum_{k=1}^n Z_k(g(Z_k,\nabla f)) - g(\nab{}{Z_k}{Z_k},\nabla f)\\
    &= \sum_{k=1}^n Z_k(Z_k(f)) - (\nab{}{Z_k}{Z_k})(f),
\end{align*}
in $U$.
\par

There is a more general notion of harmonicity for smooth maps $\phi:(M,g) \to (N,h)$ between Riemannian manifolds which will be of theoretical importance for a few of the results in this thesis. In order to cover this more general setting we must first generalise the tension field so that it can act on such maps.

For this we first make some constructions. If $\varphi: (M,g) \to (N,h)$ is a smooth map between Riemannian manifolds, we can use it to construct a vector bundle $\varphi^{-1}TN$ on $M$ called the \emph{pull-back bundle} \cite{Baird and Wood}. This is given by
\begin{align*}
    \varphi^{-1}TN = \{ (p,X_{\varphi(p)} \ | \ p\in M, X_{\varphi(p)} \in T_{\varphi(p)}N), \quad \pi:(p,X_{\varphi(p)}) \mapsto p,
\end{align*}
so that the fibre at $p \in M$ of the pull-back bundle is the fibre of $TN$ at the image $\varphi(p) \in N$. A vector field $Y \in \smooth{TN}$ then naturally corresponds to a smooth section $Y^*\in \smooth{\varphi^{-1}TN}$ of the pull-back bundle via $Y^*(p) = Y(\varphi(p))$. We then define the \emph{pull-back connection} $\nabla^{\varphi}$ on $\varphi^{-1}TN$ by requiring that
\begin{align*}
    \nab{\varphi}{X}{Y^*} = \tnab{d\varphi(X)}{Y}
\end{align*}
hold for all $X \in \smooth{TM}$ and all $Y \in \smooth{TN}$, where $\tnabla$ is the Levi-Civita connection on $(N,h)$.

We can then view the differential $d\varphi$ of $\varphi$ as a tensorial map $d\varphi: \smooth{TM} \to \smooth{\varphi^{-1}TN}$, or equivalently a section of $T^*M \otimes \varphi^{-1}TN$, by setting 
$$d\varphi(X)(p) = (p,d\varphi_p(X_p))$$
for $X \in \smooth{TM}$. It then makes sense to take the covariant derivative $\nabla d\varphi$ of $\varphi$, for which we obtain the expression \cite{Baird and Wood}
\begin{align*}
    \nabla d\varphi (X,Y) = \nab{\varphi}{X}{d\varphi(Y)}-d\varphi\left(\nab{}{X}{Y}\right)
\end{align*}
for $X, Y \in \smooth{TM}$, where $\nabla$ is the Levi-Civita connection of $(M,g)$.
\begin{definition}[\cite{Baird and Wood}]
Let $(M,g)$ and $(N,h)$ be Riemannian manifolds, and $\varphi: M \to N$ a smooth map. Then we define the tension field $\tau(\varphi)$ by
\begin{align*}
    \tau(\varphi) = \tr(\nabla d\varphi).
\end{align*}
We say $\varphi$ is \emph{harmonic} if $\tau(\varphi) = 0$.
\end{definition}
The following result on the tension field of a composition of maps is known as the \emph{composition law}.
\begin{proposition}[\cite{Baird and Wood}]\label{composition law}
Let $\phi:(M,g) \to (N,h)$ and $\psi: (N,h) \to (P,q)$ be smooth maps between Riemannian manifolds. Then the tension field of their composition satisfies
\begin{align*}
    \tau(\psi \circ \phi) = d\psi(\tau(\phi)) + \tr(\nabla d\psi(d\phi,d\phi)).
\end{align*}
\end{proposition}
The following characterisation of harmonic morphisms is due to Fuglede and Ishihara.
\begin{theorem}[\cite{Fuglede, Ishihara}]\label{harm-morph iff hwc}
A map $\varphi: (M,g) \to (N,h)$ between Riemannian manifolds is a harmonic morphism if and only if it is a horizontally (weakly) conformal harmonic map.
\end{theorem}
By horizontally (weakly) conformal we mean the following.
\begin{definition}
Let $\varphi:(M,g) \to (N,h)$ be a map between Riemannian manifolds and define for each $p\in M$ the vertical space at $p$ by $\vert_{p} = \ker d\varphi_p$ and the horizontal space at $p$ by $\hor_p = \vert_p^{\perp}$. We say that $\varphi$ is \emph{horizontally (weakly) conformal} at $p$ if
\begin{enumerate}[label = (\roman*)]
    \item $d\varphi_p = 0$, or
    \item $d\varphi_p$ is surjective and there exists a number $\lambda(p)$ such that
    \begin{align*}
        g_{p}(X_p,Y_p) = h_{\varphi(p)}(d\varphi_p(X_p),d\varphi_p(Y_p))
    \end{align*}
    holds for each $X_p,Y_p \in \hor_p$.
\end{enumerate}
We call $\varphi$ horiznotally (weakly) conformal if this holds at each point $p \in M$.
\end{definition}
We remark that the above definition reduces to the condition
\begin{align*}
g(\nabla \phi, \nabla \phi) = 0
\end{align*}
for a complex-valued function $\phi:(M,g) \to \cn$. This motivates the following definition.
\begin{definition}\label{def of kappa}
Let $(M,g)$ be a Riemannian manifold. We then define the conformality operator $\kappa$ of $M$ to be the bilinear differential operator
\begin{align*}
    \kappa(\phi,\psi) = g(\nabla \phi,\nabla \psi)
\end{align*}
for smooth maps $\phi,\psi: M \to \cn$.
\end{definition}
Hence, a complex-valued map $\phi:(M,g) \to \cn$ is horizontally conformal if and only if $\kappa(\phi,\phi) =0$.
The conformality operator $\kappa$ and tension field $\tau$ of a Riemannian manifold $(M,g)$ also satisfy the well known product identities
\begin{align*}
    \tau(\phi\cdot \psi) = \phi\cdot\tau(\psi) + 2\cdot \kappa(\phi,\psi) + \tau(\phi)\cdot\psi
\end{align*}
and
\begin{align*}
    \kappa(\phi\cdot \psi,f\cdot g) &= \phi\cdot f\cdot \kappa(\psi,g) + \phi\cdot g\cdot \kappa(\psi, f) + \psi\cdot f\cdot \kappa(\phi,g) + \psi\cdot g\cdot \kappa(\phi, f)
\end{align*}
for smooth functions $\phi,\psi,f,g: M \to \cn$.
As for the tension field, given a local orthonormal frame $\basis = \{Z_1,\cdots, Z_n\}$ of $\smooth{TM}$ we may compute $\kappa(\phi,\psi)$ explicitly as
\begin{align}\label{eq: local formula kappa}
    \kappa(\phi,\psi) = \sum_{i = 1}^n Z_k(\phi)\cdot Z_k(\psi).
\end{align}
To see this we first note that we have the expression
\begin{align*}
    \nabla \phi = \sum_{k=1}^n g(Z_k,\nabla \phi) \cdot Z_k = \sum_{k=1}^n Z_k(\phi)\cdot Z_k
\end{align*}
for $\nabla \phi$ in terms of the basis fields $Z_k$ everywhere on $U$. Inserting this into the definition of $\kappa$ yields
\begin{align*}
    \kappa(\phi,\psi) = \sum_{k=1}^n g(Z_k(\phi)\cdot Z_k,\nabla \psi) = \sum_{k=1}^n Z_k(\phi)\cdot g(Z_k,\nabla \psi) =\sum_{i = 1}^n Z_k(\phi)\cdot Z_k(\psi),
\end{align*}
as desired. 
With this we can now present some important results giving complex-valued harmonic morphisms a very geometric character.

\begin{theorem}[\cite{Baird and Eells,sg sym sps rank 1}]\label{submersive harm morph}
Let $\varphi:(M,g) \to (N^n,h)$ be a horizontally conformal submersion between Riemannian manifolds. If
\begin{enumerate}
    \item $n = 2$ then $\varphi$ is harmonic if and only if $\varphi$ has minimal fibres
    \item $n \geq 3$ then any two of the following conditions imply the other:
    \begin{enumerate}[label = (\roman*)]
        \item $\varphi$ is a harmonic map,
        \item $\varphi$ has minimal fibres
        \item $\varphi$ is horizontally homothetic.
    \end{enumerate}
\end{enumerate}
\end{theorem}
We note that by horizontal homothety, we mean that the dilation $\lambda$ of $\varphi$ is constant along any horizontal curve in $M$. As a special case when the domain is the complex plane $\cn$ we thus have the following.
\begin{corollary}\label{minimal fibres}
Let $\varphi:(M,g) \to \cn$ be a smooth complex-valued function on a Riemannian manifold. Then $\varphi$ is a harmonic morphism if and only if $\varphi$ is horizontally (weakly) conformal with minimal fibres at regular points of $\varphi$.
\end{corollary}
Locally, since the property of being a harmonic morphism is invariant under conformal changes of the codomain, we may always assume in the $n = 2$ case that the codomain is $\cn$.
\begin{corollary}
Let $M = G/K$ be a symmetric space and $\pi: G \to M$ the standard projection $\pi: p \mapsto pK$ and assume $G$ and $M$ are equipped with the standard metrics such that this projection is a Riemannian submersion. Then $\pi$ is a harmonic morphism.
\end{corollary}
\begin{proof}
The submersion $\pi$ already satisfies condition (iii) of Theorem \ref{minimal fibres}. Hence it is sufficient to verify that condition (ii) is also satisfied. We shall instead prove the stronger statement that the fibers are in fact totally geodesic, using Proposition \ref{Lie group symmetric}.
\par
Let $pk_1$ and $pk_{2}$ belong to the fibre over $pK$. Then
\begin{align*}
    s_{pk_{1}}(q) = pk_{1}\cdot k_2^{-1} p^{-1}\cdot pk_{1} =pk_{1}k_2^{-1}k_{1} \in pK,
\end{align*}
so that the fibre $pK$ is totally geodesic in $G$. This concludes the proof.
\end{proof}
The following very useful result can be thought of as a dual to Theorem \ref{minimal fibres} and can be found as Theorem 1.2.8 in \cite{Sigmundur's doctoral thesis}. 
\begin{theorem}[\cite{Sigmundur's doctoral thesis}]\label{harmonic immersion}
Let $m > n$ and $i: (N^n,h) \to (M^m,g)$ be a non-constant weakly conformal immersion with conformal factor $\lambda: N \to \rn_{0}^+$. If 
\begin{enumerate}
    \item $n=2$ then $i$ is a harmonic map if and only if $i(N)$ is minimal in $(M,g)$,
    \item $n\geq 3$, then two of the following imply the other,
    \begin{enumerate}[label = (\roman*)]
        \item $i$ is a harmonic map,
        \item $i(N)$ is minimal in $(M,g)$,
        \item $i$ is homothetic, that is, $\lambda$ is constant.
    \end{enumerate}
\end{enumerate}
\end{theorem}
With this result we are able to show that both the Cartan embedding and the Cartan map are harmonic maps. This is both of independent interest and will make the Cartan map a useful tool in finding eigenfunctions and families on the symmetric spaces.
\begin{theorem}
Let $(G,K,\sigma)$ be a symmetric triple of compact type with $K = G^{\sigma}$ and $\dim(G/K) \geq 2$. Then the Cartan map $\Phi: G \mapsto G$ is harmonic. In other words it satisfies
\begin{align*}
    \tau(\Phi) = 0.
\end{align*}
\end{theorem}
\begin{proof}
By Corollary \ref{Cartan conformal} and Theorem \ref{Cartan totally geodesic}, the Cartan embedding $\hat{\Phi}: G/K \to G$ satisfies (ii) and (iii) of Theorem \ref{harmonic immersion}. Hence $\hat{\Phi}$ is harmonic. Since we have
$$
\Phi = \hat{\Phi}\circ \pi
$$
and the natural projection $\pi: G \to G/K$ is a harmonic morphism, the Cartan map $\Phi$ is harmonic as well.
\end{proof}

This has the following consequence when combined with the composition law. We let $\varphi:G \to \cn$ be a smooth complex-valued function, and denote also by $\varphi$ the restriction of $\varphi$ to the image $\Phi(G)$. Applying the composition law to the function $\varphi\circ \Phi: G \to \cn$ then gives
\begin{align*}
    \tau(\varphi\circ \Phi) = d\varphi(\tau(\Phi))\circ \Phi + \tr \nabla d\varphi(d\Phi,d\Phi) \circ \Phi = \tr \nabla d\varphi(d\Phi,d\Phi) \circ \Phi,
\end{align*}
since $\Phi$ is harmonic. For the second term, when expanded it gives
\begin{align*}
    \tr \nabla d\varphi(d\Phi,d\Phi) \circ \Phi = \left\{\sum_{X \in \mathcal{B}_{\la{g}}} (d\Phi(X))^2(\varphi) - \left(\nab{}{d\Phi(X)}{d\Phi(X)}\right)(\varphi)\right\}\circ \Phi.
\end{align*}
As we have remarked before, the second term in the sum vanishes and for the first term, only the vector fields from $\mathcal{B}_{\la{p}}$ contribute. This then finally gives
\begin{align}\label{eq: tau f o Phi}
    \tau(\varphi\circ \Phi)  = \left\{\sum_{X \in \mathcal{B_{\la{p}}}} (d\Phi(X))^2(\varphi)\right\} \circ \Phi.
\end{align}
We get the similar formula for the conformality operator. Let $\varphi, \psi: G \mapsto \cn$ be smooth. Then
\begin{align}\label{eq: kappa f o Phi}
\kappa(\varphi\circ \Phi, \psi \circ \Phi) = \left\{\sum_{X\in\mathcal{B}_{\la{p}}}d\Phi(X)(\varphi)\cdot d\Phi(X)(\psi)\right\}\circ \Phi.
\end{align}

\begin{remark}\label{basis cartan embedding}
At a point $p \in G$ and for $X \in \la{p}$ we have the expression
\begin{align*}
    d\Phi_p(X_p) = 2(dR_{\sigma(p^{-1})})_p(X_p).
\end{align*}
Since we assume $G$ to be compact and equipped with a bi-invariant metric, the right translations are isometries. Hence, if $\basis_{p}$ is an orthonormal basis for $\la{p}_{p}$, the set
$$
\{ (dR_{\sigma(p^{-1})})_p(X_p) \ | \ X_p \in \basis_{p} \}
$$
is an orthonormal basis for $T_{\Phi(p)} N$ where $N = \textrm{Im}(\Phi)$.
\end{remark}
The next result then follows by this remark.
\begin{theorem}\label{composition relations}
Let $(G,K,\sigma)$ be a symmetric triple of compact type and let $\Phi: G \to G$ be the associated Cartan map. Let $N$ be the image $\Phi(G)$ of $\Phi$ , $\tau$ and $\kappa$ denote the tension field and conformality operator on the Lie group $G$ and let $\tau_{N}$ and $\kappa_{N}$ denote those of $N$. Then if 
$$
\varphi, \psi: G \to \cn$$
are $C^2$ complex-valued functions on $G$, the compositions 
$$
\varphi \circ \Phi,\psi \circ \Phi: G \to \cn
$$
are $C^2$ complex-valued and $K$-invariant functions on $G$. The tension field $\tau$ and conformality operator $\kappa$ then satisfy the relations
\begin{align*}
    \tau(\varphi\circ \Phi) = 4\cdot \tau_{N}(\varphi)\circ \Phi
\end{align*}
and
\begin{align*}
    \kappa(\varphi\circ \Phi,\psi \circ \Phi) = 4\cdot \kappa_{N}(\varphi,\psi)\circ \Phi.
\end{align*}
\end{theorem}
The following will help us in explicitly calculating the tension field and conformality operator on the image of the Cartan embedding.

\begin{remark}\label{expression for basis}
Let $(G,K,\sigma)$ be a symmetric triple of compact type and $\Phi:G \to G$ be the associated Cartan map. Let $N$ be the image $\Phi(G)$ of $\Phi$. By Remark \ref{basis cartan embedding}, given a point $p\in G$ and an orthonormal basis $\basis_p$ for $\la{p}_p$, an orthonormal basis for $T_{\Phi(p)} N$ is given by
$$
\{ (dR_{\sigma(p^{-1})})_p(X_p) \ | \ X_p \in \basis_{p} \}.
$$
If we now fix $X_p \in \basis_p$ and take $X \in \la{p}$ such that $(dL_p)_e(X) = X_p$ we have
\begin{align*}
    (dR_{\sigma(p^{-1})})_p(X_p) &= (dR_{\sigma(p^{-1})})_p(dL_p)_e(dL_{\sigma(p^{-1})})_{\sigma(p)}(dL_{\sigma(p)})_{e}(X)\\
    &= (dL_{\Phi(p)})_e \Ad{\sigma(p)}(X).
\end{align*}
Hence given an orthonormal basis $\basis_{\la{p}}$ for $\la{p}$ we can construct an orthonormal basis for $T_{\Phi(p)}N$ as
\begin{align*}
    \basis_{\Phi(p)} = \{\Ad{\sigma(p)}(X)_{\Phi(p)} \ | \ X \in \basis_{\la{p}} \}.
\end{align*}
\end{remark}
 We conclude the chapter by introducing the iterated tension field $\tau^p$ and $p$-harmonic functions on Riemannian manifolds.
\begin{definition}
Let $(M,g)$ be a Riemannian manifold. We define the iterated Laplace-Beltrami operator (or tension field) $\tau^p$ for an integer $p\geq 0$ recursively by
\begin{enumerate}[label = (\roman*)]
    \item $\tau^{0}(\phi) = \phi$,
    \item $\tau^{k+1}(\phi) = \tau(\tau^{k}(\phi))$,
\end{enumerate}
where $\phi:M \to \cn$ is assumed smooth.
We say that a smooth function $\phi:M \to \cn$ is \emph{$p$-harmonic} if $\tau^p(\phi) = 0$, and \emph{proper $p$-harmonic} if in addition $\tau^{p-1}(\phi)$ is not the zero function.
\end{definition}
In chapter \ref{ch: eigen} we shall see how eigenfamilies and eigenfunctions can be used to generate both explicit harmonic morphisms and explicit $p$-harmonic functions.

\chapter{The Methods of Eigenfunctions and Eigenfamilies}\label{ch: eigen}
In this chapter we describe two closely related and recently developed methods for finding complex-valued harmonic morphisms and $p$-harmonic functions, respectively. These methods both rely on finding joint eigenfunctions of the tension field and the conformality operator.
\par
We will describe how eigenfunctions defined on a symmetric space of compact type induce eigenfunctions on its non-compact dual, and vice versa. We also give a new method for constructing eigenfunctions and eigenfamilies on a Riemannian product $M = M_1\times M_2$ given eigenfunctions and eigenfamilies on the factors $M_1$ and $M_2$. These two constructions, together with the explicit eigenfunctions described in Chapter \ref{ch: results} enable us to construct eigenfunctions on a large class of symmetric spaces.
\par
We shall for the remainder of this section make the assumption that our functions are all real analytic. We make this assumption since it will allow us to use the duality principle described in Chapter \ref{ch: symmetric spaces} to construct eigenfunctions and families also on the non-compact duals. This assumption poses no real restriction for us as  all eigenfunctions found in this thesis are indeed analytic.
\begin{definition}
Let $\phi:(M,g) \to \cn$ be a smooth complex-valued function from a Riemannian manifold. We call $\phi$ an \emph{eigenfunction} if there exist $\lambda,\mu \in \cn$ such that
\begin{align*}
    \tau(\phi) = \lambda\cdot \phi \quad \textrm{and} \quad \kappa(\phi,\phi) =\mu\cdot \phi^2.
\end{align*}
\end{definition}
In other words, we call a function an eigenfunction if it is eigen to both of the operators $\tau$ and $\kappa$.
\begin{definition}
Let $(M,g)$ be a Riemannian manifold. We call a collection 
$$\mathcal{E} = \{\phi_{i}: M \to \cn \ | \ i \in I\}$$ 
of smooth complex-valued functions on $M$ an \emph{eigenfamily} if there exist $\lambda,\mu \in \cn$ such that 
\begin{align*}
    \tau(\phi) = \lambda\cdot \phi \quad \textrm{and} \quad \kappa(\phi,\psi) = \mu \cdot \phi\psi
\end{align*}
hold for all $\phi,\psi \in \mathcal{E}$.
\end{definition}
Given a finite eigenfamily on $(M,g)$ we can construct a collection of eigenfamilies on $M$ by considering homogeneous polynomials.
\begin{theorem}[\cite{p-harm on semi-riem}]
Let $(M,g)$ be a Riemannian manifold and $\mathcal{E} = \{\phi_1,\dots \phi_n\}$ be a finite eigenfamily on $M$ with eigenvalues $\lambda,\mu \in \cn$ for $\tau$ and $\kappa$, respectively. Then the set of complex homogeneous polynomials of degree $d$
\begin{align*}
    \mathcal{H}_\mathcal{E}^d = \{ P: M \to \cn \ | \ P \in [\phi_1\dots,\phi_n], P(\alpha \cdot \phi) = \alpha^p\cdot P(\phi), \alpha \in \cn\}
\end{align*}
is an eigenfamily on $M$ with
\begin{align*}
    \tau(P) = (d\lambda + d(d-1)\mu)\cdot P \quad \textrm{and}\quad \kappa(P,Q) = d^2\mu\cdot PQ
\end{align*}
for $P,Q \in \mathcal{H}_{\mathcal{E}}^d$.
\end{theorem}
We shall first describe how eigenfamilies can be used to generate complex-valued harmonic morphisms.
\begin{theorem}[\cite{Gud-Sak}]\label{thm: method of eigenfams}
Let $(M,g)$ be a Riemannian manifold and
$\mathcal{E} = \{\phi_1, \dots,\phi_n\}$
be an eigenfamily of complex-valued functions on $M$. If $P,Q: \cn^n \to \cn$ are linearly independent homogeneous polynomials of the same positive degree then the quotient
\begin{align*}
    \varphi = \frac{P(\phi_1, \dots, \phi_n)}{Q(\phi_1, \dots, \phi_n)}
\end{align*}
is a non-constant complex-valued harmonic morphism on the open and dense subset
\begin{align*}
    \{p \in M \ | \ Q(\phi_1(p),\dots,\phi_{n}(p)) \neq 0\}.
\end{align*}
\end{theorem}
Next we describe how eigenfunctions can be used to generate $p$-harmonic functions.
\begin{theorem}[\cite{Gud-Sob}]
Let $(M,g)$ be a Riemannian manifold and $\phi:M \to \cn$ be a complex-valued eigenfunction on $M$ such that the eigenvalues $\lambda,\mu$ for $\tau$ and $\kappa$, respectively, are not both zero. Then for any integer $p > 0$, the non-vanishing function
\begin{align*}
    \Phi_{p}: W = \{x \in M \ | \ \phi(x) \notin (-\infty,0]\} \to \cn
\end{align*}
with
\begin{align*}
    \Phi_p(x) = 
    \begin{cases}
    c_{1}\cdot \log(\phi(x))^{p-1}, & \textrm{if } \mu = 0, \lambda \neq 0\\
    c_{1}\cdot\log(\phi(x))^{2p-1}+c_{2}\cdot \log(\phi(x))^{2p-2}, & \textrm{if } \mu \neq 0, \lambda = \mu\\
    c_{1}\phi(x)^{1-\frac{\lambda}{\mu}}\log(\phi(x))^{p-1} + c_{2}\cdot \log(\phi(x))^{p-1}, & \textrm{if } \mu \neq 0, \lambda \neq 0 
    \end{cases}
\end{align*}
is a proper $p$-harmonic function. Here $c_1$ and $c_2$ are arbitrary complex coefficients not both equal to zero.
\end{theorem}
Motivated by these remarkable results we shall therefore seek eigenfamilies on the classical compact symmetric spaces. We may transform this eigenvalue problem on the symmetric space $G/K$ into an eigenvalue problem with a constraint on the Lie group $G$.

\begin{lemma}[\cite{Gud-Sif-Sob}]\label{lem: induced functions}
    Let $M = G/K$ be a symmetric space, $U \subset G/K$ an open subset of $M$ and $\hat{\phi},\hat{\psi}: U \to \cn$ be locally defined real analytic functions. Let $\phi,\psi:G \to \cn$ be the $K$-invariant compositions of $\hat{\phi}$ and $\hat{\psi}$ with the natural projection $\pi: G \to G/K$. We denote by $\hat{\tau}$, $\tau$, $\hat{\kappa}$ and $\kappa$ the tension field and conformality operators on $G/K$ and $G$ respectively. Then these operators satisfy
    \begin{align*}
        \hat{\tau}(\hat{\phi}) = \tau(\phi)\circ \pi
    \end{align*}
    and 
    \begin{align*}
        \hat{\kappa}(\hat{\phi},\hat{\psi}) = \kappa(\phi,\psi) \circ \pi.
    \end{align*}
\end{lemma}
\begin{proof}
To prove this result one uses that $\pi$ is a harmonic morphism and a Riemannian submersion. The proof is very similar as that of Theorem \ref{composition relations} and thus will not be presented here.
\end{proof}
We let $M = G/K$ be a symmetric space, and
\begin{align*}
    \hat{\mathcal{E}}_{\lambda,\mu} = \{ \hat{\phi}: U\subset M \to \cn \}
\end{align*}
be a locally defined eigenfamily on $M$ with eigenvalues $\lambda$ and $\mu$ for the tension field $\hat{\tau}$ and the conformality operator $\hat{\kappa}$, respectively. By Lemma \ref{lem: induced functions}, the family
\begin{align*}
    \mathcal{E}_{\lambda,\mu} = \{ \phi = \hat{\phi}\circ \pi: \pi^{-1}(U) \to \cn \ | \ \hat{\phi} \in \hat{\mathcal{E}}_{\lambda,\mu} \} 
\end{align*}
is a locally defined eigenfamily on $G$ with the same eigenvalues $\lambda$ and $\mu$, consisting of $K$-invariant functions.
\par
Conversely, given a locally defined eigenfamily $\mathcal{E}$ of $K$-invariant functions, we define for each $\phi \in \mathcal{E}$ a function $\hat{\phi}: G/K \to \cn$ by 
$$
\hat{\phi}(pK) = \phi(p).
$$
Then $\phi = \hat{\phi}\circ \pi$ and hence, by Lemma \ref{lem: induced functions}, the set $\hat{\mathcal{E}}$ of such induced functions is a locally defined eigenfamily on the symmetric space $G/K$.
Thus, there is a one-to-one correspondence between $K$-invariant eigenfunctions and eigenfamilies on $G$, and eigenfunctions and eigenfamilies on $G/K$.
\par
In Theorem \ref{thm: eigenfunctions on products} and Theorem \ref{thm: eigenfamilies on products} we give a new method for constructing eigenfunctions and eigenfamilies on the Riemannian product $(M_1 \times M_2,g)$ from eigenfunctions and families on the factors $(M_1,g_1)$ and $(M_2,g_2)$.
\begin{theorem}\label{thm: eigenfunctions on products}
Let $(M_1,g_1)$ and $(M_2,g_2)$ be Riemannian manifolds and let $N = M_1 \times M_2$ be their Riemannian product, with product metric $g$. Let $\tau_1$, $\tau_2$ and $\tau$ and $\kappa_1$, $\kappa_2$ and $\kappa$ be the tension fields and conformality operators on $M_1$, $M_2$ and $N$, respectively. Suppose that $\phi_1$ and $\phi_2$ are eigenfunctions on $M_1$ and $M_2$ satisfying
$$
\tau_1(\phi_1) = \lambda_1 \cdot \phi_1, \quad \kappa_1(\phi_1,\phi_1) = \mu_1 \cdot \phi_1^2
$$
and
$$
\tau_2(\phi_2) = \lambda_2 \cdot \phi_2, \quad \kappa_2(\phi_2,\phi_2) = \mu_2 \cdot \phi_2^2,
$$
respectively. Define $\psi: N \to \cn$ by
\begin{align*}
    \psi(p_1,p_2) = \phi_1(p_1)\cdot \phi_2(p_2).
\end{align*}
Then  $\psi$ is an eigenfunction on $N$ with
\begin{align*}
    \tau(\psi) = (\lambda_1 + \lambda_2)\cdot \psi \quad \mathrm{and} \quad \kappa(\psi,\psi) = (\mu_1 + \mu_2)\cdot \psi^2.
\end{align*}
\end{theorem}
\begin{proof}
Recall that at each point $(p_1,p_2)$ of $N$ we have the orthogonal decomposition
$$
T_{(p_1,p_2)}N \cong T_{p_1} M_1 \oplus T_{p_2}M_2.
$$
Take $p = (p_1,p_2) \in N$ and let $\{X_1,\dots, X_{m_1}, Y_1, \dots, Y_{m_2}\}$ be a local orthonormal frame for $TN$ on an open subset $U = U_1 \times U_2$ containing $(p_1,p_2)$ such that $\{X_1,\dots, X_{m_1}\}$ and $\{Y_1,\dots, Y_{m_2}\}$ are local orthonormal frames for $TM_1$ and $TM_2$ respectively. Then $\psi$ satisfies
\begin{align*}
    X_{k}(\psi)(p) &= X_{k}(\phi_1)(p_1)\cdot \phi_2(p_2),\\
    X_{k}^2(\psi)(p) &= X_{k}^2(\phi_1)(p_1)\cdot \phi_2(p_2),\\
    \left(\nab{}{X_k}{X_k}\right)(\psi)(p) &= \left(\nab{}{X_k}{X_k}\right)(\phi_1)(p_1)\cdot \phi_2(p_2),\\
    Y_{l}(\psi)(p) &= \phi_1(p_1)\cdot Y_{l}(\phi_2)(p_2),\\
    Y_{l}^2(\psi)(p) &= \phi_1(p_1)\cdot Y_{l}^2(\phi_2)(p_2)
\end{align*}
and
\begin{align*}
    \left(\nab{}{Y_{l}}{Y_{l}}\right)(\psi)(p) &= \phi_1(p_1)\cdot \left(\nab{}{Y_{l}}{Y_{l}}\right)(\phi_2)(p_2).
\end{align*}
Using the formula \eqref{eq:local tension} for the tension field $\tau$ we obtain
\begin{align*}
    &\tau(\psi)(p)\\
    &= \sum_{k=1}^{m_1} \big( X_{k}^2(\psi)(p) - \left(\nab{}{X_k}{X_k}\right)(\psi)(p)\big) + \sum_{l=1}^{m_2} \big(Y_{l}^2(\psi)(p) - \left(\nab{}{Y_{l}}{Y_{l}}\right)(\psi)(p)\big)\\
    &= \sum_{k=1}^{m_1}\big(X_{k}^2(\phi_1)(p_1)\cdot \phi_2(p_2) - \left(\nab{}{X_k}{X_k}\right)(\phi_1)(p_1)\cdot \phi_2(p_2)\big)\\
    &\sed + \sum_{l=1}^{m_2} \big(\phi_1(p_1)\cdot Y_{l}^2(\phi_2)(p_2) - \phi_1(p_1)\cdot \left(\nab{}{Y_{l}}{Y_{l}}\right)(\phi_2)(p_2)\big)\\
    &= \tau_1(\phi_1)(p_1) \cdot \phi_2(p_2) + \phi_1(p_1) \cdot \tau_2(\phi_2)(p_2)\\
    &= \lambda_1 \cdot \phi_1(p_1)\phi_2(p_2) + \lambda_2 \cdot \phi_1(p_1)\phi_2(p_2)\\
    &= (\lambda_1 + \lambda_2)\cdot \psi(p).
\end{align*}
For $\kappa$, applying the formula \eqref{eq: local formula kappa} yields
\begin{align*}
    &\kappa(\psi,\psi)(p)\\
    &= \sum_{k=1}^{m_1}\big((X_k(\psi)(p))^2\big) + \sum_{l=1}^{m_2} \big((Y_l(\psi)(p))^2\big)\\
    &= \sum_{k=1}^{m_1}\big((X_{k}(\phi_1)(p_1))^2\cdot \phi_2(p_2)^2 \big) + \sum_{l=1}^{m_2}\big( \phi_1(p_1)^2\cdot (Y_{l}(\phi_2)(p_2))^2\big)\\
    &= \kappa_1(\phi_1,\phi_1)(p_1) \cdot \phi_2(p_2)^2 + \phi_1(p_1)^2\cdot \kappa_{2}(\phi_2,\phi_2)(p_2)\\
    &= (\mu_1 + \mu_2)\cdot \psi(p)^2.
\end{align*}
Since the choice of $p$ was arbitrary we are done.
\end{proof}
This construction extends readily to eigenfamilies. 
\begin{theorem}\label{thm: eigenfamilies on products}
    Let $(M_1,g_1)$ and $(M_2,g_2)$ be Riemannian manifolds and let $N = M_1 \times M_2$ be their Riemannian product, with product metric $g$. Let $\tau_1$, $\tau_2$ and $\tau$ and $\kappa_1$, $\kappa_2$ and $\kappa$ be the tension fields and conformality operators on $M_1$, $M_2$ and $N$, respectively. Suppose that $\mathcal{E}_1$ and $\mathcal{E}_2$ are eigenfamilies on $M_1$ and $M_2$ with eigenvalues $\lambda_1$, and $\mu_1$ and $\lambda_2$ and $\mu_2$ respectively. Then the set
    \begin{align*}
        \mathcal{E} = \{ \psi: N \to \cn, \ \psi: (p_1,p_2) \mapsto \phi_1(p_1)\cdot \phi_2(p_2) \ | \ \phi_1 \in \mathcal{E}_1, \phi_2 \in \mathcal{E}_2\}
    \end{align*}
    is an eigenfamily on the product space $N$, with
    \begin{align*}
        \tau(\phi) = (\lambda_1 + \lambda_2)\cdot \phi \quad \mathrm{and} \quad \kappa(\phi,\psi) = (\mu_1 + \mu_2) \cdot \phi\psi
    \end{align*}
for all $\phi,\psi \in \mathcal{E}$.
\end{theorem}
\begin{proof}
The proof is nearly identical to that of Theorem \ref{thm: eigenfunctions on products}.
\end{proof}

\begin{corollary}\label{cor: harmonic morphisms on products}
    Let $(M_1,g_1)$ and $(M_2,g_2)$ be Riemannian manifolds and let $N = M_1 \times M_2$ be their Riemannian product, with product metric $g$. Suppose that $\phi_1: M_1 \to \cn$ and $\phi_2: M_2 \to \cn$ are harmonic morphisms on $M_1$ and $M_2$ respectively. Then the function $\phi: N \to \cn$ given by
    \begin{align*}
        \phi(p_1,p_2) = \phi_1 (p_1)\cdot \phi_2(p_2)
    \end{align*}
    is a harmonic morphism on the product space $N$.
\end{corollary}
\begin{proof}
    This is a special case of Theorem \ref{thm: eigenfunctions on products}.
\end{proof}
\par
We conclude this chapter by describing how eigenfamilies on symmetric spaces of compact type induce eigenfamilies on their non-compact duals, and vice versa. These constructions were first introduced in \cite{Gud-Sve} and further developed for our setting in \cite{Gud-Mon-Rat}.
\par
Let $(G,K)$ be a symmetric pair with $M = G/K$ simply connected and let $\la{g} = \la{k} \oplus \la{p}$ be the associated Cartan decomposition of the Lie algerbra $\la{g}$ of $G$. Let $G^{\cn}$ be the unique simply connected complex Lie group with Lie algebra $\la{g}^{\cn}$. Since $\la{g}$ is a subalgebra of $\la{g}^{\cn}$ there exists a Lie subgroup $G' \subset G^{\cn}$ with Lie algebra $\la{g}$ and similarly a subgroup $K' \subset G' \subset G^{\cn}$ with Lie algebra $\la{k}$. As we are ultimately interested in the symmetric space $M = G/K$ we may assume $G = G'$ and $K = K'$. We let as before $M^* = G^*/K^*$ be the dual space of $M$. Here we take $G^*$ to be the Lie subgroup of $G^{\cn}$ with Lie algebra
$$
\la{g}^* = \la{k} \oplus i\la{p} 
$$
and take $K^* = K' = K$.
\par
Let 
$$
\phi: U/K \subset M \to \cn
$$
be a locally defined real analytic function on the open subset $U/K$ of $M$ and denote by $\hat{\phi}$ the corresponding real analytic $K$-invariant function on the open subset $U$ of $G$. We may without loss of generality assume that the identity element $e$ of $G$ is contained in $U$ as we may always take a composition of $\hat{\phi}$ with the left translation by some element of $G$ to obtain a real analytic function on a neighbourhood of the identity. The function $\hat{\phi}$ can then be extended uniquely as a holomorphic $K$-invariant function
$$
\hat{\phi}^{\cn}: V \to \cn
$$
on some open neighbourhood $V$ of the identity in $G^{\cn}$. We set $W = G^*\cap V$ and denote by $\hat{\phi}^*$ the real analytic and $K$-invariant restriction of $\hat{\phi}^{\cn}$ to $W$. The corresponding real analytic function
$$
\phi^*: W/K \subset M^* \to \cn
$$
is what we shall call the \emph{dual function} of $\phi$. We shall also refer to the function $\hat{\phi}^*$ as the dual function of $\hat{\phi}$.

\begin{theorem}[\cite{Gud-Sve, Gud-Mon-Rat}]\label{thm: dual eigenfunctions}
    Let $(G,K)$ be an irreducible symmetric pair of compact type with $G$ semisimple, and let $(G^*,K)$ be its non-compact dual. We equip $G$ and $G^*$, respectively, with the left-invariant metrics $g$ and $g^*$ induced by the $\Ad{K}$-invariant inner products described in Remark \ref{rem: natural metrics}. We denote by $\tau$ and $\tau^*$ the tension fields on $M = G/K$ and $M^* = G^*/K$ and by $\kappa$ and $\kappa^*$ their respective conformality operators. Suppose that $\phi: U/K \to \cn$ is a locally defined real-analytic eigenfunction on the compact symmetric space $M$, satisfying
    $$
    \tau(\phi) = \lambda \cdot \phi \quad \textrm{and} \quad \kappa(\phi,\phi) = \mu \cdot \phi^2.
    $$
    Then the dual function $\phi^*: W/K \subset M^* \to \cn$ is a locally defined real-analytic eigenfunction on $M^*$ such that
    $$
    \tau^*(\phi^*) = -\lambda \cdot \phi^* \quad \textrm{and} \quad \kappa^*(\phi^*,\phi^*) = -\mu \cdot (\phi^*)^2.
    $$
\end{theorem}
\begin{proof}
    We denote by $\hat{\phi}$ and $\hat{\phi}^*$ the $K$-invariant real analytic functions on $M$ and $M^*$ corresponding to $\phi$ and $\phi^*$, respectively. We shall show that $\hat{\phi}^*$ is an eigenfunction on $W \subset G^*$ with eigenvalue $-\lambda$ for the tension field $\hat{\tau}^*$ and eigenvalue $-\mu$  for the conformality operator $\hat{\kappa}^*$. Let $\{X_1,\dots, X_k\}$ and $\{Y_1,\dots, Y_l\}$ be bases for $\la{k}$ and $\la{p}$ respectively, such that $\{X_1,\dots, X_k,Y_1, \dots, Y_l\}$ is an orthonormal basis for $\la{g}$ and $\{X_1,\dots,X_k,iY_1,\dots, Y_l\}$ is an orthormal basis for $\la{g}^*$. We then have
    \begin{align}\label{eq: tau on dual func 1}
        \hat{\tau}^*(\hat{\phi}^*)
        = \sum_{r=1}^k X_r^2(\hat{\phi}^*) - (\nab{}{X_r}{X_r})(\hat{\phi}^*) + \sum_{r=1}^l (iY_r)^2(\hat{\phi}^*) - (\nab{}{iY_r}{iY_r})(\hat{\phi}^*)
    \end{align}
    for the tension field acting on $\hat{\phi}^*$. The terms involving $X_r^2$ vanish by $K$-invariance and one shows, in a similar way as in the proof of Proposition \ref{Levi-Civita on compact}, that 
    $$\nab{}{X}{X} = \nab{}{iY}{iY} = 0$$
    for all vector fields $X \in \la{k}$ and $iY \in i\la{p}$. Hence, Equation \eqref{eq: tau on dual func 1} reduces to the formula
    \begin{align}\label{eq: tau on dual func 2}
        \hat{\tau}^*(\hat{\phi}^*) = \sum_{r=1}^l (iY_r)^2(\hat{\phi}^*).
    \end{align}
    Now if $\psi: U \to \cn$ is any real-analytic function on $U$, denoting the operation of taking the dual by $*$, we claim that $(iY)(\psi^*) = i\cdot (Y(\psi))^*$. To see this, consider the locally defined holomorphic functions $(iY)(\psi^{\cn})$ and $i\cdot (Y(\psi))^{\cn}$. The restrictions of these functions to $U \subset G$ clearly agree and are both equal to $i \cdot Y(\psi)$. Then by the definition of the dual function, and by the uniqueness of holomorphic extensions, we must have $(iY)(\psi^*) = i\cdot (Y(\psi))^*$ as desired. Inserting into the formula \eqref{eq: tau on dual func 2} yields
    \begin{align*}
        \hat{\tau}^*(\hat{\phi}^*) &= \sum_{r=1}^l (iY_r)^2(\hat{\phi}^*)\\
        & = \sum_{r=1}^l (iY_r)(i \cdot  (Y_r(\hat{\phi}))^*)\\
        & = \sum_{r=1}^l i^2 \cdot (Y_r^2 (\hat{\phi}))^*\\
        & = - \left( \sum_{r=1}^l Y_r^2(\hat{\phi})  \right)^*\\
        & = - (\hat{\tau}(\hat{\phi}))^*\\
        & = -\lambda \cdot \hat{\phi}^*,
    \end{align*}
    where $\hat{\tau}$ is the tension field on $G$. For the conformality operator we similarily have
    \begin{align*}
        \hat{\kappa}^*(\hat{\phi}^*,\hat{\phi}^*) &= \sum_{r = 1}^l (iY_r)(\hat{\phi}^*)^2\\
        & = \sum_{r = 1}^l i^2 ((Y_r(\hat{\phi})^*)^2\\
        & = - \left(\sum_{r = 1}^l (Y_r(\hat{\phi}))^2\right)^*\\
        & = - \hat{\kappa}(\hat{\phi},\hat{\phi})^*\\
        & = -\mu \cdot (\hat(\phi)^*)^2. 
    \end{align*}
    This then concludes the proof.
\end{proof}
From the construction of eigenfunctions on product spaces, the following result is then immediate.
\begin{corollary}\label{cor: harmonic morphism on M times M*}
    Let $M$ be an irreducible simply connected symmetric space of compact type, let $M^*$ be its non-compact dual and let both $M$ and $M^*$ be equipped with the natural metrics induced by the inner products in Remark \ref{rem: natural metrics}. Let $\phi: U \subset M \to \cn$ be a locally defined real-analytic eigenfunction on $M$ and $\phi^*: W \subset M^* \to \cn$ its dual eigenfunction on $M^*$. Then the function $\psi: U\times W \subset N = M\times M^* \to \cn$ defined by
    \begin{align*}
        \psi(p_1,p_2) = \phi(p_1)\cdot \phi^*(p_2)
    \end{align*}
    is a complex-valued harmonic morphism on $N = M\times M^*$.
\end{corollary}
\chapter{The Classical Symmetric Spaces and Their Involutions}\label{ch: results}
In this chapter we will describe all the classical irreducible symmetric spaces of compact type. We will in each case give the explicit Cartan involution $\sigma$ and hence get an explicit formula for the Cartan embedding $\hat{\Phi}$. We then use the techniques developed in Chapters \ref{ch: harmonic morphisms} and \ref{ch: eigen} to construct eigenfunctions and, in some cases, eigenfamilies on these symmetric spaces $G/K$. This is achieved by first using the Cartan map to construct $K$-invariant eigenfunctions and families on $G$, which then induce eigenfunctions and eigenfamilies on the quotient.

First we establish some notation. We will use $J_m$ and $I_{m,n}$ to denote the matrices
$$
J_{n} =\begin{pmatrix}
0 & I_{n}\\
-I_{n} & 0
\end{pmatrix} \quad \text{and} \quad I_{m,n} = \begin{pmatrix}
I_{m} & 0\\
0 & -I_{n}
\end{pmatrix},
$$
respectively. The subscript on $J_{n}$ will usually be supressed. We will use the notation $E_{j\alpha}$, $1 \leq j,\alpha \leq n$ for the $n \times n$ matrix satisfying $(E_{j\alpha})_{k\beta} = \delta_{jk}\delta_{\alpha\beta}$. It is straightforward to verify that the matrices $E_{j\alpha}$, $1 \leq j,\alpha \leq n$ form an orthonormal basis for $\gl{n}{\cn}$ over $\cn$ and for $\gl{n}{\rn}$ over $\rn$. We shall also denote by $\{e_{1}, \dots, e_{n}\}$ the standard basis of $\cn^n$, which we view as a collection of row vectors. We then have $e_{j}^t e_{\alpha} = E_{j\alpha}$ and $e_j A e_{\alpha}^t = A_{j\alpha}$ for any $n\times n$ complex matrix $A$.

Furthermore, we define 
$$
Y_{rs} = \tfrac{1}{\sqrt{2}}\cdot(E_{rs}-E_{sr}), \quad X_{rs} = \tfrac{1}{\sqrt{2}}\cdot(E_{rs}+E_{sr}) \quad \text{and} \ D_{t} = E_{tt},
$$
for $1 \leq r < s \leq n$ and $1 \leq t \leq n$. These matrices can be used to construct explicit orthonormal bases to the Lie algebras of the classical compact Lie groups.
\begin{proposition}\label{basis so}
Let $n >1$ be an integer and set
\begin{align*}
    \basis = \{ Y_{rs} \ | \ 1\leq r<s\leq n\}.
\end{align*}
Then $\basis$ is an orthonormal basis for the Lie algebra $\so{n}$ equipped with its standard bi-invariant metric.
\end{proposition}
\begin{proposition}\label{basis u}
Let $n>1$ be an integer and set
\begin{align*}
    \basis_1 = \{ Y_{rs} \ | \ 1 \leq r<s \leq n \}, \quad \basis_2 = \{ iX_{rs} \ | \ 1 \leq r<s \leq n \}
\end{align*}
and
\begin{align*}
    \basis_3 = \{ iD_{t} \ | \ 1 \leq t \leq n \}.
\end{align*}
Then $\basis = \basis_1 \cup \basis_2 \cup \basis_3$ is an orthonormal basis for the Lie algebra $\u{n}$ equipped with its standard bi-invariant metric.
\end{proposition}
\begin{remark}[\cite{Gud-Sob}]\label{basis for su}
We will not write down an explicit orthonormal basis for the Lie algebra $\su{n}$. Instead we note that the matrix
$$
X = \tfrac{i}{\sqrt{n}}I_{n} \in \u{n}
$$
is of unit length and generates the orthogonal complement of $\su{n}$ in $\u{n}$.
This can also be found in \cite{Gud-Sob}. 
\end{remark}
\begin{proposition}\label{basis sp}
Let $n > 1$ be an integer and set
\begin{align*}
    \basis = &\bigg\{\frac{1}{\sqrt{2}}
    \begin{pmatrix}
    Y_{rs} & 0\\
    0 & Y_{rs}
    \end{pmatrix},
    \frac{1}{\sqrt{2}}
    \begin{pmatrix}
    iX_{rs} & 0\\
    0 & -iX_{rs}
    \end{pmatrix},
    \frac{1}{\sqrt{2}}
    \begin{pmatrix}
    iD_{t} & 0\\
    0 & -iD_{t}
    \end{pmatrix},\\
    &\quad
    \frac{1}{\sqrt{2}}
    \begin{pmatrix}
    0 & X_{rs}\\
    -X_{rs} & 0
    \end{pmatrix},
    \frac{1}{\sqrt{2}}
    \begin{pmatrix}
    0 & iX_{rs}\\
    iX_{rs} & 0
    \end{pmatrix},
    \frac{1}{\sqrt{2}}
    \begin{pmatrix}
    0 & D_t\\
    -D_t & 0
    \end{pmatrix},\\
    &\quad
    \frac{1}{\sqrt{2}}
    \begin{pmatrix}
    0 & iD_t\\
    iD_t & 0
    \end{pmatrix}\ \bigg| \ 1 \leq r<s\leq n, 1\leq t \leq n\bigg\}.
\end{align*}
Then $\basis$ is a basis for the Lie algebra $\sp{n}$.
\end{proposition}
The following easily verified identities will also be useful.
\begin{lemma}\label{square sum relation}
Let $n > 1$ be an integer. Then we have
\begin{align*}
    \sum_{\mathclap{1\leq r < s \leq n}} Y_{rs}^2 = -\tfrac{n-1}{2}\cdot I_{n}, \quad \sum_{\mathclap{1\leq r < s \leq n}} X_{rs}^2 = \tfrac{n-1}{2}\cdot I_{n}, \quad \sum_{t=1}^{n}D_t^2 = I_{n}.
\end{align*}
\end{lemma}

\section[The Symmetric Spaces \textbf{SU}(n), \textbf{SO}(n) and \textbf{Sp}(n)]{The Symmetric Spaces $\SU{n}$, $\SO{n}$ and $\Sp{n}$}.
When viewed as symmetric spaces we will identify the Lie groups $\SO{n}, \SU{n}$ and $\Sp{n}$ with $${\SU{n}\times\SU{n}}/{\Delta}, \ \SO{n}\times\SO{n}/{\Delta} \ \ \text{and} \ \ \ {\Sp{n}\times\Sp{n}}/{\Delta},$$ respectively. Here $\Delta$ denotes the diagonal subgroups consisting of elements of the form $(k,k)$.
Let $G$ be one of the Lie groups $\SO{n}$, $\SU{n}$ or $\Sp{n}$, and take the involution $\sigma$ of $G\times G$ to be the transposition of the two factors
\begin{align*}
    \sigma: (p,q) \mapsto (q,p).
\end{align*}
It is straightforward in this case to verify that $\sigma$ is an involutive automorphism of $G\times G$. Furthermore we have $\sigma(p,q) = (p,q)$ if and only if $p = q$, and hence the fixed point set of $\sigma$ is $\Delta$. The Cartan embedding corresponding to $\sigma$ is then
\begin{align*}
    \Phi: (p,q) \mapsto (p,q)(q^{-1}, p^{^-1}) = (pq^{-1}, (pq^{-1})^{-1})
\end{align*}
From this we see that the image of the map $\Phi$ is
\begin{align*}
    \text{Im}(\Phi) = \{(a,a^{-1}) \ | \ a \in G \}.
\end{align*}

\section[The Symmetric Space \textbf{U}(m+n)/\textbf{U}(m) x \textbf{U}(n)]{The Symmetric Space $\U{m+n}/\U{m}\times \U{n}$}
In this section we describe the complex Grassmannians $M = \U{m+n}/\U{m}\times \U{n}$ as symmetric spaces and give an explicit formula for the Cartan map. We then apply the tools and techniques from Chapters \ref{ch: harmonic morphisms} and \ref{ch: eigen} to construct eigenfamilies on these spaces, which can then be used to yield a great variety of complex-valued $p$-harmonic functions, as well as complex-valued harmonic morphisms.
\begin{example}\label{ex: SU(n+m)/S(U(n)*U(m))}
Let $m,n$ be positive integers and set 
$$M = \U{m+n}/\U{m}\times \U{n}.$$
Here the subgroup $\U{m}\times \U{n}$ consists of the block matrices
\begin{align*}
    \left\{ 
    \begin{pmatrix}
    z^{(11)} & 0\\
    0 & z^{(22)}
    \end{pmatrix} \in \U{m+n} \ \middle| \  z^{(11)} \in \U{m}, z^{(22)} \in \U{n} 
    \right\}.
\end{align*}
We then set $\sigma(z) = I_{m,n} z I_{m,n}$ for $z \in \U{m+n}$. Since $I_{m,n}\in \U{m+n}$ the map $\sigma$ is an automorphism of $\U{m+n}$ and it is clearly involutive. Further, upon writing $z \in \U{m+n}$ as the $(m+n)\times (m+n)$ block matrix
\begin{align*}
    z = \begin{pmatrix}
    z^{(11)} & z^{(12)}\\
    z^{(21)} & z^{(22)}
    \end{pmatrix}
\end{align*}
we see that
\begin{align*}
    \sigma(z) = \begin{pmatrix}
    z^{(11)} & -z^{(12)}\\
    -z^{(21)} & z^{(22)}
    \end{pmatrix}
\end{align*}
and hence the elements fixed by $\sigma$ are precisely those belonging to $\U{m}\times \U{n}$. Thus we have found a symmetric triple for the symmetric space $M = \U{m+n}/\U{m}\times \U{n}$.
The corresponding Cartan embedding is
\begin{align*}
    \Phi: z \mapsto zI_{m,n}\bar{z}^t I_{m,n} .
\end{align*}
One can easily show that we may also write $M$ as $M =\SU{m+n}/\mathrm{\bf S}(\U{m}\times \U{n})$ which is then easily shown to be of compact type. This symmetric space is called a \emph{complex Grassmannian} and can be thought of as the space of complex $m$-dimensional subspaces of $\cn^{m+n}$. An important special case is that of $n=1$ in where we obtain the complex projective space $M = \U{m+1}/\U{m}\times \U{1} \cong \cp^n$.
\end{example}
\begin{lemma}\label{basis U(m+n)/U(m)*U(n)}
Let $\u{m+n} = \la{k} \oplus \la{p}$ be the Cartan decomposition corresponding to the symmetric triple $(\U{m+n},\U{m}\times \U{n},\sigma)$ of Example \ref{ex: SU(n+m)/S(U(n)*U(m))}. Then an orthonormal basis for $\la{p}$ is given by
\begin{align*}
    \basis_{\la{p}} = \{ Y_{rs}, iX_{rs} \ | \ 1\leq r \leq m, m+1 \leq s \leq m+n\}
\end{align*}
\end{lemma}
\begin{proof}
The set $\basis_{\la{p}}$ is a strict subset of the orthonormal basis $\basis$ for $\u{m+n}$ given in Proposition \ref{basis u}. It is then easily verified that $d\sigma(Z) = -Z$ for $Z \in \basis_{\la{p}}$ and $d\sigma(Z) = Z$ for $Z \in \basis \setminus \basis_{\la{p}}$. This concludes the proof
\end{proof}
We will now make use of Theorem \ref{composition relations} to construct eigenfamilies on the symmetric space $M$ of the previous example \ref{ex: SU(n+m)/S(U(n)*U(m))}. To this end, for $1 \leq j,\alpha \leq m+n$ we define the maps $\psi_{j\alpha}: \U{m+n} \to \cn$ by
\begin{align*}
    \psi_{j\alpha}:z \mapsto \tfrac{1}{2}\cdot (zI_{m,n}\bar{z}^t + I_{m+n})_{j\alpha}.
\end{align*}
Using that $z \in \U{m+n}$ one can then with simple computations verify that
\begin{align*}
    \psi_{j\alpha}(z) = \sum_{r=1}^{m}z_{jr}\bar{z}_{\alpha r}.
\end{align*}
It is also straightforward to verify that we may write $\psi_{j\alpha}$ on the form $\psi_{j\alpha} = f_{j\alpha} \circ \Phi$ where $f: \U{m+n} \to \cn$ is given by
\begin{align}\label{eq: comp rel U(m+n)}
    f_{j\alpha}: z \mapsto \tfrac{1}{2}\cdot(z\cdot I_{m,n} + I_{m+n})_{j\alpha} 
\end{align}
and $\Phi: \U{m+n} \to \U{m+n}$ is the Cartan map. Hence, the functions $\psi_{j\alpha}$ are $\U{m}\times\U{n}$-invariant and we can use Theorem \ref{composition relations} to prove the following Lemma.
\begin{lemma}[\cite{Gud-Gha 2}]\label{tau kappa c grass}
The tension field $\tau$ and conformality operator $\kappa$ on the unitary group $\U{m+n}$ satisfy the relations
\begin{align*}
    \tau(\psi_{j\alpha}) = -2(m+n)\cdot \psi_{j\alpha} + 2m \cdot \delta_{j\alpha}
\end{align*}
and
\begin{align*}
    \kappa(\psi_{j\alpha},\psi_{k\beta}) = -2\cdot \psi_{j\beta}\psi_{k\alpha} + \psi_{j\beta} \delta_{k\alpha} + \psi_{k\alpha}\delta_{j\beta}.
\end{align*}
\end{lemma}
\begin{proof}
By Theorem \ref{composition relations} to compute $\tau(\psi_{j\alpha})(z)$ and $\kappa(\psi_{j\alpha},\psi_{k\beta})(z)$ we need only to compute $\tau_{N}(f_{j\alpha})(\Phi(z))$ and $\kappa_{N}(f_{j\alpha},f_{k\beta})(\Phi(z))$ where $\tau_{N}$ and $\kappa_{N}$ are the tension field and conformality operators on $N = \Phi(\U{m+n})$, respectively. By Lemma \ref{basis U(m+n)/U(m)*U(n)} and Remark \ref{expression for basis} a basis for $T_{\Phi(z)} N$ is given by
\begin{align*}
    \basis = \{  \Ad{\sigma(z)}(Y_{rs})_{\Phi(z)}, i\Ad{\sigma(z)}(X_{rs})_{\Phi(z)} \ | \ 1\leq r \leq m, m+1 \leq s \leq m+n\}.
\end{align*}
We have
\begin{align*}
    \Ad{\sigma(z)}(Y_{rs})(f_{j\alpha})(\Phi(z)) &= \dtatzero f_{j\alpha}(\Phi(z)\cdot \exp(t\Ad{\sigma(z)}Y_{rs}))\\
    &= \frac{1}{2}\dtatzero e_{j}\cdot(z\exp(tY_{rs})I_{m,n}\bar{z}^t + I_{m+n})\cdot e_{\alpha}^t\\
    &= \frac{1}{2}\cdot e_{j}\cdot zY_{rs}I_{m,n}\bar{z}^t\cdot e_{\alpha}^t,
\end{align*}
\begin{align*}
    \Ad{\sigma(z)}(Y_{rs})^{2}(f_{j\alpha})(\Phi(z))&= \restr{\frac{d^2}{dt^2}}{t=0} f_{j\alpha}(\Phi(z)\cdot \exp(t\Ad{\sigma(z)}Y_{rs}))\\
    &= \frac{1}{2}\cdot e_{j}\cdot zY_{rs}^2I_{m,n}\bar{z}^t\cdot e_{\alpha}^t,
\end{align*}
and with similar computations,
\begin{align*}
    i\Ad{\sigma(z)}(X_{rs})(f_{j\alpha})(\Phi(z)) = \frac{i}{2}\cdot e_{j}\cdot zX_{rs}I_{m,n}\bar{z}^t\cdot e_{\alpha}^t
\end{align*}
and
\begin{align*}
   (i\Ad{\sigma(z)}(X_{rs}))^2(f_{j\alpha})(\Phi(z)) = -\frac{1}{2}\cdot e_{j}\cdot zX_{rs}^2I_{m,n}\bar{z}^t\cdot e_{\alpha}^t,
\end{align*}
for $z \in \U{m+n}$ and $1\leq r < s \leq m+n$.
From these relations we find
\begin{align}\label{eq: tau for c grass}
\begin{split}
    \tau_{N}(f_{j\alpha})(\Phi(z)) &= \sum_{r=1}^{m}\sum_{s= m+1}^{m+n} \tfrac{1}{2}\cdot (e_{j}\cdot zY_{rs}^2I_{m,n}\bar{z}^t\cdot e_{\alpha}^t - e_{j}\cdot zX_{rs}^2I_{m,n}\bar{z}^t\cdot e_{\alpha}^t)\\
    &= \tfrac{1}{2}\cdot e_{j}\cdot z\cdot\left(\sum_{r=1}^{m}\sum_{s= m+1}^{m+n}Y_{rs}^2-X_{rs}^2 \right)\cdot I_{m,n} \bar{z}^t\cdot e_{\alpha}^t.    
\end{split}
\end{align}
It is furthermore easy to show that $Y_{rs}^2 = -X_{rs}^2$ for $1 \leq r < s \leq m+n$. Applying Lemma \ref{square sum relation} then gives
\begin{align*}
   \sum_{r=1}^{m}\sum_{s= m+1}^{m+n}Y_{rs}^2-X_{rs}^2 &= 2 \sum_{r=1}^{m}\sum_{s= m+1}^{m+n}Y_{rs}^2\\
   &= \sum_{1\leq r< s \leq m+n} Y_{rs}^2 - \sum_{1\leq r < s \leq m} Y_{rs}^2 - \sum_{m+1\leq r < s \leq m+n} Y_{rs}^2\\
   &= -(m+n-1)\cdot\begin{pmatrix}
   I_{m} & 0\\
   0 & I_{n}
   \end{pmatrix} + (m-1)\cdot\begin{pmatrix}
   I_{m} & 0\\
   0 & 0
   \end{pmatrix}\\
   &\sed+ (n-1)\cdot \begin{pmatrix}
   0 & 0\\
   0 & I_{n}
   \end{pmatrix}\\
   &= -\tfrac{m+n}{2}\cdot I_{m+n} +\tfrac{m-n}{2}\cdot I_{m,n}.
\end{align*}
Inserting this into equation \ref{eq: tau for c grass} we obtain
\begin{align*}
    \tau_{N}(f_{j\alpha})(\Phi(z)) &= -\tfrac{m+n}{4}\cdot e_{j}\cdot z I_{m,n}\bar{z}^t\cdot e_{\alpha}^t +\tfrac{m-n}{4}\cdot\delta_{j\alpha}\\
    &= -\tfrac{m+n}{2}\cdot\tfrac{1}{2}\cdot e_j\cdot(zI_{m,n}\bar{z}^t + I_{m+n})\cdot e_{\alpha}^t +\tfrac{n}{2}\cdot \delta_{j\alpha}\\
    &= -\tfrac{m+n}{2}\cdot \psi_{j\alpha}(z) + \tfrac{n}{2}\cdot \delta_{j\alpha}, 
\end{align*}
from which the desired formula for $\tau(\psi_{j\alpha})$ is deduced via Theorem \ref{composition relations}. For the conformality operator we have
\begin{align}\label{eq: kappa for c grass}
    \begin{split}
        \kappa_{N}(f_{j\alpha},f_{k\beta})(\Phi(z)) &= \tfrac{1}{4}\bigg\{\sum_{r=1}^{m}\sum_{s=m+1}^{m+n} e_{j}\cdot zY_{rs}I_{m,n}\bar{z}^t \cdot e_{\alpha}^t\cdot e_{k} \cdot zY_{rs}I_{m,n}\bar{z}^t \cdot e_{\beta}^t \\
        &\sed - e_{j}\cdot zX_{rs}I_{m,n}\bar{z}^t\cdot e_{\alpha}^t \cdot e_{k}\cdot zX_{rs}I_{m,n}\bar{z}^t \cdot e_{\beta}^t\bigg\}\\
        &= \tfrac{1}{4}\bigg\{\sum_{r=1}^{m}\sum_{s=m+1}^{m+n} e_{j}\cdot zX_{rs}\bar{z}^t \cdot e_{\alpha}^t\cdot e_{k} \cdot zX_{rs}\bar{z}^t \cdot e_{\beta}^t \\
        &\sed - e_{j}\cdot zY_{rs}\bar{z}^t\cdot e_{\alpha}^t\cdot e_{k}\cdot zY_{rs}\bar{z}^t \cdot e_{\beta}^t\bigg\}.
    \end{split}
\end{align}
We can expand the first sum involving the $X_{rs}$ terms as
\begin{align}\label{eq: Xrs terms c grass}
    \begin{split}
        &\sum_{r=1}^{m}\sum_{s=m+1}^{m+n} e_{j}\cdot zX_{rs}\bar{z}^t\cdot e_{\alpha}^t \cdot e_{k}\cdot zX_{rs}\bar{z}^t\cdot e_{\beta}^t\\ &= \tfrac{1}{2}\sum_{r=1}^{m}\sum_{s=m+1}^{m+n}\big\{ e_{j}z e_{r}^t \cdot e_{s}\bar{z}^t e_{\alpha}^t\cdot e_k z e_{r}^t\cdot e_{s} \bar{z} e_{\beta}^t + e_{j}z e_{s}^t \cdot e_{r}\bar{z}^t e_{\alpha}^t\cdot e_k z e_{r}^t\cdot e_{s} \bar{z} e_{\beta}^t\\
        &\sed + e_{j}z e_{r}^t \cdot e_{s}\bar{z}^t e_{\alpha}^t\cdot e_k z e_{s}^t\cdot e_{r} \bar{z} e_{\beta}^t + e_{j}z e_{s}^t \cdot e_{r}\bar{z}^t e_{\alpha}^t\cdot e_k z e_{s}^t\cdot e_{r} \bar{z} e_{\beta}^t\big\}\\
        &= \tfrac{1}{2}\sum_{r=1}^{m}\sum_{s=m+1}^{m+n}\big\{ z_{jr}\bar{z}_{\alpha s}z_{kr}\bar{z}_{\beta s} + z_{js}\bar{z}_{\alpha r}z_{kr}\bar{z}_{\beta s} + z_{jr}\bar{z}_{\alpha s}z_{ks}\bar{z}_{\beta r} + z_{js}\bar{z}_{\alpha r}z_{ks}\bar{z}_{\beta r}\big\}.
    \end{split}
\end{align}
For the terms involving $Y_{rs}$ we similarly obtain
\begin{align}\label{eq: Yrs terms c grass}
    \begin{split}
        &\sum_{r=1}^{m}\sum_{s=m+1}^{m+n} e_{j}\cdot zY_{rs}\bar{z}^t \cdot e_{\alpha}^t\cdot e_{k}\cdot zY_{rs}\bar{z}^t\cdot  e_{\beta}^t\\
        &= \tfrac{1}{2}\sum_{r=1}^{m}\sum_{s=m+1}^{m+n}\big\{ z_{jr}\bar{z}_{\alpha s}z_{kr}\bar{z}_{\beta s} - z_{js}\bar{z}_{\alpha r}z_{kr}\bar{z}_{\beta s}- z_{jr}\bar{z}_{\alpha s}z_{ks}\bar{z}_{\beta r} + z_{js}\bar{z}_{\alpha r}z_{ks}\bar{z}_{\beta r}\big\}.
    \end{split}
\end{align}
Inserting equations \eqref{eq: Xrs terms c grass} and \eqref{eq: Yrs terms c grass} into equation \eqref{eq: kappa for c grass} we obtain
\begin{align*}
    &\kappa_{N}(f_{j\alpha},f_{k\beta})(\Phi(z))\\
    &= \tfrac{1}{4}\sum_{r=1}^{m}\sum_{s=m+1}^{m+n} z_{js}\bar{z}_{\alpha r}z_{kr}\bar{z}_{\beta s} + z_{jr}\bar{z}_{\alpha s}z_{ks}\bar{z}_{\beta r}\\
    &= \tfrac{1}{4}\left(\sum_{r=1}^{m}z_{kr}\bar{z}_{\alpha r}\right)\cdot \left(\sum_{s=m+1}^{m+n}z_{js}\bar{z}_{\beta z}\right) + \tfrac{1}{4}\left(\sum_{r=1}^{m}z_{jr}\bar{z}_{\beta r}\right)\cdot \left(\sum_{s=m+1}^{m+n}z_{ks}\bar{z}_{\alpha s}\right)\\
    &= \tfrac{1}{4}\{\psi_{k\alpha}\cdot(\delta_{j\beta}-\psi_{j\beta}) + \psi_{j\beta}\cdot(\delta_{k\alpha}-\psi_{k\alpha})\},
\end{align*}
where in the last line we have used the easily verified identity
\begin{align*}
    \psi_{j\alpha} = \sum_{r= 1}^{m} z_{jr}\bar{z}_{\alpha r} = \delta_{j\alpha} - \sum_{s= m+1}^{m+n} z_{js}\bar{z}_{\alpha s}.
\end{align*}
The desired formula for $\kappa(\psi_{j\alpha},\psi_{k\beta})$ then follow by Theorem \ref{composition relations}.
\end{proof}
The $\U{m}\times \U{n}$-invariant functions $\psi_{j\alpha}:\U{m+n} \to \cn$ induce functions $\hat{\psi}_{j\alpha}:M = \U{m+n}/\U{m}\times \U{n} \to \cn$. As a direct consequence of Lemma \ref{tau kappa c grass}, the induced functions then satisfy the following. 
\begin{theorem}[\cite{Gud-Gha 2}]\label{eigenfams comp grass}
Let $M = \U{m+n}/\U{m}\times \U{n}$. For a fixed natural number $1\leq \alpha\leq m+n$ the set
\begin{align*}
    \mathcal{E}_{\alpha} = \{ \hat{\psi}_{j\alpha}: M \to \cn \ | \ 1 \leq j \leq m+n, j \neq \alpha \} 
\end{align*}
is an eigenfamily on the complex Grassmannian $M$ such that the tension field $\tau$ and conformality operator $\kappa$ on $M$ satisfy
\begin{align*}
    \tau({\hat{\phi}}) = -2(m+n)\cdot \hat{\phi} \quad \textrm{and} \quad \kappa(\hat{\phi},\hat{\psi}) = -2\cdot \hat{\phi}\hat{\psi}
\end{align*}
for all $\hat{\phi},\hat{\psi} \in \mathcal{E}_{\alpha}$.
\end{theorem}

\begin{remark}
These eigenfamilies coincide with those found by S. Gudmundsson and E. Ghandour in the paper \cite{Gud-Gha 2}.
\end{remark}

\section[The Symmetric Space \textbf{SO}(m+n)/\textbf{SO}(m) x \textbf{SO}(n)]{The Symmetric Space $\SO{m+n}/\SO{m}\times \SO{n}$}
In this section we describe the real Grassmannians $M = \SO{m+n}/\SO{m}\times \SO{n}$ as symmetric spaces and give an explicit formula for the corresponding Cartan map. We then use the tools and techmiques from Chapters \ref{ch: harmonic morphisms} and \ref{ch: eigen} to construct complex-valued eigenfunctions on these spaces. These can the be utilised to give a varienty of complex-valued $p$-harmonic maps.

\begin{example}\label{SO(n+m)/(SO(n)*SO(m))}
Let $m,n$ be positive integers and set $$M = \SO{m+n}/\SO{m}\times \SO{n}.$$ Here the subgroup $\SO{m} \times \SO{n}$ consists of the block matrices
\begin{align*}
    \left\{ 
    \begin{pmatrix}
    x^{(11)} & 0\\
    0 & x^{(22)}
    \end{pmatrix} \in \SO{m+n} \ \middle| \  x^{(11)}\in \SO{m}, x^{(22)} \in \SO{n} 
    \right\}.
\end{align*}
As in the complex case, we set $\sigma(x) = I_{m,n}xI_{m,n}$ and by similar computations as in Example \ref{ex: SU(n+m)/S(U(n)*U(m))} we see that $\sigma$ is an involutive automorphism of $\SO{m+n}$ with fixed group $\SO{m}\times \SO{n}$. Hence $M$ is indeed a symmetric space. Since $\SO{m+n}$ is compact and semisimple, $M$ is of compact type. The Cartan embedding corresponding to $\sigma$ is given by
\begin{align*}
    \Phi: x \mapsto xI_{m,n}x^t I_{m,n}.
\end{align*}
These symmetric spaces are known as \emph{real oriented Grassmannians} and can be identified with the set of $m$-dimensional oriented subspaces of $\rn^{m+n}$. An important special case is the case with $n = 1$ in which case we get thath $M = \SO{m+1}/(\SO{m}\times \SO{1}$ is simply the $m$-dimensional sphere $\mathbf{S}^{m}$.
\end{example}
\begin{lemma}\label{basis SO(m+n)/SO(m)*SO(n)}
Let $\so{m+n} = \la{k}\oplus\la{p}$ be the Cartan decomposition corresponding to the symmetric triple $(\SO{m+n},\SO{m}\times \SO{n},\sigma)$ of Example \ref{SO(n+m)/(SO(n)*SO(m))}. Then an orthonormal basis for $\la{p}$ is given by
\begin{align*}
    \basis_{\la{p}} = \{ Y_{rs} \ | \ 1 \leq r \leq m, m+1 \leq s \leq m+n\} 
\end{align*}
\end{lemma}
\begin{proof}
The set $\basis_{\la{p}}$ is a strict subset of the orthonormal basis $\basis$ for $\so{m+n}$ given in Proposition \ref{basis so}. The complement $\basis \setminus \basis_{\la{p}}$ is clearly an orthonormal basis for the subalgebra $\la{k}$. The inner product on $\so{m+n}$ is a scalar multiple of the killing form. Hence, since the eigenspaces of $d\sigma$ are orthogonal with respect to the killing form, $\basis_{\la{p}}$ must be an orthonormal basis for $\la{p}$.
\end{proof}
We will now make use of Theorem \ref{composition relations} to construct eigenfamilies on the real Grassmannian $M = \SO{m+n}/\SO{m}\times \SO{n}$. In a similar fashion to the complex case, we define the maps $\psi_{j\alpha}: \SO{m+n} \to \rn$ by
\begin{align*}
    \psi_{j\alpha}: x \mapsto \tfrac{1}{2}\cdot(xI_{m,n}x^t + I_{m+n})_{j\alpha}.
\end{align*}
Using that $x \in \SO{m+n}$ one may with simple computations verify that
\begin{align*}
    \psi_{j\alpha}(x) = \sum_{r=1}^{m}x_{jr}x_{\alpha r}.
\end{align*}
Like in the complex case, we may also write $\psi_{j\alpha}$ as a composition $\psi_{j\alpha} = f_{j\alpha} \circ \Phi$ of the function
\begin{align*}
    f_{j\alpha}:x \mapsto \tfrac{1}{2}\cdot(x\cdot I_{m,n}+I_{m+n})_{j\alpha}
\end{align*}
with the Cartan map $\Phi$. Hence, $\psi_{j\alpha}$ is $\SO{m}\times \SO{n}$-invariant, and we can use Theorem \ref{composition relations} to prove the following Lemma.
\begin{lemma}[\cite{Gud-Gha 1}]\label{lem: tau kappa on real grass}
The tension field and conformality operator on the special orthogonal group $\SO{m+n}$ satisfy the relations
\begin{align*}
    \tau(\psi_{j\alpha}) = -(m+n)\cdot \psi_{j\alpha} - n\cdot\delta_{j\alpha}
\end{align*}
and
\begin{align*}
    \kappa(\psi_{j\alpha},\psi_{k\beta}) &= -(\psi_{j\beta}\psi_{k\alpha} +\psi_{jk}\psi_{\alpha\beta})\\
    &\sed + \tfrac{1}{2}\cdot(\delta_{\alpha\beta}\cdot\psi_{jk} +\delta_{k\alpha}\cdot\psi_{j\beta} + \delta_{j\beta}\cdot\psi_{k\alpha} +\delta_{jk}\cdot\psi_{\alpha\beta}).
\end{align*}
\end{lemma}
\begin{proof}
By Theorem \ref{composition relations} to compute $\tau(\psi_{j\alpha})(x)$ and $\kappa(\psi_{j\alpha},\psi_{k\beta})(x)$ we need only compute $\tau_{N}(f_{j\alpha})(\Phi(x))$ and $\kappa_{N}(f_{j\alpha},f_{k\beta})(\Phi(x))$, where $\tau_{N}$ and $\kappa_{N}$ are the tension field and conformality operators on $N = \Phi(\SO{m+n})$, respectively. By Lemma \ref{basis SO(m+n)/SO(m)*SO(n)} and Remark \ref{expression for basis}, a basis for $T_{\Phi(x)}N$ is given by
\begin{align*}
    \basis = \{\Ad{\sigma(x)}(Y_{rs})_{\Phi(x)} \ | \ 1\leq r \leq m, m+1 \leq s \leq m+n \}.
\end{align*}
We have, by similiar computations to those in the proof of Lemma \ref{tau kappa c grass}, 
\begin{align*}
    \Ad{\sigma(x)}(Y_{rs})(f_{j\alpha})(\Phi(x)) =\frac{1}{2}\cdot e_{j}\cdot xY_{rs}I_{m,n}x^t\cdot e_{\alpha}^t
\end{align*}
and
\begin{align*}
    \Ad{\sigma(x)}(Y_{rs})^2(f_{j\alpha})(\Phi(x)) = \frac{1}{2}\cdot e_{j}\cdot xY_{rs}^2I_{m,n}x^t\cdot e_{\alpha}^t.
\end{align*}
for $x \in \SO{m+n}$ and $1 \leq r < s \leq m+n$. From these relations we find
\begin{align}\label{eq: tau r grass}
\begin{split}
    \tau_{N}(f_{j\alpha})(\Phi(x)) &= \tfrac{1}{2}\cdot \sum_{r=1}^{m}\sum_{s= m+1}^{m+n} e_{j}\cdot xY_{rs}^2I_{m,n}x^t\cdot e_{\alpha}^t\\
    &= \tfrac{1}{2}\cdot e_{j}\cdot x \cdot \left(\sum_{r=1}^{m}\sum_{s= m+1}^{m+n}Y_{rs}^2\right)\cdot x^t\cdot e_{\alpha}^t.     
\end{split}
\end{align}
By similar arguments and computations as in the proof of Lemma \ref{tau kappa c grass} we find
\begin{align*}
    \sum_{r=1}^{m}\sum_{s= m+1}^{m+n}Y_{rs}^2 = -\tfrac{m+n}{4}\cdot I_{m+n} + \tfrac{m-n}{4}\cdot I_{m,n}.
\end{align*}
Inserting this into equation \ref{eq: tau r grass} we obtain
\begin{align*}
    \tau_{N}(f_{j\alpha})(\Phi(x)) &= -\tfrac{m+n}{8}\cdot e_{j}\cdot xI_{m,n}x^{t}\cdot e_{\alpha}^t + \tfrac{m-n}{8}\cdot \delta_{j\alpha}\\
    &= -\tfrac{m+n}{4}\cdot \tfrac{1}{2}\cdot e_{j}\cdot (xI_{m,n}x^{t} +I_{m+n})\cdot e_{\alpha}^t + \tfrac{n}{4}\cdot \delta_{j\alpha}\\
    &= -\tfrac{m+n}{4}\cdot \psi_{j\alpha} + \tfrac{n}{4}\cdot \delta_{j\alpha}
\end{align*}
from which the desired formula is deduced via Theorem \ref{composition relations}. For the conformality operator we have, again with similar computations to those in the proof of Lemma \ref{tau kappa c grass}
\begin{align*}
    &\kappa_{N}(f_{j\alpha},f_{k\beta})(\Phi(x))\\
    &= \tfrac{1}{4}\sum_{r=1}^{m}\sum_{s=m+1}^{m+n} e_{j}\cdot x Y_{rs}I_{m,n} x^t\cdot e_{\alpha}^t \cdot  e_{k}\cdot x Y_{rs}I_{m,n} x^t\cdot e_{\beta}^t\\
    &= \tfrac{1}{4}\sum_{r=1}^{m}\sum_{s=m+1}^{m+n} e_{j}\cdot x X_{rs} x^t\cdot e_{\alpha}^t\cdot  e_{k}\cdot x\cdot X_{rs}\cdot x^t\cdot e_{\beta}^t\\
    &= \tfrac{1}{8}\sum_{r=1}^{m}\sum_{s=m+1}^{m+n} \big\{ x_{jr}x_{\alpha s} x_{kr}x_{\beta s} + x_{js}x_{\alpha r} x_{kr}x_{\beta s}  + x_{jr}x_{\alpha s} x_{ks}x_{\beta r} + x_{js}x_{\alpha r} x_{ks}x_{\beta r}\\
    &= \tfrac{1}{8}\left(\sum_{r=1}^{m} x_{jr}x_{kr} \right)\cdot \left(\sum_{s=m+1}^{m+n} x_{\alpha s}x_{\beta s} \right) + \tfrac{1}{8}\left(\sum_{r=1}^{m} x_{kr}x_{\alpha r} \right)\cdot \left(\sum_{s=m+1}^{m+n} x_{j s}x_{\beta s} \right)\\
    &\qquad + \tfrac{1}{8}\left(\sum_{r=1}^{m} x_{jr}x_{\beta r} \right)\cdot \left(\sum_{s=m+1}^{m+n} x_{k s}x_{\alpha s} \right) + \tfrac{1}{8}\left(\sum_{r=1}^{m} x_{\alpha r}x_{\beta r} \right)\cdot \left(\sum_{s=m+1}^{m+n} x_{j s}x_{k s} \right)\\
    &= \tfrac{1}{8}\cdot \{\psi_{jk}\cdot(\delta_{\alpha\beta} - \psi_{\alpha\beta}) + \psi_{k\alpha}\cdot(\delta_{j\beta} - \psi_{j\beta})\\
    &\sed + \psi_{j\beta}\cdot(\delta_{k\alpha} - \psi_{k\alpha}) + \psi_{\alpha\beta}\cdot(\delta_{jk} - \psi_{jk}) \}.
\end{align*}
From this the desired formula for $\kappa(\psi_{j\alpha},\psi_{k\beta})$ can be deduced upon simplifying and applying Theorem \ref{composition relations}.
\end{proof}
\begin{theorem}[\cite{Gud-Gha 1}]\label{thm: eigenfunctions on real grass}
    Let $A$ be a complex symmetric $(m+n)\times (m+n)$-matrix of rank $1$ such that $A^2 =0$ and $\tr(A) = 0$. Then the $\SO{m} \times \SO{n}$-invariant function $\psi_A: \SO{m+n} \to \cn$, given by
    \begin{align*}
        \psi_A(x) = \sum_{j,\alpha =1}^{m+n} a_{j\alpha}\cdot\psi_{j\alpha}(x)
    \end{align*}
    induces an eigenfunction $\hat{\psi}_A:\SO{m+n}/ \SO{m} \times \SO{n} \to \cn$ on the real Grassmannian with
    \begin{align*}
        \tau(\hat{\psi}_A) = -(m+n)\cdot \hat{\psi}_A \quad \text{and} \quad \kappa(\hat{\psi}_A,\hat{\psi}_A) = -2 \cdot\hat{\psi}_A^2.
    \end{align*}
\end{theorem}
\begin{proof}
    It follows from Lemma \ref{lem: tau kappa on real grass} and the fact that $A$ is traceless that the tension field $\tau$ on $\SO{m+n}$ satisfies
    \begin{align*}
        \tau(\psi_A) = -(m+n)\cdot \psi_{A} + n\cdot \tr(A) = -(m+n)\cdot \psi_A.
    \end{align*}
    For the conformality operator $\kappa$ we have
    \begin{align*}
        &2\cdot\psi_A^2 + \kappa(\psi_A,\psi_A)\\
        &= 2\cdot\psi_A^2 + \sum_{j,\alpha,k,\beta = 1}^{m+n} \kappa(a_{j\alpha}\psi_{j\alpha},a_{k\beta}\psi_{k\beta})\\
        &= 2\cdot\psi_A^2 + \sum_{j,\alpha,k,\beta = 1}^{m+n}a_{j\alpha}a_{k\beta}\kappa(\psi_{j\alpha},\psi_{k\beta})\\
        &= 2\cdot\psi_A^2 - \sum_{j,\alpha, k, \beta = 1}^{m+n} a_{j\alpha}a_{k\beta}(\psi_{j\beta}\psi_{k\alpha} + \psi_{jk}\psi_{\alpha\beta})\\
        &\sed +\tfrac{1}{2}\cdot\sum_{j,\alpha,k,\beta=1}^{m+n}a_{j\alpha}a_{k\beta}(\delta_{\alpha\beta}\psi_{jk} + \delta_{k\alpha}\psi_{j\beta} +\delta_{j\beta}\psi_{k\alpha} + \delta_{jk}\psi_{\alpha\beta})\\
        &= \sum_{j,\alpha,k,\beta=1}^{m+n} (2a_{j\alpha}a_{k\beta} - a_{j\beta}a_{k\alpha} - a_{jk}a_{\alpha\beta})\cdot \psi_{j\alpha}\psi_{k\beta} + 2\sum_{j,\alpha,t=1}^{m+n} a_{jt}a_{\alpha t} \psi_{j\alpha}\\
        &= \sum_{j,\alpha,k,\beta=1}^{m+n}\left(\det 
        \begin{pmatrix}
            a_{j\alpha} & a_{j\beta}\\
            a_{k\alpha} & a_{k\beta}
        \end{pmatrix} + 
        \det 
        \begin{pmatrix}
            a_{j\alpha} & a_{jk}\\
            a_{\beta\alpha} & a_{\beta k} 
        \end{pmatrix}
        \right)\cdot \psi_{j\alpha}\psi_{k\beta}\\
        &\sed +2\sum_{j,\alpha = 1}^{m+n} \ip{a_j}{a_\alpha}\cdot\psi_{j\alpha}\\
        &= 0,
    \end{align*}
    as desired. The last line follows from the fact that $A$ is symmetric of rank $1$ and $A^2 = 0$.
\end{proof}
In the following we construct matrices satisfying the conditions of Theorem \ref{thm: eigenfunctions on real grass}, and hence get a multi-dimensional family of eigenfunctions on the real Grassmannian $\SO{m+n}/\SO{m}\times \SO{n}$.
\begin{example}[\cite{Gud-Gha 1}]\label{ex: explicit eigenfuncts on real grass}
Let $p = (p_1,\dots, p_{m+n}) \in \cn^{m+n}$ be a non-zero isotropic vector i.e.
    \begin{align*}
        \ip{p}{p} = p_1^2 + \cdots + p_{m+n}^2 =0.
    \end{align*}
Then the complex $(m+n)\times (m+n)$ 
symmetric matrix $A = p^t \cdot p$ satisfies $A^2 = 0$, $\tr(A) = 0$ and rank $A =1$. Hence by Theorem \ref{thm: eigenfunctions on real grass}, the $\SO{m}\times \SO{n}$-invariant function 
\begin{align*}
    \psi_{p} = \sum_{j,\alpha = 1}p_{j}p_{\alpha} \cdot \psi_{j\alpha}
\end{align*}
induces an eigenfunction $\hat{\psi}_p$ on $\SO{m+n}/\SO{m}\times \SO{n}$, satisfying
    \begin{align*}
        \tau(\hat{\psi}_p) = -(m+n)\cdot \hat{\psi}_p \quad \text{and} \quad \kappa(\hat{\psi}_p,\hat{\psi}_p) = -2 \cdot\hat{\psi}_p^2.
    \end{align*}
This provides a complex $(m+n-1)$-dimensional family of eigenfunctions on the real Grassmannian $\SO{m+n}/\SO{m}\times \SO{n}$.
\end{example}

\begin{remark}
    The eigenfunctions found in this section coincides with those found by Gudmundsson and Ghandour in \cite{Gud-Gha 1}.
\end{remark}

\section[The Symmetric Space \textbf{Sp}(m+n)/\textbf{Sp}(m) x \textbf{Sp}(n)]{The Symmetric Space $\Sp{m+n}/\Sp{m}\times\Sp{n}$}
In this section we describe the quaternionic Grassmannians $\Sp{m+n}/\Sp{m}\times\Sp{n}$ as symmetric spaces and give an explicit formula for the Cartan map. We then use the techniques developed in Chapters \ref{ch: harmonic morphisms} and \ref{ch: eigen} to construct a multitude of eigenfunctions and eigenfamilies on these spaces. In this case we find new eigenfamilies which significantly extend the eigenfamilies already found by Gudmundsson and Ghandour in \cite{Gud-Gha 2}.
\begin{example}\label{Quat-Grass}
Let $m,n$ be positive integers. We then set $M = \Sp{m+n}/\Sp{m}\times\Sp{n}$. 
In this case it is easier to first think of the elements of $\Sp{k}$ as quaternionic matrices. The subgroup $\Sp{m}\times \Sp{n}$ then sits inside $\Sp{m+n}$ as the group of quaternionic block matrices
\begin{align*}
    \Sp{n}\times\Sp{m} \cong 
    \left\{ 
    \begin{pmatrix}
    q^{(11)} & 0\\
    0 & q^{(22)}
    \end{pmatrix}
    \in \Sp{m+n} \
    \middle | \
    q^{(11)} \in \Sp{m}, \ q^{(22)} \in \Sp{n}
    \right\}.
\end{align*}
As in Examples \ref{ex: SU(n+m)/S(U(n)*U(m))} and \ref{SO(n+m)/(SO(n)*SO(m))}, we set $\sigma(q) = I_{m,n}qI_{m,n}$ and by similar arguments and computations as in those cases we see that $\sigma$ is an involutive automorphism of $\Sp{m+n}$ with fixed subgroup $\Sp{m}\times \Sp{n}$. Hence $M = \Sp{m+n}/\Sp{m}\times \Sp{n}$ is a symmetric space. Since $\Sp{m+n}$ is compact and semisimple, $M$ is of compact type. Note also that $I_{m,n} \in \Sp{m+n}$.

As our aim is to find complex-valued eigenfamilies, it will be more useful to describe $M$ and its Cartan embedding in terms of the standard complex representation, as found in Proposition \ref{Sp(n), GL(H) Lie groups},  of $\Sp{m+n}$. For the remainder of this section we therefore shall think of $q \in \Sp{m+n}$ as a $2(m+n)\times 2(m+n)$ complex matrix. In terms of this representation the automorphism $\sigma$ is described by $\sigma(q) = \tilde{I}_{m,n}q\tilde{I}_{m,n}$, where $\tilde{I}_{m,n}$ is the 
$$((m+n) + (m+n))\times ((m+n) + (m+n))$$ block matrix
\begin{align*}
    \tilde{I}_{m,n} = \begin{pmatrix}
    I_{m,n} & 0\\
    0 & I_{m,n}
    \end{pmatrix}
\end{align*}
The corresponding Cartan embedding is given by
\begin{align*}
    \Phi: q \mapsto q\tilde{I}_{m,n}\bar{q}^t\tilde{I}_{m,n}.
\end{align*}
\end{example}
In a similar vein to the complex and real Grassmannians, we define the maps $\psi_{j\alpha}: \Sp{m+n} \to \cn$ by
\begin{align*}
    \psi_{j\alpha}: q \mapsto \tfrac{1}{2}\cdot(q\tilde{I}_{m,n}\bar{q}^t + I_{2(m+n)})_{j\alpha},
\end{align*}
for $1 \leq j,\alpha \leq 2(m+n)$. It is easy to verify that $\psi_{j\alpha}$ then satisfies
\begin{align}\label{eq: psi as poly}
    \psi_{j\alpha} = \sum_{r=1}^{m} q_{jr}\bar{q}_{\alpha r} + \sum_{r= m+n+1}^{2m + n} q_{jr}\bar{q}_{\alpha r}.
\end{align}
We furthermore see that we may write $\psi_{j\alpha}$ as a composition $\psi_{j\alpha} = \eta_{j\alpha} \circ \Phi$, where $\eta_{j\alpha}: \Sp{m+n} \to \cn$ is given by
\begin{align*}
    \eta_{j\alpha}: q \mapsto \tfrac{1}{2}\cdot(q\cdot \tilde{I}_{m,n} + I_{2(m+n)})_{j\alpha}
\end{align*}
and $\Phi: \Sp{m+n} \to \Sp{m+n}$ is the Cartan map. Hence $\psi_{j\alpha}$ is $\Sp{m}\times \Sp{n}$-invariant.

We could now employ Theorem \ref{composition relations} to calculate the tension field and conformality operator acting on the functions $\psi_{j\alpha}$. However, in practice the computations are simpler if we instead use equation \eqref{eq: psi as poly} and previous results concering the actions of these operators on the matrix elements $q_{j\alpha}$. Recall that for the standard complex representation we have
\begin{align*}
    q = \begin{pmatrix}
    z & w\\
    -\bar{w} & \bar{z}
    \end{pmatrix}.
\end{align*}
We introduce the functions $f_{j\alpha},g_{j\alpha}: \Sp{m+n} \to \cn$ with
\begin{align*}
    f_{j\alpha}(q) = \sum_{r=1}^{m}( z_{jr}\bar{z}_{\alpha r} + w_{jr}\bar{w}_{\alpha r}),
\end{align*}
and
\begin{align*}
    g_{j\alpha}(q) = \sum_{r=1}^{m} (w_{jr}z_{\alpha r} - z_{jr}w_{\alpha r}).
\end{align*}
This means that
\begin{align}\label{eq: psi comp}
    \begin{split}
    \psi_{j\alpha } &= f_{j\alpha },\\
    \psi_{j,\alpha +m+n} &= g_{j\alpha },\\
    \psi_{j+m+n,\alpha } &= -\bar{g}_{j\alpha },\\
    \psi_{j+m+n,\alpha +m+n} &= \bar{f}_{j\alpha },  
    \end{split}
\end{align}
for $1 \leq j,\alpha, \leq m+n$. We will now determine the action of the tension field and the conformality operator on these functions. For this we will make use of the following result, which is Lemma 9.1 of \cite{Gud-Gha 2}.
\begin{lemma}[\cite{Gud-Gha 2},\cite{Gud-Sak}]\label{tau kappa quat grass}
Let $z_{j\alpha},w_{j\alpha}: \Sp{n+m} \to \cn$ be the matrix elements of the standard complex representation of the quaternionic unitary group $\Sp{n+m}$. Then the tension field $\tau$ and conformality operator $\kappa$ satisfy the following relations
\begin{align*}
    \tau(z_{j\alpha}) = -\tfrac{2(n+m)+1}{2}\cdot z_{j\alpha}, &\quad \tau(w_{j\alpha}) = -\tfrac{2(n+m)+1}{2}\cdot w_{j\alpha},\\
    \kappa(z_{j\alpha},z_{k \beta}) = -\tfrac{1}{2}\cdot z_{j\beta}z_{k\alpha}, &\quad \kappa(w_{j\alpha},w_{k\beta}) = -\tfrac{1}{2}\cdot w_{j\beta}w_{k\alpha},\\
    \kappa(z_{j\alpha},w_{k\beta}) &= -\tfrac{1}{2}\cdot z_{k\alpha}w_{j\beta},\\
    \tau(\bar{z}_{j\alpha}) = -\tfrac{2(n+m)+1}{2}\cdot \bar{z}_{j\alpha}, &\quad \tau(\bar{w}_{j\alpha}) = -\tfrac{2(n+m)+1}{2}\cdot \bar{w}_{j\alpha},\\
    \kappa(\bar{z}_{j\alpha},\bar{z}_{k\beta}) = -\tfrac{1}{2}\cdot \bar{z}_{j \beta}\bar{z}_{k\alpha}, &\quad \kappa(\bar{w}_{j\alpha},\bar{w}_{k\beta}) = -\tfrac{1}{2}\cdot \bar{w}_{j\beta}\bar{w}_{k\alpha},\\
    \kappa(\bar{z}_{j\alpha},\bar{w}_{k\beta}) &= -\tfrac{1}{2}\cdot \bar{z}_{k\alpha}\bar{w}_{j\beta},
\end{align*}
\begin{equation*}
    \kappa(z_{j\alpha},\bar{z}_{k\beta}) = \tfrac{1}{2}\cdot(w_{j\beta}\bar{w}_{k\alpha} + \delta_{jk}\delta_{\alpha\beta}), \quad \kappa(z_{j\alpha},\bar{w}_{k\beta}) = -\tfrac{1}{2}\cdot z_{j\beta}\bar{w}_{k\alpha},
\end{equation*}
\begin{equation*}
    \kappa(w_{j\alpha},\bar{z}_{k\beta}) = -\tfrac{1}{2}\cdot w_{j\beta}\bar{z}_{k\alpha}, \quad \kappa(w_{j\alpha},\bar{w}_{k\beta}) = \tfrac{1}{2}\cdot(z_{j\beta}\bar{z}_{k\alpha} + \delta_{jk}\delta_{\alpha\beta}).
\end{equation*}
\end{lemma}

In order to calculate the tension field and conformality operator acting on the matrix elements $\psi_{j\alpha}$ it is, by equation \eqref{eq: psi comp}, sufficient to calculate the tension field and conformality operator on the functions $f_{j\alpha}$, $g_{j\alpha}$ and their complex conjugates. For this we have the following result.
\begin{lemma}\label{quat grass u f kappa tau}
The the tension field $\tau$ and the conformality operator $\kappa$ on the quaternion unitary group $\Sp{m+n}$ satisfy the following relations
$$
    \tau(f_{j\alpha}) = -2(m+n)\cdot f_{j\alpha} + 2n\cdot \delta_{j\alpha}, \quad \tau(g_{j\alpha}) = -2(m+n)\cdot g_{j\alpha},
$$
$$
    \tau(\bar{f}_{j\alpha}) = -2(m+n)\cdot \bar{f}_{j\alpha} + 2n\cdot \delta_{j\alpha}, \quad \tau(\bar{g}_{j\alpha}) = -2(m+n)\cdot \bar{g}_{j\alpha},
$$
$$
    \kappa(f_{j\alpha},f_{k\beta}) = -f_{j\beta}f_{k\alpha} + g_{kj}\bar{g}_{\alpha \beta} + \tfrac{1}{2}\cdot(f_{k\alpha }\delta_{j\beta} + f_{j\beta}\delta_{\alpha k}),
$$
$$
    \kappa(\bar{f}_{j\alpha },\bar{f}_{k\beta}) = -\bar{f}_{j\beta}\bar{f}_{k\alpha } + \bar{g}_{kj}g_{\alpha \beta} + \tfrac{1}{2}\cdot(\bar{f}_{k\alpha }\delta_{j\beta} + \bar{f}_{j\beta}\delta_{\alpha k}),
$$
$$
    \kappa(g_{j\alpha },g_{k\beta}) = -g_{\beta j}g_{\alpha k} + g_{\alpha \beta}g_{kj},
$$
$$
    \kappa(\bar{g}_{j\alpha },\bar{g}_{k\beta}) = -\bar{g}_{\beta j}\bar{g}_{\alpha k} + \bar{g}_{\alpha \beta}\bar{g}_{kj},
$$
$$
    \kappa(f_{j\alpha },\bar{f}_{k\beta}) = -f_{jk}f_{\beta \alpha } + g_{\beta j}\bar{g}_{\alpha k} + \tfrac{1}{2}\cdot(f_{\beta \alpha }\delta_{jk} + f_{jk}\delta_{\alpha \beta}),
$$
$$
    \kappa(g_{j\alpha },\bar{g}_{k\beta}) = f_{j\beta}f_{\alpha k} - f_{jk}f_{\alpha \beta} + \tfrac{1}{2}\cdot (f_{jk}\delta_{\alpha \beta} + f_{j\beta}\delta_{jk} - f_{j\beta}\delta_{\alpha k} - f_{\alpha k}\delta_{j\beta}),
$$
$$
 \kappa(g_{j\alpha },f_{k\beta}) = g_{kj}f_{\alpha \beta} + g_{\alpha k}f_{j\beta} + g_{jk}\delta_{\alpha \beta} + g_{k\alpha }\delta_{j\beta},
$$
$$
 \kappa(g_{j\alpha },\bar{f}_{k\beta}) = -g_{j\beta}\bar{f}_{k\alpha } + g_{\alpha \beta}\bar{f}_{kj} + g_{j\beta}\delta_{\alpha k} + g_{\beta \alpha }\delta_{jk},
$$
for each $1\leq j,\alpha,k,\beta \leq m+n$.
\end{lemma}
\begin{proof}
For the tension field of $f_{j\alpha}$ we have, using Lemma \ref{tau kappa quat grass},
\begin{align*}
    \tau(f_{j\alpha}) &= \sum_{r=1}^{m} \tau(z_{jr}\bar{z}_{\alpha r}) + \tau(w_{jr}\bar{w}_{\alpha r})\\
    &= \sum_{r=1}^{m}\big\{ \tau(z_{jr})\bar{z}_{\alpha r} + 2\kappa(z_{jr},\bar{z}_{\alpha r}) + z_{jr}\tau(\bar{z}_{\alpha r})\\
    &\sed+ \tau(w_{jr})\bar{w}_{\alpha r} + 2\kappa(w_{jr},\bar{w}_{\alpha r}) + w_{jr}\tau(\bar{w}_{\alpha r})\big\}\\
    &= \sum_{r=1}^{m}\big\{ -\tfrac{2(m+n)+1}{2}\cdot z_{jr}\bar{z}_{\alpha r} + w_{jr}\bar{w}_{\alpha r} + \delta_{j\alpha}\delta_{rr} -\tfrac{2(m+n)+1}{2}\cdot z_{jr}\bar{z}_{\alpha r}\\
    &\sed -\tfrac{2(m+n)+1}{2}\cdot w_{jr}\bar{w}_{\alpha r} + z_{jr}\bar{z}_{\alpha r} + \delta_{j\alpha}\delta_{rr} -\tfrac{2(m+n)+1}{2}\cdot w_{jr}\bar{w}_{\alpha r}\big\}\\
    &= -2(m+n)\cdot \sum_{r=1}^{m}\big\{ z_{jr}\bar{z}_{\alpha r} + w_{jr}\bar{w}_{\alpha r}\big\} + 2n\cdot \delta_{j\alpha}\\
    &= -2(m+n)\cdot f_{j\alpha} + 2n\cdot \delta_{j\alpha}.
\end{align*}
For $g_{j\alpha}$ we get
\begin{align*}
    \tau(g_{j\alpha}) &= \sum_{r=1}^{m}\big\{ \tau(w_{jr})z_{\alpha r} + 2\kappa(w_{jr},z_{\alpha r}) + w_{jr}\tau(z_{\alpha r})\\
    &\sed -\tau(z_{jr})w_{\alpha r} - 2\kappa(z_{jr},w_{\alpha r}) - z_{ik}\tau(w_{\alpha r})\big\}\\
    &= \sum_{r=1}^{m}\big\{ -\tfrac{2(m+n)+1}{2}\cdot w_{jr}z_{\alpha r} - w_{\alpha r}z_{jr} - \tfrac{2(m+n)+1}{2}\cdot w_{jr}z_{\alpha r}\\
    &\sed +\tfrac{2(m+n)+1}{2}\cdot z_{jr}w_{\alpha r} + z_{\alpha r}w_{jr} + \tfrac{2(m+n)+1}{2}\cdot z_{jr}w_{\alpha r}\big\}\\
    &= -2(m+n)\cdot \sum_{r=1}^{m}\big\{ w_{jr}z_{\alpha r} - z_{jr}w_{\alpha r}\big\}\\
    &= -2(m+n)\cdot g_{j\alpha}.
\end{align*}
We can then use the complex linearity of $\tau$ to derive the formulae for $\tau(\bar{f}_{j\alpha})$ and $\tau(\bar{g}_{j\alpha})$. To derive the formulae involving the conformality operator we only need to calculate $\kappa(f_{j\alpha},f_{k\beta})$, $\kappa(g_{j\alpha},g_{k\beta})$, $\kappa(g_{j\alpha},\bar{g}_{k\beta})$ and $\kappa(g_{j\alpha},\bar{f}_{k\beta})$ as we can then use the complex bilinearity of $\kappa$ together with the relation $\bar{f}_{j\alpha} = f_{\alpha j}$ for the remaining cases. For $\kappa(f_{j\alpha},f_{k\beta})$ we have
\begin{align*}
    &\kappa(f_{j\alpha},f_{k\beta})\\ 
    &= \kappa\left(\sum_{r=1}^{m} z_{jr}\bar{z}_{\alpha r}+w_{jr}\bar{w}_{\alpha r},\sum_{s=1}^{m} z_{ks}\bar{z}_{\beta s} + w_{ks}\bar{w}_{\beta s}\right)\\
    &= \sum_{r,s = 1}^{m} \big\{ \kappa(z_{jr}\bar{z}_{\alpha r},z_{ks}\bar{z}_{\beta s}) + \kappa(w_{jr}\bar{w}_{\alpha r},z_{ks}\bar{z}_{\beta s})\\
    &\sed + \kappa(z_{jr}\bar{z}_{\alpha r},w_{ks}\bar{w}_{\beta s}) + \kappa(w_{jr}\bar{w}_{\alpha r},w_{ks}\bar{w}_{\beta s})\big\}\\
    &= \sum_{r,s= 1}^{m} \big\{ z_{jr}z_{ks}\kappa(\bar{z}_{\alpha r},\bar{z}_{\beta s}) + \bar{z}_{\alpha r}z_{ks}\kappa(z_{jr},\bar{z}_{\beta s}) + z_{jr}\bar{z}_{\beta s}\kappa(\bar{z}_{\alpha r},z_{ks})\\
    &\sed + \bar{z}_{\alpha r}\bar{z}_{\beta s}\kappa(z_{jr},z_{ks}) + w_{jr}z_{ks}\kappa(\bar{w}_{\alpha r},\bar{z}_{\beta s}) + \bar{w}_{\alpha r}z_{ks}\kappa(w_{jr},\bar{z}_{\beta s})\\
    &\sed + w_{jr}\bar{z}_{\beta s}\kappa(\bar{w}_{\alpha r},z_{ks}) + \bar{w}_{\alpha r}\bar{z}_{\beta s}\kappa(w_{jr},z_{ks}) + z_{jr}w_{ks}\kappa(\bar{z}_{\alpha r},\bar{w}_{\beta s})\\
    &\sed + \bar{z}_{\alpha r}w_{ks}\kappa(z_{jr},\bar{z}_{\beta s})
    + z_{jr}\bar{w}_{\beta s}\kappa(\bar{z}_{\alpha r},w_{ks}) + \bar{z}_{\alpha r}\bar{w}_{\beta s}\kappa(z_{jr},w_{ks})\\ &\sed + w_{jr}w_{ks}\kappa(\bar{w}_{\alpha r},\bar{w}_{\beta s})
    + \bar{w}_{\alpha r}w_{ks}\kappa(w_{jr},\bar{w}_{\beta s})+ w_{jr}\bar{w}_{\beta s}\kappa(\bar{w}_{\alpha r},w_{ks})\\
    &\sed+ \bar{w}_{\alpha r}\bar{w}_{\beta s}\kappa(w_{jr},w_{ks})\big\}\\
    &= \tfrac{1}{2}\cdot \sum_{r,s = 1}^{m} \big\{ -z_{jr}z_{ks}\bar{z}_\alpha s\bar{z}_{\beta r} + \bar{z}_{\alpha r}z_{ks}w_{js}\bar{w}_{\beta r} + \bar{z}_{\alpha r}z_{ks}\delta_{j\beta}\delta_{rs} + z_{jr}\bar{z}_{\beta s}\bar{w}_{\alpha s}w_{kr}\\
    &\sed + z_{jr}\bar{z}_{\beta s}\delta_{\alpha k}\delta_{rs} - \bar{z}_{\alpha r}\bar{z}_{\beta s}z_{js}z_{kr} - w_{jr}z_{ks}\bar{w}_{\beta r}\bar{z}_{\alpha s} - \bar{w}_{\alpha r}z_{ks}w_{js}\bar{z}_{\beta r}\\
    &\sed - w_{jr}\bar{z}_{\beta s}\bar{w}_{\alpha s}z_{kr} - \bar{w}_{\alpha r}\bar{z}_{\beta s}w_{kr}z_{js} - z_{jr}w_{ks}\bar{z}_{\beta r}\bar{w}_{\alpha s} - \bar{z}_{\alpha r}w_{ks}z_{js}\bar{w}_{\beta r}\\
    &\sed - z_{jr}\bar{w}_{\beta s}\bar{z}_{\alpha s}w_{kr} -\bar{z}_{\alpha r}\bar{w}_{\beta s}z_{kr}w_{js} - w_{jr}w_{ks}\bar{w}_{\alpha s}\bar{w}_{\beta r} + \bar{w}_{\alpha r}w_{ks}z_{js}\bar{z}_{\beta r}\\
    &\sed + \bar{w}_{\alpha r}w_{ks}\delta_{j\beta}\delta_{rs} + w_{jr}\bar{w}_{\beta s}\bar{z}_{\alpha s}z_{kr} + w_{jr}\bar{w}_{\beta s}\delta_{\alpha k}\delta_{rs} - \bar{w}_{\alpha r}\bar{w}_{\beta s}w_{js}w_{kr}\big\}\\
    &= \sum_{r,s=1}^{m}\big\{ -z_{jr}\bar{z}_{\beta r}z_{ks}\bar{z}_{\alpha s} + z_{ks}w_{js}\bar{z}_{\alpha r}\bar{w}_{\beta r} +z_{jr}w_{kr}\bar{z}_{\beta s}\bar{w}_{\alpha s} - z_{ks}\bar{z}_{\alpha s}w_{jr}\bar{w}_{\beta r}\\
    &\sed -z_{ks}w_{js}\bar{z}_{\beta r}\bar{w}_{\alpha r} - z_{js}\bar{z}_{\beta s}w_{kr}\bar{w}_{\alpha r} - z_{js}w_{ks}\bar{z}_{\alpha r}\bar{w}_{\beta r} - w_{jr}\bar{w}_{\beta r}w_{ks}\bar{w}_{\alpha s}\\
    &\sed +\tfrac{1}{2}\cdot(\bar{z}_{\alpha r}z_{ks}\delta_{j\beta}\delta_{rs} + z_{jr}\bar{z}_{\beta s}\delta_{\alpha k}\delta_{rs} + \bar{w}_{\alpha r}w_{ks}\delta_{j\beta}\delta_{rs} + w_{jr}\bar{w}_{\beta s}\delta_{\alpha k}\delta_{rs}) \big\}\\
    &= \sum_{r,s = 1}^{m}\big\{(w_{kr}z_{jr}-z_{jr}w_{kr})(\bar{w}_{\alpha s}\bar{z}_{\beta s}-\bar{z}_{\beta s}\bar{w}_{\alpha s})-(z_{jr}\bar{z}_{\beta r} + w_{jr}\bar{w}_{\beta r})(z_{ks}\bar{z}_{\alpha s} + w_{ks}\bar{w}_{\alpha s})\big\}\\
    &\sed + \tfrac{1}{2}\cdot \sum_{r=1}^{m}\big\{(z_{kr}\bar{z}_{\alpha r}+w_{kr}\bar{w}_{\alpha r})\cdot \delta_{j\beta} +(z_{jr}\bar{z}_{\beta r} + w_{jr}\bar{w}_{\beta r})\cdot \delta_{\alpha k}\big\}\\
    &= g_{kj}\bar{g}_{\alpha \beta} - f_{j\beta}f_{k\alpha} + \tfrac{1}{2}\cdot (f_{k\alpha}\delta_{j\beta}+f_{j\beta}\delta_{\alpha k})
\end{align*}
as desired. For $\kappa(g_{j\alpha},g_{k\beta})$ we get
\begin{align*}
    \kappa(g_{j\alpha},g_{k\beta} ) &= \kappa\left(\sum_{r=1}^{m}w_{jr}z_{\alpha r}-z_{jr}w_{\alpha r},\sum_{s=1}^{m}w_{ks}z_{\beta  s}-z_{ks}w_{\beta  s}\right)\\
    &= \sum_{r,s = 1}^{m} \big\{ z_{jr}z_{ks}\kappa(w_{\alpha r},w_{\beta  s}) + z_{jr}w_{\beta  s}\kappa(w_{\alpha r},z_{ks}) + w_{\alpha r}z_{ks}\kappa(z_{jr},w_{\beta  s})\\
    &\sed + w_{\alpha r}w_{\beta  s}\kappa(z_{jr},z_{ks}) - z_{\alpha r}z_{ks}\kappa(w_{jr},w_{\beta  s}) - z_{\alpha r}w_{\beta  s}\kappa(w_{jr},z_{ks})\\
    &\sed -w_{jr}z_{ks}\kappa(z_{\alpha r},w_{\beta  s}) - w_{jr}w_{\beta  s}\kappa(z_{\alpha r},z_{ks}) - z_{jr}z_{\beta  s}\kappa(w_{\alpha r},w_{ks})\\
    &\sed - z_{jr}w_{ks}\kappa(w_{\alpha r},z_{\beta  s}) - w_{\alpha r}z_{\beta  s}\kappa(z_{jr},w_{ks}) - w_{\alpha r}w_{ks}\kappa(z_{jr},z_{\beta  s})\\
    &\sed + z_{\alpha r}z_{\beta  s}\kappa(w_{jr},w_{ks}) + z_{\alpha r}w_{ks}\kappa(w_{jr},z_{\beta  s}) + w_{jr}z_{\beta  s}\kappa(z_{\alpha r},w_{ks})\\
    &\sed + w_{jr}w_{ks}\kappa(z_{\alpha r},z_{\beta  s})\big\}.
\end{align*}
Applying Lemma \ref{tau kappa quat grass} and combining terms then yields
\begin{align*}
    &\kappa(g_{j\alpha},g_{k\beta})\\
    &= -\sum_{r,s = 1}^{m} \big\{w_{\beta r}z_{jr}w_{\alpha s}z_{ks} + w_{kr}z_{jr}w_{\beta s}z_{\alpha s} + w_{js}z_{ks}w_{\alpha r}z_{\beta r} + w_{kr}z_{\alpha r}w_{js}z_{\beta s}\\
    &\sed - z_{\alpha r}w_{\beta r}z_{ks}w_{js}- z_{\alpha r}w_{kr}z_{js}w_{\beta s} -  z_{ks}w_{\alpha s}z_{\beta r}w_{jr} -  z_{jr}w_{kr}z_{\beta s}w_{\alpha s}\big\}\\
    &= - \sum_{r,s = 1}^{m} \big\{(w_{\beta r}z_{jr}-z_{jr}w_{\beta r})(w_{\alpha s}z_{ks}-z_{\alpha s}w_{ks})\\
    &\sed- (w_{\alpha r}z_{\beta r}-z_{\alpha r}w_{\beta r})(w_{ks}z_{js}-z_{ks}w_{js}) \big\}\\
    &= -g_{\beta j}g_{\alpha k}+g_{\alpha \beta}g_{kj}.
\end{align*}

The calculations for $\kappa(g_{j\alpha},\bar{g}_{k\beta})$ and $\kappa(g_{j\alpha},\bar{f}_{k\beta})$ are similar and will not be presented here. This then concludes the proof.
\end{proof}
As a consequence of Lemma \ref{quat grass u f kappa tau} we then have the following relations for the functions $\psi_{j \alpha}$.
\begin{proposition}\label{quat grass psi kappa tau}
Let $1 \leq j,k,\alpha  \leq 2(m+n)$ with $j,k \neq \alpha$ . Then the tension field $\tau$ and the conformality operator $\kappa$ on the quaternionic unitary group $\Sp{m+n}$ satisfy
\begin{align*}
    \tau(\psi_{j\alpha}) = -2(m+n)\cdot \psi_{j\alpha} \quad \text{and} \quad \kappa(\psi_{j\alpha},\psi_{k\alpha}) = -\psi_{j\alpha}\psi_{k\alpha}.
\end{align*}
\end{proposition}
\begin{proof}
It follows directly from Lemma \ref{quat grass u f kappa tau} and Equation \eqref{eq: psi comp} that
\begin{align*}
    \tau(\psi_{j\alpha}) = -2(m+n)\cdot \psi_{j\alpha}
\end{align*}
whenever $j \neq \alpha$. For the conformality operator we divide this into cases depending on the values of the indices. For $j,k,\alpha \leq m+n$ we have
\begin{align*}
    \kappa(\psi_{j\alpha},\psi_{k\alpha}) &= \kappa(f_{j\alpha},f_{k\alpha})\\
    &= -f_{j\alpha}f_{k\alpha} + g_{kj}\bar{g}_{\alpha \alpha} + \tfrac{1}{2}\cdot (f_{k\alpha}\delta_{j\alpha} + f_{j\alpha}\delta_{\alpha k})\\
    &= -f_{j\alpha}f_{k\alpha}\\
    &= -\psi_{j\alpha}\psi_{k\alpha},
\end{align*}
since $g_{\alpha\alpha} = 0$ and $j,k \neq \alpha$. As $\bar{\psi}_{j+m+n,\alpha+m+n} = \psi_{j\alpha}$ for $j,k,\alpha \leq m+n$ this also covers the case with $j,k,\alpha \geq m+n+1$.
In the case with $j,k \leq m+n$ and $\alpha \geq m+n+1$, setting $\tilde{\alpha} = \alpha-(m+n)$ we get
\begin{align*}
    \kappa(\psi_{j\alpha},\psi_{k\alpha}) &= \kappa(g_{j\tilde{\alpha}},g_{k\tilde{\alpha}})\\ 
    &= -g_{\tilde{\alpha} j}\cdot g_{\tilde{\alpha} k} + g_{\tilde{\alpha}\tilde{\alpha}}\cdot g_{kj}\\
    &= -g_{j\tilde{\alpha}}\cdot g_{k\tilde{\alpha}}\\
    &= -\psi_{j\alpha}\psi_{k\alpha},
\end{align*}
since $g_{j\alpha} = - g_{\alpha j}$. As $\bar{\psi}_{j+(m+n),\alpha} = \psi_{j,\alpha +m+n}$ for $j,k \leq m+n$ and $\alpha \geq m+n+1$ this also covers the case with $j,k \geq m+n+1$ and $\alpha \leq m+n$. For $j\leq m+n < k$ and $\alpha \geq m+n+1$, setting $\tilde{\jmath} = j - (m+n)$ and $\tilde{\alpha} = \alpha-(m+n)$ we obtain
\begin{align*}
    \kappa(\psi_{j\alpha},\psi_{k\alpha}) &= \kappa(g_{j\tilde{\alpha}},\bar{f}_{\tilde{\jmath}\tilde{\alpha}})\\
    &= -g_{\tilde{\jmath}\tilde{\alpha}}\bar{f}_{k\tilde{\alpha}} + g_{\tilde{\alpha}\tilde{\alpha}}\bar{f}_{k\tilde{\jmath}} + g_{\tilde{\jmath}\tilde{\alpha}}\delta_{\tilde{\alpha}k} + g_{\tilde{\alpha}\tilde{\alpha}}\delta_{\tilde{\jmath}k}\\
    &= -g_{\tilde{\jmath}\tilde{\alpha}}\bar{f}_{k\tilde{\alpha}}\\
    &= -\psi_{j\alpha}\psi_{k\alpha}.
\end{align*}
Conjugating on both sides then gives the desired relation also in the final case of $j\leq m+n < k$ and $\alpha \leq m+n$. This concludes the proof
\end{proof}
The $\Sp{m}\times \Sp{n}$-invariant functions $\psi_{j\alpha}: \Sp{m+n} \to \cn$ induce functions $\hat{\psi}_{j\alpha}: \Sp{m+n}/\Sp{m}\times \Sp{n} \to \cn$. The eigenfunctions obtained by taking $1 \leq j,\alpha \leq m+n$ were found by Gudmundsson and Ghandour already in \cite{Gud-Gha 2}. The remaining cases, however, are entirely new as of this thesis. Hence, in the following, we find new eigenfamilies on the quaternionic Grassmannians in addition to greatly expanding those already discovered in \cite{Gud-Gha 2}.
\begin{theorem}\label{eigenfams on quat grass}
For a fixed natural number $1 \leq \alpha \leq 2(m+n)$ the set
\begin{align*}
    \mathcal{E}_{\alpha} = \{ \hat{\psi}_{j\alpha}:\Sp{m+n}/\Sp{n}\times \Sp{m} \to \cn \ | \ 1 \leq j \leq 2(m+n),\ j \neq \alpha\}
\end{align*}
is an eigenfamily on the quaternionic Grassmannian $\Sp{m+n}/\Sp{m}\times \Sp{n}$ such that the tension field $\tau$ and conformality operator $\kappa$ on $\Sp{m+n}/\Sp{m}\times \Sp{n}$ satsify
\begin{align*}
    \tau(\hat{\phi}) = -2(m+n)\cdot \hat{\phi} \quad \text{and} \quad \kappa(\hat{\phi},\hat{\psi}) = -\hat{\phi}\cdot \hat{\psi},
\end{align*}
for all $\hat{\phi},\hat{\psi} \in \mathcal{E}_{\alpha}$.
\end{theorem}

\section[The Symmetric Space \textbf{SU}(n)/\textbf{SO}(n)]{The Symmetric Space $\SU{n}/\SO{n}$}
In this section we describe the family of symmetric spaces $M = \SU{n}/\SO{n}$ and  derive explicit formulae for the corresponding Cartan maps. We then use the techniques from Chapters \ref{ch: harmonic morphisms} and \ref{ch: eigen} to find complex-valued eigenfunctions, which may then be used to construct a variety of $p$-harmonic functions on these spaces.
\begin{example}\label{SU/SO}
We now turn to the case $M = {\SU{n}}/{\SO{n}}$. It is clear that for $z \in \SU{n}$, $z$ belongs to $\SO{n}$ if and only if $z = \Bar{z}$. This indicates that a good choice for an involutive automorphism might be 
\begin{align*}
    \sigma: \SU{n} \to \SU{n}, \ z \mapsto \bar{z}.
\end{align*}
Indeed for $z$, $w \in \SU{n}$ we have $$\sigma(zw) = \overline{zw} = \Bar{z}\Bar{w} = \sigma(z)\sigma(w) ,\ \sigma(e) = e \ \ \text{and} \ \ \sigma^2(z) = z,$$ so that $\sigma$ is an involutive automorphism of $\SU{n}$ fixing $\SO{n}$. Thus, $M = {\SU{n}}/{\SO{n}}$ is indeed a symmetric space. 
The corresponding Cartan embedding is 
\begin{align*}
    \Phi: z \mapsto z\overline{\bar{z}}^t = zz^t.
\end{align*}
The image of $\Phi$ is 
\begin{align*}
    \text{Im}(\Phi) = \{zz^t \ | \ z \in \SU{n}\} \subset (\SU{n} \cap \sym{\cn^{n\times n}}).
\end{align*}
Since $\SU{n}$ is semisimple and compact, $M$ is of compact type.
\end{example}
\begin{lemma}\label{basis SU/SO}
Let $\su{m} = \la{k} \oplus \la{p}$ be the Cartan decomposition corresponding to the symmetric triple $(\SU{n},\SO{n},\sigma)$ of Example \ref{SU/SO}. Then the set 
\begin{align*}
   \mathcal{C} = \{Y_{rs} \ | \ 1\leq r < s \leq n \} \cup \{\tfrac{i}{\sqrt{n}}\cdot I_{n}\}
\end{align*}
is an orthonormal basis for the orthogonal complement to $\la{p}$ in $\u{n}$.
\end{lemma}
\begin{proof}
The set
\begin{align*}
   \basis_{\la{k}} = \{Y_{rs} \ | \ 1\leq r < s \leq n \}
\end{align*}
is clearly an orthonormal basis for the Lie subalgebra $\la{k}$ which is the orthogonal complement of $\la{p}$ in $\su{n}$. By Remark \ref{basis for su} the unit vector $\tfrac{i}{\sqrt{n}}\cdot I_{n}$ in turn generates the orthogonal complement of $\su{n}$ in $\u{n}$, and so we are done.
\end{proof}
The remainder of this section will be devoted to obtaining eigenfunctions on the symmetric space $M = \SU{n}/\SO{n}$. We define $\phi_{j\alpha}: \SU{n} \to \cn$ by
\begin{align*}
    \phi_{j\alpha}: z \mapsto (\Phi(z))_{j\alpha} = e_{j}\cdot zz^t \cdot e_{\alpha}^t
\end{align*}
for $1 \leq j, \alpha \leq n$. The functions $\phi_{j\alpha}$ are clearly $\SO{n}$-invariant and we have $\phi_{j\alpha} = z_{j\alpha}\circ \Phi$. Hence, we may use Theorem \ref{composition relations} to deduce the action of the tension field and the conformality operator of $\SU{n}$ on the functions $\phi_{j\alpha}$.
\begin{lemma}[\cite{Gud-Sif-Sob}]\label{tau kappa SU/SO}
The tension field and the conformality operator on the special unitary group $\SU{n}$ satisfy the relations
\begin{align*}
    \tau(\phi_{j\alpha}) = \tfrac{-2(n^2 +n -2)}{n}\cdot \phi_{j\alpha}
\end{align*}
and
\begin{align*}
    \kappa(\phi_{j\alpha},\phi_{k\beta}) = -2\cdot \phi_{jk}\phi_{\alpha \beta} - 2\cdot\phi_{j\beta}\phi_{k\alpha} + \tfrac{4}{n}\cdot \phi_{j\alpha}\phi_{k\beta}.
\end{align*}
\end{lemma}
\begin{proof}
By Theorem \ref{composition relations} it is sufficient to compute 

$$\tau_{N}(z_{j\alpha})(\Phi(z)) \quad \mathrm{and}\quad \kappa_{N}(z_{j\alpha},z_{k\beta})(\Phi(z))$$
for $1 \leq j,k,\alpha,\beta \leq n$. Here $\tau_{N}$ and $\kappa_{N}$ are the tension field and conformality operator on $N = \Phi(\SU{n})$, respectively. By Remark \ref{basis cartan embedding}, if $\basis_{\la{p}}$ is an orthonormal basis for $\la{p}$ we have
\begin{align*}
    \tau_{N}(x_{j\alpha})(\Phi(z)) = \sum_{X\in \basis_{\la{p}}} e_{j}\cdot zX^2z^t\cdot e_{\alpha}^t
\end{align*}
and
\begin{align*}
    \kappa_N(x_{j\alpha},x_{k\beta})(\Phi(z)) = \sum_{X\in \basis_{\la{p}}} e_j\cdot zXz^t\cdot e_{\alpha}^t\cdot e_{k}\cdot zXz^t\cdot e_{\beta}^t.
\end{align*}
By combining Lemma \ref{basis SU/SO} with Proposition \ref{basis u}, and using the identities in Proposition \ref{square sum relation}, we get for $\tau_{N}$ 
\begin{align*}
        \tau_{N}(x_{j\alpha})(\Phi(z)) &= \sum_{X\in \basis_{\la{p}}} e_{j}\cdot zX^2z^t\cdot e_{\alpha}^t\\
        &= -\sum_{r < s}\big\{ e_{j}\cdot zX_{rs}^2z^t\cdot e_{\alpha}^t\big\} - \sum_{t=1}^n \big\{ e_{j}\cdot zD_{t}^2z^t\cdot e_{\alpha}^t\big\} + \tfrac{1}{n} e_{j}zI_n^2z^te_{\alpha}^t\\
        &= (-\tfrac{n-1}{2}-1+\tfrac{1}{n})\cdot\phi_{j\alpha}\\
        &= -\tfrac{n^2 + n - 2}{2n}\cdot\phi_{j\alpha}.
\end{align*}
For the conformality operator $\kappa_{N}$ we have
\begin{align}\label{eq: kappa SU/SO}
    \begin{split}
        &\kappa_N(x_{j\alpha},x_{k\beta})(\Phi(z))\\
        &= \sum_{X\in \basis_{\la{p}}} e_j\cdot zXz^t\cdot e_{\alpha}^t\cdot e_{k}\cdot zXz^t\cdot e_{\beta}^t\\
        &= -\sum_{r<s} e_j\cdot zX_{rs}z^t\cdot e_{\alpha}^t\cdot e_{k}\cdot zX_{rs}z^t\cdot e_{\beta}^t - \sum_{t=1}^{n} e_j\cdot zD_tz^t\cdot e_{\alpha}^t\cdot e_{k}\cdot zD_tz^t\cdot e_{\beta}^t\\
        &\sed + \tfrac{1}{n}\cdot e_j\cdot zI_nz^t\cdot e_{\alpha}^t\cdot e_{k}\cdot zI_nz^t\cdot e_{\beta}^t.    
    \end{split}
\end{align}
For the first sum we obtain
\begin{align*}
    &-\sum_{r<s} e_j\cdot zX_{rs}z^t\cdot e_{\alpha}^t\cdot e_{k}\cdot zX_{rs}z^t\cdot e_{\beta}^t\\
    &= -\tfrac{1}{2}\cdot\sum_{r<s} z_{jr}z_{\alpha s}z_{kr}z_{\beta s} + z_{jr}z_{\alpha s}z_{ks}z_{\beta r} + z_{js}z_{\alpha r}z_{kr}z_{\beta s} + z_{js}z_{\alpha r}z_{ks}z_{\beta r}\\
    &= -\tfrac{1}{2}\cdot\sum_{\mathclap{\substack{r,s = 1\\ r\neq s}}}^n z_{jr}z_{kr}z_{\alpha s}z_{\beta s} + z_{jr}z_{\beta r}z_{ks}z_{\alpha s},
\end{align*}
and for the second we get
\begin{align*}
    - \sum_{t=1}^{n} e_j\cdot zD_tz^t\cdot e_{\alpha}^t\cdot e_{k}\cdot zD_tz^t\cdot e_{\beta}^t = -\sum_{t = 1}^{n}z_{jt}z_{\alpha t}z_{kt}z_{\beta t}. 
\end{align*}
Inserting these into equation \eqref{eq: kappa SU/SO} and simplifying gives
\begin{align*}
    \kappa_N(x_{j\alpha},x_{k\beta})(\Phi(z)) &= -\tfrac{1}{2}\cdot\left\{\sum_{r,s = 1}^n z_{jr}z_{kr}z_{\alpha s}z_{\beta s} + z_{jr}z_{\beta r}z_{kr}z_{\alpha r}\right\} + \tfrac{1}{n}\cdot \phi_{j\alpha}\phi_{k\beta}\\
    &= -\tfrac{1}{2}\cdot \phi_{jk}\phi_{\alpha\beta} -\tfrac{1}{2}\cdot \phi_{j\beta}\phi_{k\alpha} + \tfrac{1}{n}\phi_{j\alpha}\phi_{k\beta}.
\end{align*}
Applying Theorem \ref{composition relations} to the expressions for $\tau_{N}$ and $\kappa_{N}$ then yields the desired formulae for $\tau$ and $\kappa$.
\end{proof}
\begin{theorem}[\cite{Gud-Sif-Sob}]\label{eigenfunctions SU/SO}
Let $A$ be a symmetric matrix given by $A = a^t a$ for some non-zero vector $a \in \cn^{n}$. Define $\phi: \SU{n} \to \cn$ by
\begin{align*}
    \phi(z) = \tr(A \cdot \Phi(z)) = \sum_{j,\alpha = 1}^{n}a_{j}a_{\alpha}\phi_{j\alpha}.
\end{align*}
Then $\phi$ is an $\SO{n}$-invariant eigenfunction on $\SU{n}$ satisfying
\begin{align*}
    \tau{\phi} = -\tfrac{2(n^2 +n -2)}{n}\cdot \phi \quad \textrm{and}\quad \kappa(\phi,\phi) = -\tfrac{4(n-1)}{n}\cdot \phi^2.
\end{align*}
\end{theorem}
\begin{proof}
The expression for $\tau(\phi)$ follows directly from Lemma \ref{tau kappa SU/SO} by the linearity of $\tau$. For $\kappa$ we have by bilinearity
\begin{align*}
    \kappa(\phi,\phi) &= \sum_{j,\alpha,k,\beta} a_{j}a_{\alpha}a_{k}a_{\beta} \cdot \kappa(\phi_{j\alpha},\phi_{k\beta})\\
    &= \sum_{j,\alpha,k,\beta}a_{j}a_{\alpha}a_{k}a_{\beta}\cdot (-2\cdot \phi_{jk}\phi_{\alpha \beta} - 2\cdot\phi_{j\beta}\phi_{k\alpha} + \tfrac{4}{n}\cdot \phi_{j\alpha}\phi_{k\beta})\\
    &= -2\cdot \left(\sum_{j,k =1}^n a_{j}a_{k}\phi_{jk}\right)\left(\sum_{\alpha,\beta = 1}^n a_{\alpha}a_{\beta} \phi_{\alpha\beta}\right)\\
    &\sed -2\cdot \left(\sum_{j,\beta =1}^n a_{j}a_{\beta}\phi_{j\beta}\right)\left(\sum_{k,\alpha = 1}^n a_{k}a_{\alpha} \phi_{k\alpha}\right)\\
    &\sed + \tfrac{1}{n}\cdot \left(\sum_{j,\alpha =1}^n a_{j}a_{\alpha}\phi_{j\alpha}\right)\left(\sum_{k,\beta = 1}^n a_{k}a_{\beta} \phi_{k\beta}\right)\\
    &= -\tfrac{4(n-1)}{n}\cdot \phi^2
\end{align*}
as desired.
\end{proof}
\begin{remark}
The $\SO{n}$-invariant eigenfunction $\phi$ induces an eigenfunction $\hat{\phi}: \SU{n}/\SO{n} \to \cn$ with the same eigenvalues. These eigenfunctions coincide with those found in the paper \cite{Gud-Sif-Sob} by S. Gudmundsson, A. Siffert and M. Sobak.
\end{remark}
\section[The Symmetric Space \textbf{SO}(2n)/\textbf{U}(n)]{The Symmetric Space $\SO{2n}/\U{n}$}
In this section we describe the family of symmetric spaces $M = \SO{2n}/\U{n}$ and  derive explicit formulae for the corresponding Cartan maps. We then use the techniques from Chapters \ref{ch: harmonic morphisms} and \ref{ch: eigen} to find complex-valued eigenfunctions, which may then be used to construct a variety of $p$-harmonic functions on these spaces.
\begin{example}\label{SO(2n)/U(n)}
Here we will discuss the symmetric space $$M = \SO{2n}/\U{n}.$$ To this end we must first describe how $\U{n}$ is embedded into $\SO{2n}$. We may decompose an element $z$ of $\U{n}$ into real and imaginary components as
\begin{align*}
    z = x + iy
\end{align*}
where $x$ and $y$ are real. We then define a map $\varphi: \U{n} \to \GL{2n,\rn}$ by
\begin{align*}
    \varphi: x+iy \mapsto \begin{pmatrix}
    x & y\\
    -y & x
    \end{pmatrix}. 
\end{align*}
The computation
\begin{align*}
    \varphi((x+iy)(a+ib)) &= \varphi(xa - yb + i(xb + ya))\\
    &=
    \begin{pmatrix}
        xa - yb & xb + ya\\
        -(xb + yb) & xa - yb
    \end{pmatrix}\\
    &= \begin{pmatrix}
        x & y\\
        -y & x
    \end{pmatrix}
    \begin{pmatrix}
        a & b\\
        -b & a
    \end{pmatrix}\\
    &= \varphi(x+iy)\varphi(a+ib),
\end{align*}
shows $\varphi$ to be an group homomorphism. Since $\varphi$ is also smooth it is a Lie group homomorphism. As $\varphi(x+iy) = \varphi(a+ib)$ if and only if $x = a$ and $y = b$, $\varphi$ is also injective. Since $\varphi(\bar{z}^t) = \varphi(z)^t$ we must have $\text{Im}(\varphi) \subset \O{2n}$. Using elementary row and column operations on the block matrices we get
\begin{align*}
    \det{\begin{pmatrix}
        x & y\\
        -y & x
    \end{pmatrix}} &= 
    \det{\begin{pmatrix}
        x - iy & y+ix\\
        -y & x
    \end{pmatrix}}\\
    &= \det{\begin{pmatrix}
        x - iy & 0\\
        -y & x + iy
    \end{pmatrix}}\\
    &=
    \det{\begin{pmatrix}
        x - iy & 0\\
        \frac{1}{2}ix - \frac{1}{2}y & x + iy
    \end{pmatrix}}\\
    &=
    \det{\begin{pmatrix}
        x - iy & 0\\
        0 & x + iy
    \end{pmatrix}}\\
    &=
    \det{\left[\begin{pmatrix}
        x - iy & 0\\
        0 & I_n
    \end{pmatrix}
    \begin{pmatrix}
        I_n & 0 \\
        0 & x + iy
    \end{pmatrix}\right]}\\
    &= \det{(x+iy)}\det{(x-iy)}\\
    &= |\det(x+iy)|^2.    
\end{align*}
Hence, the image of $\varphi$ is contained in $\SO{2n}$. We will henceforth identify this image with $\U{n}$. If we let $K$ be the set of all $2n\times 2n$ block matrices in $\SO{2n}$ of the form
\begin{align*}
    k = \begin{pmatrix}
        x & y\\
        -y & x
    \end{pmatrix}
\end{align*}
then

\begin{align*}
    \psi: \begin{pmatrix}
        x & y\\
        -y & x
    \end{pmatrix} \mapsto x + iy
\end{align*}
clearly maps $K$ into $\U{n}$ and is the inverse map of $\varphi$. Hence we may identify $\U{n}$ with $K$. Let
\begin{align*}
    x = \begin{pmatrix}
    x^{(11)} & x^{(12)}\\
    x^{(21)} & x^{(22)}
    \end{pmatrix} \in \SO{2n},
\end{align*}
with $x^{(11)}$, $x^{(12)}$, $x^{(21)}$ and $x^{(22)}$ real $n\times n$ matrices, and set

\begin{align*}
    J = \begin{pmatrix}
    0 & I_n\\
    -I_n & 0
    \end{pmatrix}.
\end{align*}
Then $JJ^t = e$ and $\det(J) = 1$ so that $y \in \SO{2n}$,  and
\begin{align*}
    JxJ^t &= 
    \begin{pmatrix}
    0 & I_n\\
    -I_n & 0
    \end{pmatrix}
    \begin{pmatrix}
    x^{(11)} & x^{(12)}\\
    x^{(21)} & x^{(22)}
    \end{pmatrix}
\begin{pmatrix}
    0 & -I_n\\
    I_n & 0
    \end{pmatrix}\\
    &= 
    \begin{pmatrix}
    x^{(22)} & -x^{(21)}\\
    -x^{(12)} & x^{(11)}
    \end{pmatrix}.
\end{align*}
Hence, the conjugation by $J$ is an automorphism of $\SO{2n}$ keeping $\U{n}$ fixed. It is also clear from the computation that this is an involution and so we set $\sigma(x) = JxJ^t$. Thus, $M = \SO{2n}/\U{n}$ is a symmetric space. Since $\SO{2n}$ is compact and semisimple, $M$ is of compact type. The corresponding Cartan embedding is given by
\begin{align*}
    \Phi: x \mapsto xJx^t J^t = -xJx^tJ.
\end{align*}
\end{example}
\begin{lemma}\label{basis SO/U}
Let $\so{2n} = \la{k}\oplus \la{p}$ be the Cartan decomposition corresponding to the symmetric triple $(\SO{2n},\U{n},\sigma)$ of Example \ref{SO(2n)/U(n)}. Then an orthonormal basis for $\la{p}$ is given by
\begin{align*}
    \basis_{\la{p}} = \left\{ \tfrac{1}{\sqrt{2}}\begin{pmatrix}
    Y_{rs}& 0\\
    0 & -Y_{rs}
    \end{pmatrix}, \tfrac{1}{\sqrt{2}}\begin{pmatrix}
    0 & Y_{rs}\\
    Y_{rs} & 0
    \end{pmatrix} \ \middle| \ 1 \leq r<s \leq n \right\}.
\end{align*}
\end{lemma}
\begin{proof}
It is easily verified that the subalgebra $\la{k} \cong \u{n}$ is the subspace
\begin{align*}
    \la{k} = \left\{ \begin{pmatrix}
    X & Y\\
    -Y & X
    \end{pmatrix} \ \middle | \ X + X^t = Y - Y^t = 0\right\}.
\end{align*}
Hence the set
\begin{align*}
    \basis_{\la{k}} = \left\{ \tfrac{1}{\sqrt{2}}\begin{pmatrix}
    Y_{rs}& 0\\
    0 & Y_{rs}
    \end{pmatrix}, \tfrac{1}{\sqrt{2}}\begin{pmatrix}
    0 & X_{rs}\\
    -X_{rs} & 0
    \end{pmatrix} \ \middle| \ 1 \leq r<s \leq n\right\}
\end{align*}
is an orthonormal basis for $\la{k}$. Furthermore $\basis = \basis_{\la{p}} \cup \basis_{\la{p}}$ gives an orthonormal basis for $\so{2n}$. Thus, since $\la{k}$ are orthogonal $\la{p}$ with respect to the standard bi-invariant metric on $\so{2n}$ we are done.
\end{proof}
The remainder of this section will be devoted to obtaining eigenfunctions on the symmetric space $M = \SO{2n}/\U{n}$. We define $\psi_{j\alpha}: \SO{2n} \to \cn$ by
\begin{align*}
    \psi_{j\alpha}: x \mapsto (xJx^t)_{j\alpha}
\end{align*}
for $1 \leq j, \alpha \leq 2n$. One then easily verifies that $\psi_{j\alpha}$ satisfies
\begin{align*}
    \psi_{j\alpha}(x) = \sum_{r=1}^{n}x_{jr}x_{\alpha,r+n} - x_{j,r+n}x_{\alpha r}.
\end{align*}
We also see that we may write $\psi_{j\alpha}$ as a composition $\psi_{j\alpha} = f_{j\alpha} \circ \Phi$ of the function 
\begin{align*}
    f_{j\alpha}:x \mapsto (xJ)_{j\alpha}
\end{align*}
with the Cartan map $\Phi$. This means that the functions $\psi_{j\alpha}$ are $\U{n}$-invariant and we can use Theorem \ref{composition relations} to prove the following result.
\begin{lemma}[\cite{Gud-Sif-Sob}]\label{tau kappa on SO/U}
The tension field and the conformality operator on the special orthogonal group $\SO{2n}$ satisfy the relations
\begin{align*}
    \tau(\psi_{j\alpha}) = -2(n-1)\cdot \psi_{j\alpha}
\end{align*}
and
\begin{align*}
    \kappa(\psi_{j\alpha},\psi_{k\beta}) = -(\psi_{j\beta}\psi_{k\alpha} + \psi_{jk}\psi_{\alpha\beta}) -(\delta_{j\beta}\delta_{k\alpha} - \delta_{jk}\delta_{\alpha\beta})
\end{align*}
\end{lemma}
\begin{proof}
By Theorem \ref{composition relations} it is sufficient to compute $$\tau_{N}(f_{j\alpha})(\Phi(x)) \quad \mathrm{and} \quad \kappa_{N}(f_{j\alpha},f_{k\beta})(\Phi(x))$$
for $1 \leq j,k,\alpha,\beta \leq 2n$. Here $\tau_{N}$ and $\kappa_{N}$ are the tension field and conformality operator on $N = \Phi(\SO{2n})$, respectively. By Remark \ref{basis cartan embedding} and using the basis for $\la{p}$ from Lemma \ref{basis SO/U} we have 
\begin{align*}
    \tau_N(f_{j\alpha})(\Phi(x)) &=\tfrac{1}{2}\cdot\sum_{r < s}^{n}\bigg\{e_{j}\cdot x
\begin{psmallmatrix}
Y_{rs} & 0\\
0 & -Y_{rs}
\end{psmallmatrix}^2 Jx^t\cdot e_{\alpha}^t
 + e_{j}\cdot x
\begin{psmallmatrix}
0 & Y_{rs}\\
Y_{rs} & 0
\end{psmallmatrix}^2 Jx^t\cdot  e_{\alpha}^t\bigg\}\\
&= \sum_{r < s}^{n}e_{j}\cdot x\begin{psmallmatrix}
Y_{rs}^2 & 0\\
0 & Y_{rs}^2
\end{psmallmatrix}Jx^t\cdot e_{\alpha}^t\\
&= -\tfrac{n-1}{2}\cdot \psi_{j\alpha}.
\end{align*}
For $\kappa_{N}$ we get
\begin{align}\label{kappa SO/U}
\begin{split}
&\kappa_{N}(f_{j\alpha},f_{k\beta})(\Phi(x))\\
&= \tfrac{1}{2}\cdot\sum_{r<s}^{n} e_{j}\cdot x
\begin{psmallmatrix}
Y_{rs} & 0\\
0 & -Y_{rs}
\end{psmallmatrix} Jx^t\cdot  e_{\alpha}^t\cdot e_{k}\cdot x
\begin{psmallmatrix}
Y_{rs} & 0\\
0 & -Y_{rs}
\end{psmallmatrix} Jx^t\cdot e_{\beta}^t\\
&\sed +\tfrac{1}{2}\cdot \sum_{r<s}^{n} e_{j}\cdot x
\begin{psmallmatrix}
0 & Y_{rs}\\
Y_{rs} & 0
\end{psmallmatrix} Jx^t \cdot e_{\alpha}^t\cdot  e_{k}\cdot x
\begin{psmallmatrix}
0 & Y_{rs}\\
Y_{rs} & 0
\end{psmallmatrix} Jx^t\cdot  e_{\beta}^t\\
&= \tfrac{1}{2}\cdot\sum_{r<s}^{n} e_{j}\cdot x
\begin{psmallmatrix}
0 & Y_{rs}\\
Y_{rs} & 0 
\end{psmallmatrix} x^t\cdot e_{\alpha}^t\cdot e_{k}\cdot x
\begin{psmallmatrix}
0 & Y_{rs}\\
Y_{rs} & 0 
\end{psmallmatrix} x^t\cdot e_{\beta}^t\\
&\sed + \tfrac{1}{2}\cdot\sum_{r<s}^{n} e_{j}\cdot x
\begin{psmallmatrix}
-Y_{rs} & 0\\
0 & Y_{rs}
\end{psmallmatrix} x^t\cdot e_{\alpha}^t\cdot e_{k}\cdot x
\begin{psmallmatrix}
-Y_{rs} & 0\\
0 & Y_{rs}
\end{psmallmatrix} x^t\cdot e_{\beta}^t.    
\end{split}
\end{align}
Deonting by $\tilde{r}$ and $\tilde{s}$ the indicies $r+n$ and $s+n$, respectively, the first sum simplifies to
\begin{align*}
&\tfrac{1}{2}\cdot\sum_{r<s}^{n} e_{j}\cdot x
\begin{psmallmatrix}
0 & Y_{rs}\\
Y_{rs} & 0 
\end{psmallmatrix} x^t\cdot e_{\alpha}^t\cdot e_{k}\cdot x
\begin{psmallmatrix}
0 & Y_{rs}\\
Y_{rs} & 0 
\end{psmallmatrix} x^t\cdot e_{\beta}^t\\
&=\tfrac{1}{4}\cdot \sum_{r\neq s}^{n} x_{jr}x_{kr}x_{\alpha\tilde{s}}x_{\beta\tilde{s}} - x_{jr}x_{\beta \tilde{r}}x_{ks}x_{\alpha\tilde{s}} + x_{jr}x_{k\tilde{r}}x_{\beta s}x_{\alpha\tilde{s}} - x_{jr}x_{\beta r}x_{k\tilde{s}}x_{\alpha\tilde{s}}\\
 &\sed +x_{j\tilde{r}}x_{kr}x_{\alpha s}x_{\beta\tilde{s}} - x_{j\tilde{r}}x_{\beta\tilde{r}}x_{ks}x_{\alpha s} + x_{j\tilde{r}}x_{k\tilde{r}}x_{\alpha s}x_{\beta s} - x_{j\tilde{r}}x_{\beta r}x_{k\tilde{s}}x_{\alpha s}.
\end{align*}
For the second sum we get
\begin{align*}
&\tfrac{1}{2}\cdot\sum_{r<s}^{n} e_{j}\cdot x
\begin{psmallmatrix}
-Y_{rs} & 0\\
0 & Y_{rs}
\end{psmallmatrix} x^t\cdot e_{\alpha}^t\cdot e_{k}\cdot x
\begin{psmallmatrix}
-Y_{rs} & 0\\
0 & Y_{rs}
\end{psmallmatrix} x^t\cdot e_{\beta}^t\\
&=\tfrac{1}{4}\cdot \sum_{r\neq s} x_{jr}x_{kr}x_{\alpha s}x_{\beta s} - x_{jr}x_{\beta r}x_{ks}x_{\alpha s}- x_{jr}x_{k\tilde{r}}x_{\alpha s}x_{\beta\tilde{s}} + x_{jr}x_{\beta\tilde{r}}x_{k\tilde{s}}x_{\alpha s}\\
&\sed -x_{j\tilde{r}}x_{kr}x_{\alpha\tilde{s}}x_{\beta s} + x_{j\tilde{r}}x_{\beta r}x_{ks}x_{\alpha\tilde{s}} + x_{j\tilde{r}}x_{k\tilde{r}}x_{\alpha\tilde{s}}x_{\beta\tilde{s}} - x_{j\tilde{r}}x_{\beta\tilde{r}}x_{k\tilde{s}}x_{\alpha\tilde{s}}.
\end{align*}
Inserting these into the expression \eqref{tau kappa on SO/U} for $\kappa_N$ we obtain, upon factoring,
\begin{align*}
    \kappa_{N}(f_{j\alpha},f_{k\beta})(\Phi(x)) &= \tfrac{1}{4}\cdot \sum_{r\neq s} (x_{jr}x_{kr}+x_{j\tilde{r}}x_{k\tilde{r}})(x_{\alpha s}x_{\beta s}+x_{\alpha\tilde{s}}x_{\beta\tilde{s}})\\
    &\sed - (x_{jr}x_{\beta r}+ x_{j\tilde{r}}x_{\beta\tilde{r}})(x_{ks}x_{\alpha s}+x_{k\tilde{s}}x_{\alpha\tilde{s}})\\
    &\sed + (x_{j\tilde{r}}x_{kr} - x_{jr}x_{k\tilde{r}})(x_{\alpha s}x_{\beta\tilde{s}} -x_{\alpha\tilde{s}}x_{\beta s})\\
    &\sed + (x_{j\tilde{r}}x_{\beta r} - x_{jr}x_{\beta\tilde{r}})(x_{ks}x_{\alpha\tilde{s}}-x_{k\tilde{s}}x_{\alpha s}).
\end{align*}
This expression can be simplified further, using the identity $xx^t = I_{2n}$, to yield
\begin{align*}
    \kappa_{N}(f_{j\alpha},f_{k\beta})(\Phi(x)) =-\tfrac{1}{4}(\psi_{j\beta}\psi_{k\alpha} + \psi_{jk}\psi_{\alpha\beta} - \delta_{j\beta}\delta_{k\alpha} + \delta_{jk}\delta_{\alpha\beta}).
\end{align*}
The desired formulae for $\tau$ and $\kappa$ are then easily deduced from Theorem \ref{composition relations}.
\end{proof}

\begin{theorem}[\cite{Gud-Sif-Sob}]\label{eigenfunctions on SO/U}
Let $V \subset \cn^{2n}$ be an isotropic subspace, i.e. $\ip{Z}{W} = 0$ for all $Z,W\in V$, where $\ip{\cdot}{\cdot}$ is the $\cn$-bilinear extension of the standard inner product on $\rn^{2n}$. Let $a,b \in V$ be linearly independent and define $A$ to be the skew-symmetric matrix
\begin{align*}
    A = \sum_{r,s = 1}^{2n}a_{r} b_{s} Y_{rs}.
\end{align*}
Then the function $\psi:\SO{2n} \to \cn$ given by
\begin{align*}
    \psi: x \mapsto -\sum_{j\alpha} A_{j\alpha}\psi_{j\alpha}
\end{align*}
is a $\U{n}$-invariant eigenfunction on $\SO{2n}$ satisfying
\begin{align*}
    \tau(\psi) = -2(n-1)\cdot \psi \quad \textrm{and} \quad \kappa(\psi,\psi) = -\psi^2.
\end{align*}
\end{theorem}
\begin{proof}
The expression for $\tau(\psi)$ follows directly from Lemma \ref{tau kappa on SO/U} by the linearity of $\tau$. The components of the matrix $A$ have the form
\begin{align*}
    A_{j\alpha} = \tfrac{1}{\sqrt{2}}\cdot(a_{j}b_{\alpha}-b_{j}a_{\alpha}).
\end{align*}
For the conformality operator $\kappa$ we then use bilinearity to obtain
\begin{align*}
    \kappa(\psi,\psi) &= \sum_{j,\alpha,k,\beta} A_{j\alpha}A_{k\beta}\kappa(\psi_{j\alpha},\psi_{k\beta})\\
    &= -\sum_{j,\alpha,k,\beta} A_{j\alpha}A_{k\beta}(\psi_{j\beta}\psi_{k\alpha} + \psi_{jk}\psi_{\alpha\beta})\\
    &\sed - \sum_{j,\alpha,k,\beta} A_{j\alpha}A_{k\beta}(\delta_{j\beta}\delta_{k\alpha}-\delta_{jk}\delta_{\alpha\beta})\\
    &= -\tfrac{1}{2}\cdot\sum_{j,\alpha,k,\beta}(a_{j}b_{\alpha} - b_{j}a_{\alpha})(a_{k}b_{\beta} - b_{k}a_{\beta})(\psi_{j\beta}\psi_{k\alpha} + \psi_{jk}\psi_{\alpha\beta})\\
    &\sed -\sum_{jk}A_{jk}A_{jk} + \sum_{j\beta} A_{j\beta}A_{\beta j}\\
    &= -\tfrac{1}{2}\cdot\sum_{j,\alpha,k,\beta}(a_{j}b_{\alpha} - b_{j}a_{\alpha})(a_{k}b_{\beta} - b_{k}a_{\beta})(\psi_{j\beta}\psi_{k\alpha} + \psi_{jk}\psi_{\alpha\beta})\\
    &\sed +2\cdot \sum_{jk}A_{jk}A_{jk}\\
    &= -\tfrac{1}{2}\cdot\sum_{j,\alpha,k,\beta}(a_{j}b_{\alpha} - b_{j}a_{\alpha})(a_{k}b_{\beta} - b_{k}a_{\beta})(\psi_{j\beta}\psi_{k\alpha} + \psi_{jk}\psi_{\alpha\beta}),
\end{align*}
where the sum over $A_{jk}$ vanishes since it is equal to $2(\ip{a}{a}\ip{b}{b}-\ip{a}{b}^2)$, and $a$ and $b$ belong to the same isotropic subspace. By expanding the remaining term and factoring and simplifying, making use of the skew symmetry of $A$, we get
\begin{align*}
    \kappa(\psi,\psi) = -2\cdot\left(\sum_{j,\alpha}a_{j}b_{\alpha}\psi_{j\alpha}\right)^2 = -2\cdot \left(\tfrac{1}{2}\psi^2\right) = -\psi^2
\end{align*}
as desired.
\end{proof}
\begin{remark}
The $\U{n}$-invariant eigenfunctions in Theorem \ref{eigenfunctions on SO/U} induce eigenfunctions on the symmetric space $M = \SO{2n}/\U{n}$ with the same eigenvalues. These eigenfunctions coincide with those found in the paper \cite{Gud-Sif-Sob} by S. Gudmundsson, A. Siffert and M. Sobak.
\end{remark}
\section[The Symmetric Space \textbf{Sp}(n)/\textbf{U}(n)]{The Symmetric Space $\Sp{n}/\U{n}$}
In this section we describe the family of symmetric spaces $M = \Sp{n}/\U{n}$ and  derive explicit formulae for the corresponding Cartan maps. We then use the techniques from Chapters \ref{ch: harmonic morphisms} and \ref{ch: eigen} to find complex-valued eigenfunctions, which may then be used to construct a variety of $p$-harmonic functions on these spaces.
\begin{example}\label{ex Sp/U}
We now let $M = \Sp{n}/\U{n}$. $\U{n}$ sits inside $\Sp{n}$ as the group of block matrices
\begin{align*}
    \U{n} \cong 
    \left\{
    \begin{pmatrix}
    x & y\\
    -y & x
    \end{pmatrix} \in \Sp{n}
    \ \middle | \ x,y \in \GL{n,\rn}, \ x + iy \in \U{n} 
    \right\}.
\end{align*}
As in the case $\SU{n}/\SO{n}$ set $\sigma(q) = \bar{q}$. Then for 
$$
q = 
\begin{pmatrix}
z & w\\
-\bar{w} & \bar{z}
\end{pmatrix} \in \Sp{n},
$$
we get
\begin{align*}
    \sigma(q) = 
    \begin{pmatrix}
    \bar{z} & \bar{w}\\
    -{w} & {z}
    \end{pmatrix}
\end{align*}
so that $\sigma(q) = q$ if and only if $w = \bar{w}$ and $z = \bar{z}$, meaning $q \in \U{n}$. The map $\sigma$ is an involution and hence $M$ is indeed a symmetric space. Since $\Sp{n}$ is compact and semisimple, $M$ is of compact type. The corresponding Cartan embedding is
\begin{align*}
    \Phi: q \mapsto qq^t.
\end{align*}
\end{example}
\begin{lemma}[\cite{Gud-Sif-Sob}]\label{basis Sp/U}
Let $\sp{n} = \la{k}\oplus \la{p}$ be the Cartan decomposition of $\sp{n}$ corresponding to the symmetric triple $(\Sp{n},\U{n},\sigma)$ of Example \ref{ex Sp/U}. Then an orthonormal basis for $\la{p}$ is given by
\begin{align*}
\basis_{\la{p}} = &\bigg\{
    \tfrac{1}{\sqrt{2}}
    \begin{psmallmatrix}
    iX_{rs} & 0\\
    0 & -iX_{rs}
    \end{psmallmatrix},
    \tfrac{1}{\sqrt{2}}
    \begin{psmallmatrix}
    iD_{t} & 0\\
    0 & -iD_{t}
    \end{psmallmatrix}, \\
    &\quad \tfrac{1}{\sqrt{2}}
    \begin{psmallmatrix}
    0 & iX_{rs}\\
    iX_{rs} & 0
    \end{psmallmatrix},
    \tfrac{1}{\sqrt{2}}
    \begin{psmallmatrix}
    0 & iD_t\\
    iD_t & 0
    \end{psmallmatrix}\ \bigg| \ 1 \leq r<s\leq n, 1\leq t \leq n\bigg\}.    
\end{align*}
\end{lemma}
\begin{proof}
The complement of $\basis_{\la{p}}$ in the orthonormal basis $\basis$ for $\sp{n}$ from Lemma \ref{basis sp} is clearly a basis for $\la{k} \cong \u{n}$. Hence, since $\la{p}$ and $\la{k}$ are orthogonal, $\basis_{\la{p}}$ is an orthonormal basis for $\la{p}$.
\end{proof}
The remainder of this section will be devoted to obtaining eigenfunctions on the symmetric space $M = \Sp{n}/\U{n}$. We define $\phi_{j\alpha}: \Sp{n} \to \cn$ by
\begin{align*}
    \phi_{j\alpha}: q \mapsto (\Phi(q))_{j\alpha} = e_{j}\cdot qq^t\cdot e_{\alpha}^t.
\end{align*}
The functions $\phi_{j\alpha}$ are then clearly $\U{n}$-invariant as we have $\phi_{j\alpha} = q_{j\alpha}\circ \Phi$. Hence, we may use Theorem \ref{composition relations} to deduce the actions of the tension field and the conformality operator of $\Sp{n}$ on the functions $\phi_{j\alpha}$.
\begin{lemma}[\cite{Gud-Sif-Sob}]\label{tau kappa Sp/U}
The tension field and the conformality operators on the quaternionic unitary group $\Sp{n}$ satisfy
\begin{align*}
    \tau(\phi_{j\alpha}) = -2(n+1)\cdot \phi_{j\alpha}
\end{align*}
and
\begin{align*}
    \kappa(\phi_{j\alpha},\phi_{k\beta}) = -(\phi_{j\beta}\phi_{k\alpha} + \phi_{jk}\phi_{\alpha\beta}) +(J_n)_{j\beta}(J_n)_{\alpha k} + (J_n)_{jk}(J_n)_{\alpha\beta}.
\end{align*}
\end{lemma}
\begin{proof}
By Theorem \ref{composition relations} it is sufficient to compute $\tau_{N}(f_{j\alpha})(\Phi(q))$ and \\$\kappa_{N}(f_{j\alpha},f_{k\beta})(\Phi(q))$ for $1 \leq j,k,\alpha,\beta \leq 2n$. Here $\tau_{N}$ and $\kappa_{N}$ are the tension field and the conformality operator on $N = \Phi(\Sp{n})$, respectively. By Remark \ref{basis cartan embedding} and using the basis for $\la{p}$ from Lemma \ref{basis Sp/U} we have
\begin{align*}
    \tau_{N}(f_{j\alpha})(\Phi(q)) &= -\tfrac{1}{2}\cdot\sum_{r<s} e_{j}\cdot q
    \begin{psmallmatrix}
    X_{rs} & 0\\
    0 & -X_{rs}
    \end{psmallmatrix}^2
    q^t \cdot e_{\alpha}^t + e_j \cdot q
    \begin{psmallmatrix}
    0 & X_{rs}\\
    X_{rs} & 0
    \end{psmallmatrix}^2
    q^t\cdot e_{\alpha}^t\\
    &\sed -\tfrac{1}{2}\sum_{t=1}^n e_{j}\cdot q 
    \begin{psmallmatrix}
    D_t & 0\\
    0 & -D_{t}
    \end{psmallmatrix}
    q^t \cdot e_{\alpha}^t + e_{j}\cdot q
    \begin{psmallmatrix}
    0 & D_{t}\\
    D_t & 0
    \end{psmallmatrix}
    q^t \cdot e_{\alpha}^t.
\end{align*}
Then by a simple computation using Lemma \ref{square sum relation} we find
\begin{align*}
    \tau_{N}(f_{j\alpha})(\Phi(q)) = -\tfrac{n+1}{2}\cdot\phi_{j\alpha}. 
\end{align*}
For $\kappa_{N}$ we have
\begin{align}\label{kappa Sp/U}
\begin{split}
    &\kappa_{N}(f_{j\alpha},f_{k\beta})(\Phi(q))\\
    &= -\tfrac{1}{2}\sum_{r<s} e_{j}\cdot q
    \begin{psmallmatrix}
    X_{rs} & 0\\
    0 & -X_{rs}
    \end{psmallmatrix}
    q^t \cdot e_{\alpha}^t \cdot e_{k} \cdot q
    \begin{psmallmatrix}
    X_{rs} & 0\\
    0 & -X_{rs}
    \end{psmallmatrix}
    q^t \cdot e_{\beta}^t\\
    &\sed -\tfrac{1}{2}\sum_{r<s} e_{j}\cdot q
    \begin{psmallmatrix}
    0 & X_{rs}\\
    X_{rs} & 0
    \end{psmallmatrix}
    q^t \cdot e_{\alpha}^t \cdot e_{k} \cdot q
    \begin{psmallmatrix}
    0 & X_{rs}\\
    X_{rs} & 0
    \end{psmallmatrix}
    q^t \cdot e_{\beta}^t\\
    &\sed - \tfrac{1}{2}\sum_{t=1}^n e_{j}\cdot q
    \begin{psmallmatrix}
    D_t & 0\\
    0 & D_t
    \end{psmallmatrix}
    q^t \cdot e_{\alpha}^t \cdot e_{k}\cdot q
    \begin{psmallmatrix}
    D_t & 0\\
    0 & D_t
    \end{psmallmatrix}
    q^t \cdot e_{\beta}^t\\
    &\sed - \tfrac{1}{2}\sum_{t=1}^n e_{j}\cdot q
    \begin{psmallmatrix}
    0 & D_{t}\\
    -D_{t} & 0
    \end{psmallmatrix}
    q^t \cdot e_{\alpha}^t \cdot e_{k}\cdot q
    \begin{psmallmatrix}
    0 & D_{t}\\
    -D_{t} & 0
    \end{psmallmatrix}
    q^t \cdot e_{\beta}^t.
    \end{split}
\end{align}
With similar computations as for the other symmetric spaces, we obtain
\begin{align*}
    &-\tfrac{1}{2}\sum_{r<s} e_{j}\cdot q
    \begin{psmallmatrix}
    X_{rs} & 0\\
    0 & -X_{rs}
    \end{psmallmatrix}
    q^t \cdot e_{\alpha}^t \cdot e_{k} \cdot q
    \begin{psmallmatrix}
    X_{rs} & 0\\
    0 & -X_{rs}
    \end{psmallmatrix}
    q^t \cdot e_{\beta}^t\\
    &= -\tfrac{1}{4}\sum_{r\neq s} q_{jr}q_{kr}q_{\alpha s}q_{\beta s} + q_{jr}q_{\beta r}q_{ks}q_{\alpha s} - q_{jr}q_{k,r+n}q_{\alpha s}q_{\beta, s+n} - q_{jr}q_{\beta, r+n}q_{k,s+n}q_{\alpha s}\\
    &\sed -q_{j,r+n}q_{kr}q_{\alpha,s+n}q_{\beta s} - q_{j,r+n}q_{\beta r}q_{ks}q_{\alpha,s+n} + q_{j,r+n}q_{k,r+n}q_{\alpha,s+n}q_{\beta,s+n}\\
    &\sed+ q_{j,r+n}q_{\beta,r+n}q_{k,s+n}q_{\alpha,s+n}
\end{align*}
for the first sum. For the second sum we obtain
\begin{align*}
    &-\tfrac{1}{2}\sum_{r<s} e_{j}\cdot q
    \begin{psmallmatrix}
    0 & X_{rs}\\
    X_{rs} & 0
    \end{psmallmatrix}
    q^t \cdot e_{\alpha}^t \cdot e_{k} \cdot q
    \begin{psmallmatrix}
    0 & X_{rs}\\
    X_{rs} & 0
    \end{psmallmatrix}
    q^t \cdot e_{\beta}^t\\
    &= -\tfrac{1}{4}\sum_{r\neq s} q_{jr}q_{kr}q_{\alpha,s+n}q_{\beta,s+n} + q_{jr}q_{\beta,s+n}q_{ks}q_{\alpha,s+n} + q_{jr}q_{k,r+n}q_{\alpha,s+n}q_{\beta s}\\
    &\sed + q_{jr}q_{\beta r}q_{k,s+n}q_{\alpha,s+n} +q_{j,r+n}q_{kr}q_{\alpha s}q_{\beta,s+n} + q_{j,r+n}q_{\beta,r+n}q_{ks}q_{\alpha s}\\
    &\sed +q_{j,r+n}q_{k,r+n}q_{\alpha s}q_{\beta s} + q_{j,r+n}q_{\beta r}q_{k,s+n}q_{\alpha s}.
\end{align*}
For the third and fourth sums we get
\begin{align*}
    &- \tfrac{1}{2}\sum_{t=1}^n e_{j}\cdot q
    \begin{psmallmatrix}
    D_t & 0\\
    0 & D_t
    \end{psmallmatrix}
    q^t \cdot e_{\alpha}^t \cdot e_{k}\cdot q
    \begin{psmallmatrix}
    D_t & 0\\
    0 & D_t
    \end{psmallmatrix}
    q^t \cdot e_{\beta}^t\\
    &= -\tfrac{1}{2}\cdot \sum_{t=1}^n q_{jt}q_{\alpha t}q_{kt}q_{\beta t} -q_{jt}q_{\alpha t}q_{k,t+n}q_{\beta, t+n}\\
    &\sed - q_{j,t+n}q_{\alpha,t+n}q_{kt}q_{\beta t} + q_{j,t+n}q_{\alpha,t+n}q_{k,t+n}q_{\beta,t+n}
\end{align*}
and
\begin{align*}
    &- \tfrac{1}{2}\sum_{t=1}^n e_{j}\cdot q
    \begin{psmallmatrix}
    0 & D_{t}\\
    -D_{t} & 0
    \end{psmallmatrix}
    q^t \cdot e_{\alpha}^t \cdot e_{k}\cdot q
    \begin{psmallmatrix}
    0 & D_{t}\\
    -D_{t} & 0
    \end{psmallmatrix}
    q^t \cdot e_{\beta}^t\\
    &= -\tfrac{1}{2}\cdot\sum_{t=1}^{n} q_{jt}q_{\alpha,t+n}q_{kt}q_{\beta,t+n} + q_{jt}q_{\alpha,t+n}q_{k,t+n}q_{\beta t}\\
    &\sed + q_{j,t+n}q_{\alpha t}q_{kt}q_{\beta,t+n} + q_{j,t+n}q_{\alpha t}q_{k,t+n}q_{\beta t}.
\end{align*}
Inserting these into equation \eqref{kappa Sp/U}, simplifying, factorising and making use of the identity $qJ_nq^t =J_n$ for all $q \in \Sp{n}$ yields
\begin{align*}
    \kappa_{N}(f_{j\alpha},f_{k\beta})(\Phi(q)) = -\tfrac{1}{4}(\phi_{k\alpha}\phi_{j\beta} + \phi_{jk}\phi_{\alpha\beta}) + \tfrac{1}{4}\cdot ((J_n)_{\alpha k}(J_n)_{j\beta} + (J_n)_{jk}(J_n)_{\alpha\beta}.
\end{align*}
The desired formulae for $\tau$ and $\kappa$ are then easily deduced via Theorem \ref{composition relations}.
\end{proof}
\begin{theorem}[\cite{Gud-Sif-Sob}]\label{eigenfunctions on Sp/U}
Let the complex symmetric matrix $A$ be given by $A = a^t\cdot a$ for some non-zero element $a\in \cn^{2n}$. Define the function $\phi: \Sp{n} \to \cn$ by
\begin{align*}
    \phi(q) = \tr(A \cdot \Phi(z)) = \sum_{j,\alpha =1}^{2n} a_{j}a_{\alpha} \phi_{j\alpha}.
\end{align*}
Then $\phi$ is a $\U{n}$-invariant eigenfunction on $\Sp{n}$ satisfying
\begin{align*}
    \tau(\phi) =-2(n+1)\cdot \phi \quad \mathrm{and}\quad \kappa(\phi,\phi) = -2\cdot\phi^2.
\end{align*}
\end{theorem}
\begin{proof}
The computations are very similar to those in the proof of Theorem \ref{eigenfunctions SU/SO}. The formula involving the tension field $\tau$ follows immediately from Lemma \ref{tau kappa Sp/U} and the linearity of $\tau$. For the conformality operator $\kappa$ we get, by bilinearity,
\begin{align*}
    &\kappa(\phi,\phi)\\
    &= \sum_{j,\alpha,k,\beta} a_{j}a_{\alpha}a_{k}a_{\beta}\cdot \kappa(\phi_{j\alpha},\phi_{k\beta})\\
    &= \sum_{j,\alpha,k,\beta} a_{j}a_{\alpha}a_{k}a_{\beta}\cdot (-\phi_{j\beta}\phi_{k\alpha}-\phi_{jk}\phi_{\alpha\beta} + (J_n)_{j\beta}(J_n)_{\alpha k} + (J_n)_{jk}(J_n)_{\alpha\beta}))\\
    &= -2\cdot\phi^2 + 2\cdot \tr(AJ_n)^2\\
    &= -2\cdot \phi^2.
\end{align*}
In the last line we have used that the product of a symmetric and a skew-symmetric matrix is always traceless.
\end{proof}
\begin{remark}
The eigenfunctions from Theorem \ref{eigenfunctions on Sp/U} induce eigenfunctions on the symmetric space $M = \Sp{n}/\U{n}$ with the same eigenvalues. These coincide with the eigenfunctions found by  Gudmundsson, Siffert and Sobak in their paper \cite{Gud-Sif-Sob}.
\end{remark}
\section[The Symmetric Space \textbf{SU}(2n)/\textbf{Sp}(n)]{The Symmetric Space $\SU{2n}/\Sp{n}$}
In this section we describe the family of symmetric spaces $M = \SU{2n}/\Sp{n}$ and  derive explicit formulae for the corresponding Cartan maps. We then use the techniques from Chapters \ref{ch: harmonic morphisms} and \ref{ch: eigen} to find complex-valued eigenfunctions, which may then be used to construct a variety of $p$-harmonic functions on these spaces.
\begin{example}\label{SU/Sp}
As we have seen, $\Sp{n} \subset \SU{2n}$ and indeed $M = \SU{2n}/\Sp{n}$ is one of the classical compact symmetric spaces. In the proof of Proposition \ref{Sp(n), GL(H) Lie groups} we saw that with
$$
J = 
\begin{pmatrix}
0 & I_n\\
-I_n & 0
\end{pmatrix}
$$
an element $z \in \GL{2n,\cn}$ belongs to $\GL{n,\mathbb{H}}$ if and only if $Jz =\bar{z}J$, or equivalently $JzJ^t = \bar{z}$. Hence the map $\sigma: \SU{2n} \to \SU{2n}$ defined by
\begin{align*}
    \sigma: z \mapsto J\bar{z}J^t
\end{align*}
fixes precisely the subgroup $\Sp{n}$. Obviously $\sigma$ is an automorphism of $\SU{2n}$. Since $\bar{J} = J$ and $J^2=-I_{2n}$ we have
\begin{align*}
    \sigma^2(z) = J^2z(J^t)^2 = -I_{2n}z(-I_{2n}) = z
\end{align*}
for all $z \in \SU{2n}$. Therefore, $\sigma$ is also an involution. Hence $M = \SU{2n}/\Sp{n}$ is a symmetric space. Since $\SU{2n}$ is compact and semisimple, $M$ is of compact type. The Cartan embedding corresponding to $\sigma$ is then given by
\begin{align*}
    \Phi: z \mapsto zJz^tJ^t.
\end{align*}
\end{example}

We let $\su{2n} = \la{k}\oplus \la{p}$, $\la{k} \cong \sp{n}$ be the Cartan decomposition of $\su{2n}$ corresponding to the symmetric triple $(\SU{2n},\Sp{n},\sigma)$ of Example \ref{SU/Sp}. The conditions 
$$
d\sigma(Z) = J\bar{Z}J^t = -Z, \quad\bar{Z}^t = -Z \quad \mathrm{and} \quad \tr(Z) = 0
$$
must be satisfied for all $Z \in \la{p}$. From these conditions we deduce that  that $\la{p}$ is the subspace of $\su{2n}$ defined by
\begin{align}\label{eq: p in SU/Sp}
    \la{p} = \left\{\begin{pmatrix}
    Z & W\\
    \bar{W} & -\bar{Z}\end{pmatrix} \ \middle| \ \bar{Z}^t + Z = W^t + W = 0\quad \mathrm{and}\quad \tr(Z) = 0
    \right\}.
\end{align}
For the upcoming calculations it will be useful to introduce the space $\la{p}\subset \tilde{\la{p}}\subset \u{n}$ satisfying
\begin{align*}
    \tilde{\la{p}} = \left\{ \begin{pmatrix}
    Z & W\\
    \bar{W} & -\bar{Z}
    \end{pmatrix} \ \middle| \ \bar{Z}^t + Z = W^t + W = 0 \right\}.
\end{align*}
It is then easy to verify that the set
\begin{align}\label{eq: almost basis SU/Sp}
    \begin{split}
       \basis_{\tilde{\la{p}}} &= \Big\{
       \tfrac{1}{\sqrt{2}}\begin{psmallmatrix}
            Y_{rs} & 0\\
            0 & -Y_{rs}
        \end{psmallmatrix},
        \tfrac{1}{\sqrt{2}}\begin{psmallmatrix}
            iX_{rs} & 0\\
            0 & iX_{rs}
        \end{psmallmatrix},
        \tfrac{1}{\sqrt{2}}\begin{psmallmatrix}
            iD_{t} & 0\\
            0 & iD_{t}
        \end{psmallmatrix},\\
        &\sed \tfrac{1}{\sqrt{2}}\begin{psmallmatrix}
            0 & Y_{rs}\\
            Y_{rs} & 0
        \end{psmallmatrix},
        \tfrac{1}{\sqrt{2}}\begin{psmallmatrix}
            0 & iY_{rs}\\
            -iY_{rs} & 0
        \end{psmallmatrix}
        \ \Big| \ 1 \leq r < s \leq n, \quad 1 \leq t \leq n \Big\} 
    \end{split}
\end{align}
is an orthonormal basis for $\tilde{\la{p}}$. Further, the orthogonal complement of $\la{p}$ in $\tilde{\la{p}}$ is $1$-dimensional and is generated by the unit vector
\begin{align*}
    X = \tfrac{i}{\sqrt{2n}}\cdot I_{2n}.
\end{align*}

The remainder of this section will be devoted to finding eigenfunctions on the symmetric space $M = \SU{2n}/\Sp{n}$. We define the functions $\psi_{j\alpha}: \SU{2n} \to \cn$ by
\begin{align*}
    \psi_{j\alpha}: z \mapsto (zJz^t)_{j\alpha}
\end{align*}
for $1 \leq j,\alpha \leq 2n$. One easily checks that these functions satisfy 
\begin{align*}
    \psi_{j\alpha}(z) = \sum_{r=1}^{n} z_{jr}z_{\alpha,r+n} - z_{j,r+n}z_{\alpha r}.
\end{align*}
We also see that we may write  $\psi_{j\alpha}$ as a composition $\psi_{j\alpha} = f_{j\alpha} \circ \Phi$, where $f_{j\alpha}$ are the functions
\begin{align*}
    f_{j\alpha}:z \mapsto (zJ)_{j\alpha}
\end{align*}
and $\Phi$ is the Cartan map. This means that the functions $\psi_{j\alpha}$ are $\Sp{n}$-invariant and we may utilise Theorem \ref{composition relations} to prove the following result.
\begin{lemma}\label{tau kappa SU/Sp}
The tension field and the conformality operator on the special unitary group $\SU{2n}$ satisfy the relations
\begin{align*}
    \tau(\psi_{j\alpha}) = -\tfrac{2(2n^2-n-1)}{n}\cdot \psi_{j\alpha}
\end{align*}
and
\begin{align*}
    \kappa(\psi_{j\alpha},\psi_{k\beta}) = -2\cdot \psi_{j\beta}\psi_{k\alpha} - 2\cdot \psi_{jk}\psi_{\alpha\beta} + \tfrac{2}{n}\cdot \psi_{j\alpha}\psi_{k\beta}.
\end{align*}
\end{lemma}
\begin{proof}
By Theorem \ref{composition relations} it is sufficient to compute $\tau_{N}(f_{j\alpha})(\Phi(z))$ and \\$\kappa_{N}(f_{j\alpha},f_{k\beta})(\Phi(z))$ for $1\leq j,k,\alpha,\beta \leq 2n$. Here $\tau_{N}$ are the tension field and the conformality operator on $N = \Phi(\SU{2n})$, respectively. From Remark \ref{basis cartan embedding} we see that the tension field $\tau_N$ and $\kappa_N$ satisfies
\begin{align*}
    \tau_{N}(f_{j\alpha}) = \sum_{Z \in \basis_{\la{p}}} e_{j}\cdot zZ^2 Jz^t\cdot e_{\alpha}^t.  
\end{align*}
Since $\tilde{\la{p}} = \la{p} \oplus \mathrm{span}(\tfrac{i}{\sqrt{2n}}\cdot I_{2n})$, we have
\begin{align*}
    \tau_{N}(f_{j\alpha})(\Phi(z))
    &= \sum_{Z \in \basis_{\tilde{\la{p}}}}\big( e_{j}\cdot zZ^2 Jz^t\cdot e_{\alpha}^t\big) + \tfrac{1}{2n}\cdot e_{j}\cdot zJz^t \cdot e_{\alpha}^t\\
    &= \tfrac{1}{2} \cdot\sum_{r<s} e_{j}\cdot
    \begin{psmallmatrix}
    Y_{rs} & 0\\
    0 & -Y_{rs}
    \end{psmallmatrix}^2 J z^t \cdot e_{\alpha}^t\\
    &\sed - \tfrac{1}{2}\cdot\sum_{r<s}e_{j}\cdot z
    \begin{psmallmatrix}
    X_{rs} & 0\\
    0 & X_{rs}
    \end{psmallmatrix}^2
    J z^t \cdot e_{\alpha}^t\\
    &\sed -\tfrac{1}{2}\cdot \sum_{t=1}^n e_{j}\cdot z 
    \begin{psmallmatrix}
    D_t & 0\\
    0 & D_t
    \end{psmallmatrix}^2
    J z^t \cdot e_{\alpha}^t\\
    &\sed - \tfrac{1}{2}\cdot \sum_{r<s} e_{j}\cdot z
    \begin{psmallmatrix}
    0 & Y_{rs}\\
    -Y_{rs} & 0
    \end{psmallmatrix}^2
    j z^t \cdot e_{\alpha}^t\\
    &\sed +\tfrac{1}{2}\cdot \sum_{r<s} e_{j} \cdot z 
    \begin{psmallmatrix}
    0 & Y_{rs}\\
    Y_{rs} & 0
    \end{psmallmatrix}^2
    J z^t \cdot e_{\alpha}^t\\
    &\sed+ \tfrac{1}{2n}\cdot e_{j}\cdot zJz^t \cdot e_{\alpha}^t\\
    &= -(n-1)\cdot e_{j}\cdot zJz^t\cdot e_{\alpha}^t - \tfrac{1}{2}\cdot e_{j}\cdot  zJz^t \cdot e_{\alpha}^t + \tfrac{1}{2n}\cdot e_{j}\cdot zJz^t \cdot e_{\alpha}^t\\
    &= \tfrac{1}{4}\cdot\left(-\tfrac{2(2n^2-n-1)}{n}\right)\cdot \psi_{j\alpha}.
\end{align*}
Here, in the second to last line, we have used the identities of Lemma \ref{square sum relation}. Similarily, for the conformality operator $\kappa_{N}$ we get
\begin{align*}
    &\kappa_N(f_{j\alpha},f_{k\beta})(\Phi(z))\\
    &= \tfrac{1}{2}\cdot \sum_{r<s}e_{j}\cdot z 
    \begin{psmallmatrix}
    Y_{rs} & 0\\
    0 & -Y_{rs}
    \end{psmallmatrix}
    J z^t \cdot e_{\alpha}^t \cdot e_{k}\cdot z
    \begin{psmallmatrix}
    Y_{rs} & 0\\
    0 & -Y_{rs}
    \end{psmallmatrix}
    J z^t e_{\beta}^t\\
    &\sed - \tfrac{1}{2}\cdot \sum_{r<s}e_{j}\cdot z 
    \begin{psmallmatrix}
    X_{rs} & 0\\
    0 & X_{rs}
    \end{psmallmatrix}
    J z^t \cdot e_{\alpha}^t \cdot e_{k}\cdot z
    \begin{psmallmatrix}
    X_{rs} & 0\\
    0 & X_{rs}
    \end{psmallmatrix}
    J z^t e_{\beta}^t\\
    &\sed - \tfrac{1}{2}\cdot \sum_{r<s}e_{j}\cdot z 
    \begin{psmallmatrix}
    D_t & 0\\
    0 & D_t
    \end{psmallmatrix}
    J z^t \cdot e_{\alpha}^t \cdot e_{k}\cdot z
    \begin{psmallmatrix}
    D_t & 0\\
    0 & D_t
    \end{psmallmatrix}
    J z^t e_{\beta}^t\\
    &\sed - \tfrac{1}{2}\cdot \sum_{r<s}e_{j}\cdot z 
    \begin{psmallmatrix}
    0 & Y_{rs}\\
    -Y_{rs} & 0
    \end{psmallmatrix}
    J z^t \cdot e_{\alpha}^t \cdot e_{k}\cdot z
    \begin{psmallmatrix}
    0 & Y_{rs}\\
    -Y_{rs} & 0
    \end{psmallmatrix}
    J z^t e_{\beta}^t\\
    &\sed + \tfrac{1}{2}\cdot \sum_{r<s}e_{j}\cdot z 
    \begin{psmallmatrix}
    0 & Y_{rs}\\
    Y_{rs} & 0
    \end{psmallmatrix}
    J z^t \cdot e_{\alpha}^t \cdot e_{k}\cdot z
    \begin{psmallmatrix}
    0 & Y_{rs}\\
    Y_{rs} & 0
    \end{psmallmatrix}
    J z^t e_{\beta}^t\\
    &\sed + \tfrac{1}{2n}\cdot e_{j}\cdot zJz^t\cdot e_{\alpha}^t\cdot e_{k}\cdot zJz^t \cdot e_{\beta}^t.
\end{align*}
Using similar computations as in the proof of Lemma \ref{tau kappa on SO/U} we then find
\begin{align*}
    \kappa_N(f_{j\alpha},f_{k\beta})(\Phi(z)) = -\tfrac{1}{2}\cdot \psi_{j\beta}\psi_{k\alpha} -\tfrac{1}{2}\cdot \psi_{jk}\psi_{\alpha\beta} + \tfrac{1}{2n}\cdot \psi_{j\alpha}\psi_{k\beta}.
\end{align*}
From Theorem \ref{composition relations} we then obtain the desired formulae for the tension field $\tau$ and the conformality operator $\kappa$.
\end{proof}
\begin{theorem}[\cite{Gud-Sif-Sob}]\label{eigenfunctions on SU/Sp}
For non-zero linearly independent elements $a,b\in \cn^{2n}$, let $A \in \cn^{2n \times 2n}$ be the skew-symmetric matrix
\begin{align*}
    A = \sum_{r,s = 1}^{2n} a_{r}b_{s}Y_{rs}.
\end{align*}
Define the function $\psi: \SU{2n} \to \cn$ by
\begin{align*}
    \psi(z) = -\sum_{j,\alpha} A_{j\alpha} \psi_{j\alpha}.
\end{align*}
Then $\psi$ is an $\Sp{n}$-invariant eigenfunction on $\SU{2n}$ satisfying
\begin{align*}
    \tau(\psi) = -\tfrac{2(2n^2-n-1)}{2}\cdot \psi \quad \mathrm{and} \quad \kappa(\psi,\psi) = -\tfrac{2(n-1)}{n}\cdot \psi^2.
\end{align*}
\end{theorem}
\begin{proof}
To prove the Theorem we employ Lemma \ref{tau kappa SU/Sp} together with the linearity and bilinearity of the tension field and the conformality operator, respectively.  The computations are almost identical to those in the proof of Theorem \ref{eigenfunctions on SO/U} and thus won't be presented here.
\end{proof}
\begin{remark}
The eigenfunctions in Theorem \ref{eigenfunctions on SU/Sp} induce eigenfunctions on the symmetric space $M = \SU{2n}/\Sp{n}$ with the same eigenvalues. These eigenfunctions are the same as those found by Gudmundsson, Siffert and Sobak in thei paper \cite{Gud-Sif-Sob}.
\end{remark}
\appendix
\chapter{Calculating the Killing Forms}
In this appendix we will calculate explicitly the Killing forms of the Lie algebras $\gl{n}{\rn}, \sl{n}{\rn}, \so{n}, \u{n}, \su{n}$ and $\sp{n}$. To this end, the following Lemma will prove useful.
\begin{lemma}[\cite{Helgason}]\label{restriction of Killing form}
Let $\la{g}$ be a Lie algebra and $\la{a}$ an ideal in $\la{g}$. Then the Killing form $B_{\la{a}}$ of $\la{a}$ is the restriction $\restr{B}{\la{a}}$ of the Killing form $B$ of $\la{g}$ to $\la{a}$.
\end{lemma}
\begin{proof}
Pick a basis $\mathcal{B} =\{e_1, \dots, e_k, f_1 \dots, f_l\}$ of $\la{g}$ such that $\mathcal{B}_{\la{a}} = \{e_1, \dots, e_k\}$ is a basis for $\la{a}$. Let $X, Y \in \la{a}$. Then, since $\la{a}$ is an ideal, the matrix for $\ad_X$ with respect to the basis $\mathcal{B}$ has the form
\begin{align*}
    [\ad_X]
    = \begin{pmatrix}
    [\ad_X]_{\mathcal{B}_{\la{a}}} & *\\
    0 & 0
    \end{pmatrix}
\end{align*}
and similarly for the matrix of $\ad_Y$. Then the matrix of $\ad_X \circ \ad_Y$ will look like
\begin{align*}
    [\ad_X\ad_Y] =
    \begin{pmatrix}
    [\ad_X\ad_Y]_{\mathcal{B}_{\la{a}}} & *\\
    0 & 0
    \end{pmatrix}
\end{align*}
so that
\begin{align*}
    B(X,Y) &= \tr_{\la{g}}(\ad_X \circ \ad_Y)\\
    &= \tr_{\la{a}}(\ad_X \circ \ad_Y)\\
    &= B_{\la{a}}(X,Y).
\end{align*}
\end{proof}
As in the proof of the preceeding lemma we will use the notation $[X]_{\mathcal{B}}$ for the matrix of a linear map $X$ given some basis $\mathcal{B}$. If the choice of basis is clear from context we will omit the subscript. 
\par
First we shall calculate the Killing forms of  $\gl{n}{\rn}$ and $\sl{n}{\rn}$.
\begin{proposition}
The Killing forms of $\gl{n}{\rn}$ and $\sl{n}{\rn}$ are
\begin{align*}
    B_{\gl{n}{\rn}}(X,Y) = 2n\cdot\tr(X\cdot Y) - 2 \cdot \tr(X)\cdot \tr(Y)
\end{align*}
and
\begin{align*}
    B_{\sl{n}{\rn}}(X,Y) = 2n \cdot \tr(X \cdot Y).
\end{align*}
\end{proposition}
\begin{proof}
The set  
\begin{align*}
    \mathcal{B} = \{ E_{j\alpha} \ | \ 1 \leq j\alpha \leq n \}
\end{align*}
forms an orthonormal basis of  $\gl{n}{\rn}$ with respect to the inner product
\begin{align*}
    \ip{X}{Y} = \tr(X^t \cdot  Y).
\end{align*}
We may then calculate $B = B_{\gl{n}{\rn}}$ as
\begin{align}\label{appendix eq 1}
\begin{split}
    B(X,Y) &= \tr(\ad_X \circ \ad_Y)\\
    &= \sum_{j,\alpha = 1}^{n} \tr(E_{\alpha j} \cdot \ad_X \ad_Y (E_{j\alpha})).    
\end{split}
\end{align}
Writing
\begin{align*}
    (X)_{j\alpha} = x_{j\alpha}
\end{align*}
A simple computation shows $E_{k\beta}XE_{j\alpha} = x_{\beta j} E_{k\alpha}$. Letting $X, Y \in \gl{n}{\rn}$ and $1 \leq j,\alpha \leq n$ we then have
\begin{align*}
    E_{\alpha j}(\ad_X(\ad_Y(E_{j\alpha}))) &= E_{\alpha j}XYE_{j\alpha} - E_{\alpha j}XE_{j\alpha}Y - E_{\alpha j}YE_{j\alpha}X - E_{\alpha j}E_{j\alpha} YX\\
    &= (XY)_{jj}E_{\alpha\alpha} - x_{jj}E_{\alpha\alpha}Y - y_{jj}E_{\alpha\alpha}X + E_{\alpha\alpha}YX
\end{align*}
so that
\begin{align*}
    \tr(E_{\alpha j}\ad_X\ad_YE_{j\alpha}) &= \tr((XY)_{jj}E_{\alpha\alpha} - x_{jj}E_{\alpha\alpha}Y - y_{jj}E_{\alpha\alpha}X + E_{\alpha\alpha}YX)\\
    &= (XY)_{jj} - x_{jj}y_{\alpha\alpha} - y_{jj}x_{\alpha\alpha} + (YX)_{\alpha\alpha}.
\end{align*}
Inserting this into equation \eqref{appendix eq 1} we get
\begin{align*}
    B(X, Y) &= \sum_{j,\alpha = 1}^{n} (XY)_{jj} - x_{jj}y_{\alpha\alpha} - y_{jj}x_{\alpha\alpha} + (YX)_{\alpha\alpha}\\
    &= 2n \cdot \tr(X\cdot Y) - 2\cdot \sum_{j,\alpha = 1}^{n} x_{jj}y_{\alpha\alpha}\\
    &= 2n \cdot \tr(X\cdot Y) - 2\cdot \tr(X)\cdot\tr(Y).
\end{align*}
To get the formula for $B_{\sl{n}{\rn}}$ we note that $\sl{n}{\rn}$ is an ideal in $\gl{n}{\rn}$, since for each $X, Y \in \gl{n}{\rn}$,
\begin{align*}
    \tr([X,Y]) = \tr(X\cdot Y)- \tr(Y\cdot X) = 0.
\end{align*}
We can therefore apply Lemma \ref{restriction of Killing form} to get
\begin{align*}
    B_{\sl{n}{\rn}}(X,Y) = B_{\gl{n}{\rn}}(X, Y) = 2n \, \tr(X\cdot Y).
\end{align*}
\end{proof}
Next we compute the Killing form of $\so{n}$. 
\begin{proposition}\label{killing form so}
The Killing form of $\so{n}$ is given by
\begin{align*}
    B(X,Y) = (n-2) \, \tr(X \cdot Y)
\end{align*}
for $X,Y \in \so{n}$.
\end{proposition}
\begin{proof}
As we have seen before, the set
\begin{align*}
\basis_{\so{n}} = \{Y_{rs} \ | \ 1\leq r < s \leq n\}
\end{align*}
is an orthonormal basis for $\so{n}$. We may thus calculate the killing form as
\begin{align*}
    B(X,Y) &= \tr(\ad_{X}\circ\ad_{Y})\\
    &= -\sum_{r < s} \tr(Y_{rs} \cdot \ad_{X}\ad_Y (Y_{rs})),
\end{align*}
for $X,Y \in \so{n}$.
Fixing $1\leq r < s \leq n$ and denoting $(X)_{j\alpha} = x_{j\alpha}, (Y)_{j\alpha} = y_{j\alpha}$ we have
\begin{align*}
    &\tr(Y_{rs} \cdot \ad_{X}\ad_{Y}(Y_{rs}))\\
    &= -\bigg(\sum_{t= 1}^n e_t \cdot (Y_{rs}\cdot XY\cdot Y_{rs} -Y_{rs}\cdot X\cdot Y_{rs}\cdot Y\\
    &\sed - Y_{rs}\cdot Y\cdot Y_{rs}\cdot X + Y_{rs}^2\cdot YX) \cdot e_t^t\bigg)\\
    &= -\tfrac{1}{2}\cdot \bigg(\sum_{t=1}^n (XY)_{rs} \delta_{tr}\delta_{ts} -(XY)_{ss}\delta_{tr} - (XY)_{rr}\delta_{ts} + (XY)_{rs}\delta_{ts}\delta_{tr}
     -x_{rs}y_{st}\delta_{rt}\\
     &\sed +x_{ss}y_{rt}\delta_{rt} + x_{rr}y_{st}\delta_{st} - x_{rs}y_{rt}\delta_{st} -x_{st}y_{sr}\delta_{rt} + x_{rt}y_{ss}\delta_{rt}
     +x_{st}y_{rr}\delta_{st}\\
     &\sed-x_{rt}y_{rs}\delta_{st} + (YX)_{st}\delta_{sr}\delta_{tr} - (YX)_{rt}\delta_{rt} - (YX)_{st}\delta_{st} + (YX)_{rt}\delta_{tr}\delta_{ts}\bigg)\\
     &= -\tfrac{1}{2}\cdot \big((XY)_{sr}\delta_{sr} -(XY)_{ss} - (XY)_{rr} + \delta_{sr}(XY)_{rs}\\
     &\sed - 2\cdot x_{sr}y_{sr} + 2\cdot x_{ss}y_{rr} + 2\cdot x_{rr}y_{ss} - 2\cdot x_{rs}y_{rs}\\
     &\sed +(YX)_{sr}\delta_{sr} - (YX)_{rr} - (YX)_{ss} + (YX)_{rs}\delta_{rs}\big)\\
     &= (XY)_{ss} + (XY)_{rr} - 2\cdot x_{sr}y_{rs}.
\end{align*}
In the last step we have used that $Y$ and $X$ are skew-symmetric, and that $\delta_{rs} = 0$ since $r \neq s$. From this we then have
\begin{align*}
    -\sum_{r<s} \tr(Y_{rs}\cdot \ad_{X}\ad_{Y}(Y_{rs})) &= \sum_{r < s} (XY)_{ss} + (XY)_{rr} - 2 \cdot x_{sr}y_{rs}\\
    &= (n-1)\cdot \tr(XY) - \sum_{r\neq s} x_{sr}y_{rs}\\
    &= (n-1)\cdot \tr(XY) - \tr(XY)\\
    &= (n-2)\cdot\tr(XY)
\end{align*}
as desired.
\end{proof}
\begin{proposition}\label{killing form u & su}
The killing forms of $\u{n}$ and $\su{n}$ are given by
$$
    B_{\u{n}}(Z,W) = 2n\cdot \tr(Z\cdot W) - 2\cdot \tr(Z)\cdot \tr(W)
$$
and
$$
    B_{\su{n}}(Z,W) = 2n\cdot \tr(Z\cdot W),
$$
respectively.
\end{proposition}
\begin{proof}
Put $B = B_{\u{n}}$. The set 
\begin{align*}
    \basis_{\u{n}} = \{ Y_{rs}, iX_{rs}, iD_{t} \ | \ 1 \leq r < s \leq n, 1\leq t \leq n\}
\end{align*}
is an orthonormal basis for $\u{n}$. For $Z,W\in \u{n}$ we may therefore calculate $B(Z,W)$ as
\begin{align*}
    B(Z,W) &= \tr(\ad_Z \circ \ad_W)\\
    &= -\sum_{r<s} \tr(Y_{rs} \cdot \ad_Z\ad_W(Y_{rs}))\\
    &\sed + \sum_{r < s} \tr(X_{rs} \cdot \ad_Z\ad_W(X_{rs}))\\
    &\sed + \sum_{t=1}^n \tr(D_t \cdot \ad_Z \ad_W(D_t))\\
    & = \sum_{r<s} \tr(X_{rs} \cdot \ad_Z\ad_W(X_{rs}) - Y_{rs} \cdot \ad_Z\ad_W(Y_{rs}))\\
    &\sed + \sum_{t=1}^n \tr(D_t \cdot \ad_Z \ad_W(D_t)).
\end{align*}
We shall denote $(Z)_{j\alpha}$ by $z_{j\alpha}$ and $(W)_{j\alpha}$ by $w_{j\alpha}$. By similar calculations as in the proof of Proposition \ref{killing form so} we find
\begin{align*}
    &\sum_{r<s} \tr(X_{rs} \cdot \ad_Z\ad_W(X_{rs}) - Y_{rs} \cdot \ad_Z\ad_W(Y_{rs}))\\
    &= 2(n-1)\cdot \tr(Z\cdot W) - 2\cdot\sum_{r\neq s} z_{ss}w_{rr}
\end{align*}
and
\begin{align*}
    \sum_{t=1}^n \tr(D_t \cdot \ad_Z \ad_W(D_t)) = 2\cdot \tr(Z\cdot W) - 2\cdot \sum_{t=1}^n z_{tt}w_{tt}.
\end{align*}
Combining these then yields
\begin{align*}
    B(Z,W) = 2n\cdot \tr(Z\cdot W) - 2\cdot \tr(Z)\cdot \tr(W)
\end{align*}
as desired. For the killing form $B_{\su{n}}$ of $\su{n}$ we use Lemma \ref{restriction of Killing form}, since $\su{n}$ is an ideal in $\u{n}$. We thus obtain
\begin{align*}
    B_{\su{n}}(Z,W) = B(Z,W) = 2n\cdot \tr(Z\cdot W)
\end{align*}
for $Z, W \in \su{n}$. The last equality follows since $\tr(Z) = \tr(W) = 0$.
\end{proof}
\begin{proposition}
    The Killing form of $\sp{n}$ is given by
    \begin{align*}
        B_{\sp{n}}(W,Z) = 2n\cdot\tr(Z\cdot W).
    \end{align*}
\end{proposition}
This can be shown by similar computations as in Propositions \ref{killing form so} and \ref{killing form u & su} although they are more lengthy.

\addcontentsline{toc}{chapter}{Bibliography}{}
\backcover


\begin{thebibliography}{99}\label{biblio}
\normalsize

\bibitem{Baird and Eells}
P. Baird, J. Eells, {\it A conservation law for harmonic maps}, Geometry Symposium Utrecht 1980, Lecture notes in Mathematics {\bf 894}, 1-25, Springer (1981)

\bibitem{Baird and Wood}
P. Baird, J.C. Wood, {\it Harmonic Morphisms Between Riemannian Manifolds}, The London Mathematical Society Monographs {\bf 29}, Oxford University Press (2003).

\bibitem{Cartan 1}
E. Cartan, {\it Sur une classe remarquable d'espaces de Riemann}, Bull. Soc. Math. France {\bf 54}, (1926), 214-264.

\bibitem{Cartan 2}
E. Cartan, {\it Sur une classe remarquable d'espaces de Riemann II}, Bull. Soc. Math. France {\bf 55}, (1927), 114-134. 

\bibitem{Cheeger-Ebin}
J. Cheeger, D. G. Ebin, {\it Comparison Theorems in Riemannian Geometry}, North-Holland Mathematical Library {\bf 9} (1975),viii+174.

\bibitem{Fuglede}
B. Fuglede, {\it Harmonic morphisms between Riemannian manifolds}, Ann. Inst. Fourier {\bf 28} (1978), 107-144

\bibitem{biblio harm-morph}
S. Gudmundsson, {\it The Bibliography of Harmonic Morphisms}, 

{\tt https://www.matematik.lu.se/matematiklu/personal/sigma/harmonic/bibliography.html}

\bibitem{biblio p-harm}
S. Gudmundsson, {\it The Bibliography of $p$-Harmonic Functions}

{\tt https://www.matematik.lu.se/matematiklu/personal/sigma/harmonic/p-bibliography.html}

\bibitem{Sigmundur's notes}
S. Gudmundsson, {\it An Introduction to Riemannian Geometry}, Lecture Notes in Mathematics, Lund University (2023).

{\tt https://www.matematik.lu.se/matematiklu/personal/sigma/Riemann.pdf}

\bibitem{sg sym sps rank 1}
S. Gudmundsson {\it On the existence of harmonic morphisms from symmetric spaces of rank one}, Manuscripta Math. {\bf 93} (1997), 421-433.

\bibitem{Sigmundur's doctoral thesis}
S.Gudmundsson, {\it The Geometry of Harmonic Morphisms}, University of Leeds (1992) 

{\tt https://www.matematik.lu.se/matematiklu/personal/sigma/papers/Doctoral-thesis.pdf}

\bibitem{Gud-Mon-Rat}
S. Gudmundsson, S. Montaldo, A. Ratto. {\it Biharmonic functions on the classical compact simple Lie
groups}. J. Geom. Anal. {\bf 28} 1525-1547 (2018).

\bibitem{Gud-Sak}
S. Gudmundsson, A. Sakovich, {\it Harmonic morphisms from the classical compact semisimple Lie groups}, Ann. Global. Anal. Geom. {\bf 33} (2008) 343-356.

\bibitem{Gud-Sif-Sob}
S. Gudmundsson, A. Siffert, M. Sobak, {\it Explicit proper p-harmonic functions on the Riemannian symmetric spaces $\SU{n}/\SO{n}, \Sp{n}/\U{n},\SO{2n}/\U{n}, \SU{2n}/\Sp{n}$.} J. Geom. Anal. {\bf 31}, (2021), 11386-11409.

\bibitem{Gud-Sob}
S. Gudmundsson, M. Sobak {\it Proper $r$-harmonic functions from Riemannian manifolds}, Ann. Global Anal. Geom. {\bf 57} (2020) 217-223.

\bibitem{Gud-Sve}
S. Gudmundsson, M. Svensson, {\it Harmonic morphisms from the Grassmannians and their non-compact
duals}. Ann. Glob. Anal. Geom. {\bf 30}. 313-333 (2006).

\bibitem{p-harm on semi-riem}
S. Gudmundsson, E. Ghandour {\it Proper p-harmonic functions and harmonic morphisms on the classical non-compact semi-Riemannian Lie groups}, 	arXiv:2102.07547 [math.DG], preprint (2021).

\bibitem{Gud-Gha 1}
S. Gudmundsson, E. Ghandour {\it Explicit $p$-harmonic functions and harmonic morphisms on the real Grassmannians}, Adv. Geom. (to appear).

\bibitem{Gud-Gha 2}
S. Gudmundsson, E. Ghandour {\it Explicit harmonic morphisms and $p$-harmonic functions from the complex and quaternionic Grassmannians}, preprint (2023) arXiv:2007.14819 

\bibitem{Helgason}
S. Helgason, {\it Differential Geometry, Lie Groups, and Symmetric Spaces}, Graduate Studies in Mathematics {\bf 34}, AMS (2001).

\bibitem{Ishihara}
T. Ishihara, {\it A mapping of Riemannian manifolds which preserves harmonic functions}, J. Math. Soc. Japan {\bf 7} (1979), 345-370

\bibitem{Knapp}
A. W. Knapp, {\it Lie Groups Beyond an Introduction},  Progress in Mathematics {\bf 140}, Birkh\"auser (2002).

\bibitem{Jacobi}
C. G. J. Jacobi, {\it\"{U}ber eine particul\"{a}re {L}\"{o}sung der partiellen
              {D}ifferentialgleichung {$$\frac{\partial^2V}{\partial
              x^2}+\frac{\partial^2V}{\partial
              y^2}+\frac{\partial^2V}{\partial z^2}=0,$$}} J. Reine Angew. Math. {\bf 36} (1848), 113-134.

\bibitem{Kobayashi-Nomizu}
S. Kobayashi, K. Nomizu {\it Foundations of Differential Geometry Vol 1}, Interscience Publishers (1963).

\bibitem{Lee}
J. M. Lee, {\it Manifolds and Differential Geometry}, Graduate Studies in Mathematics {\bf 107}, AMS (2009).

\bibitem{Loos}
O. Loos, {\it Symmetric Spaces. I: General Theory}, W. A. Benjamin Inc. (1969).

\bibitem{Myers-Steenrod}
S. B. Myers, N. B. Steenrod {\it The group of isometries of a Riemannian manifold}, Ann. of Math. (2) {\bf 40}, (1939).

\bibitem{The fundamental equations of a submersion}
B. O'Neill, {\it The fundamental equations of a submersion}, Mich. Math. J. {\bf 13} (1966), 459-469.

\bibitem{Petersen}
P. Petersen, {\it Riemannian Geometry} ($3^{\text{rd}}$ edition), Graduate Texts in Mathematics {\bf 171}, Springer (2016).



\bibitem{Ziller's notes}
W. Ziller, {\it Lie Groups, representation Theory and Symmetric Spaces}, University of Pennsylvania (2021).




\end{thebibliography}
\end{document}